\documentclass[a4paper,12pt,customfont,custombib]{Classes/PhDThesisPSnPDF} % times, print, index, draft, twoside, print

\input{preamble}

\title{Extremal solutions to some art~gallery and terminal-pairability problems}
\author{Tamás Róbert Mezei}

\dept{Department of Mathematics and its Applications}

\university{Central European University}
\crest{\includegraphics[width=0.65\textwidth]{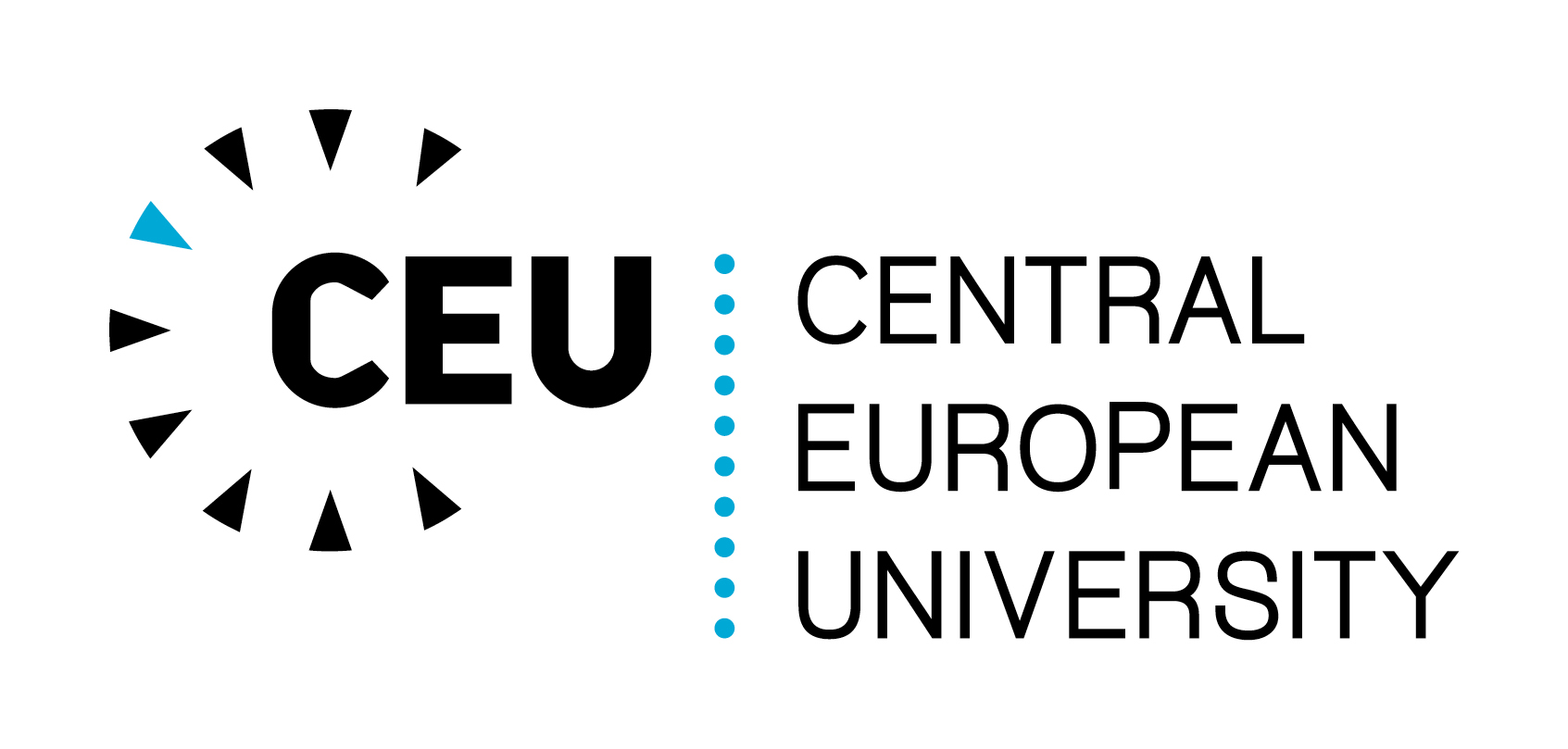}}

\supervisor{Prof.\ Ervin Győri}

\renewcommand{\submissiontext}{In partial fulfillment of the requirements for the degree of}
\degreetitle{Doctor of Philosophy}

\college{}

\degreedate{Budapest, September 2017}

\subject{Mathematics} \keywords{{PhD Thesis} {Mathematics} {Combinatorics} {Central European University, Budapest} {Art gallery theorems} {Orthogonal polygons} {Edge-Disjoint paths} {Terminal-pairability} {Algorithms}}

\ifdefineAbstract
\fi

\ifdefineChapter
\fi

\begin{document}

\frontmatter

\maketitle

\begin{dedication}

	\itshape I would like to dedicate this thesis to my family, without whose support this would never have been written.

\end{dedication}

% ******************************* Thesis Declaration ***************************

\begin{declaration}

	I hereby declare that except where specific reference is made to the work of
	others, the contents of this dissertation are original and have not been
	submitted in whole or in part for consideration for any other degree or
	qualification in this, or any other university.
	This dissertation is my own work and contains nothing which is the outcome of work done in collaboration with others, except as specified in the text.
	% and Acknowledgments

\end{declaration}

% ************************** Thesis Acknowledgements **************************

\begin{acknowledgements}

	I would like to thank Ervin Győri and Gábor Mészáros for suggesting the topics discussed in this thesis. I would also like to thank them for a fruitful collaboration. I am ever so grateful to Ervin for all he has taught me (and the great stories he told me!) during the many hours of supervision over the years.

	\medskip

	I am thankful to the examiners for their helpful comments.

\end{acknowledgements}

\begin{abstract}
	The thesis consists of two parts. In both parts, the problems studied are of significant interest, but are either \textsc{NP}-hard or unknown to be polynomially decidable. Realistically, this forces us to relax the objective of optimality or restrict the problem. As projected by the title, the chosen tool of this thesis is an \emph{extremal type approach}. The lesson drawn by the theorems proved in the thesis is that surprisingly small compromise is necessary on the efficacy of the solutions to make the approach work. The problems studied have several connections to other subjects (e.g.,\ geometric algorithms, graph immersions, multi-commodity flow problem) and practical applications (e.g., VLSI design, image processing, routing traffic in networks).
	Therefore, even slightly improving constants in existing results is beneficial.
	% an abundance of

	\medskip

	The first part of the thesis is concerned with orthogonal art galleries. A sharp extremal bound is proved on partitioning orthogonal polygons into at most 8-vertex polygons using established techniques in the field of art gallery problems. This fills in the gap between already known results for partitioning into at most 6- and 10-vertex orthogonal polygons.

	\medskip

	Next, these techniques are further developed to prove a new type of extremal art gallery result. The novelty provided by this approach is that it establishes a connection between mobile and stationary guards. This theorem has strong computational consequences, in fact, it provides the basis for an $\frac83$-approximation algorithm for guarding orthogonal polygons with rectangular vision.

	\medskip

	In the second part, the graph theoretical concept of terminal-pairability is studied in complete and complete grid graphs. Once again, the extremal approach is conductive to discovering efficient methods to solve the problem.

	\medskip

	In the case of a complete base graph, the new demonstrated lower bound on the maximum degree of realizable demand graphs is 4 times higher than previous best results. The techniques developed are then used to solve the classical extremal edge number problem for the terminal-pairability problem in complete base graphs.

	\medskip

	The complete grid base graph lies on the other end of the spectrum in terms density amongst path-pairable graphs. It is shown that complete grid graphs are relatively efficient in routing edge-disjoint paths. In fact, as a corollary, the minimum maximum degree a path-pairable graph may have is lowered to $O(\log n)$ (prior studies show a lower bound of $\Omega(\log n/\log\log n)$).
\end{abstract}

% *********************** Adding TOC and List of Figures ***********************

\listoftodos

\tableofcontents

\listoffigures

\listoftables

%\listofalgorithms

% \printnomenclature[space] space can be set as 2em between symbol and description
%\printnomenclature[3em]

%\printnomenclature

% ******************************** Main Matter *********************************
\mainmatter

\chapter*{Preface}
\addcontentsline{toc}{chapter}{Preface}

By merely reading the title of this work, the reader might wonder (and justifiably so) how the two main problems discussed in this thesis relate to each other. Well, I believe one of the main connection between them is my and Ervin's taste in mathematics. Let me explain.

\bigskip

Both the art gallery and the terminal-pairability problems encumber a vast family of natural questions. This is a direct consequence of the intuitiveness of these problems: they are very abstract models of challenges that appear in the real world, therefore they lend themselves to innumerable variations.
Unfortunately, the generality of these problems --- their computational complexity is either \textsc{NP}-hard or unknown to be polynomial --- prevents us from finding an optimal solution (a lazy excuse, I know).

\bigskip

However, this is no reason to give up. The logical next step (at least to us) is to relax the goal of seeking an optimal solution to finding a bound, that which a solution achieving is guaranteed to exist. Hence, the purpose of this thesis is to find such bounds that are either sharp (the orthogonal art gallery theorems in Part~\ref{part:artgalleries}), or up to a small constant sharp (the terminal-pairability theorems in Part~\ref{part:terminals}). This \emph{extremal} approach is a main theme of this thesis.

\bigskip

A pleasant phenomenon accompanying this approach is that we are also able to find efficient algorithms that construct the above described solutions. Moreover, our theorems guarantee that these solutions are constant approximations of the optimal solution, and thus are even relevant \emph{in practice}.

\bigskip

I am hoping this preface provides a satisfying explanation of the apparent dichotomy present in the title. Now, I invite you, dear reader, to join me in my 3-year journey into discrete geometry, graph theory, algorithms, complexity, and combinatorics in general. % of the mysteries of

% \nomenclature[z-cif]{$CIF$}{Cauchy's Integral Formula}                                % first letter Z is for Acronyms
% \nomenclature[a-F]{$F$}{complex function}                                                   % first letter A is for Roman symbols
% \nomenclature[g-p]{$\pi$}{ $\simeq 3.14\ldots$}                                             % first letter G is for Greek Symbols
% \nomenclature[g-i]{$\iota$}{unit imaginary number $\sqrt{-1}$}                      % first letter G is for Greek Symbols
% \nomenclature[g-g]{$\gamma$}{a simply closed curve on a complex plane}  % first letter G is for Greek Symbols
% \nomenclature[x-i]{$\oint_\gamma$}{integration around a curve $\gamma$} % first letter X is for Other Symbols
% \nomenclature[r-j]{$j$}{superscript index}                                                       % first letter R is for superscripts
% \nomenclature[s-0]{$0$}{subscript index}                                                        % first letter S is for subscripts
%
%
% \nomenclature[z-DEM]{DEM}{Discrete Element Method}
% \nomenclature[s-crit]{crit}{Critical state}

\part{Orthogonal art galleries}\label{part:artgalleries}
\chapter{Introduction to orthogonal art galleries}

\section{Origins and summary of the new results}
The original art gallery problem was stated by Victor Klee in 1973~\cite{MR0472273}. He posed the following question: given a simple polygon of $n$ vertices, how many stationary guards are required to cover the interior of the polygon? To clarify, a point in the gallery is visible to the guard if the line segment spanned by the point and the guard lies in the closed gallery (line of sight vision).

\medskip

The problem was solved by Vašek Chvátal in 1975:
\begin{theorem}[\citet{Chvatal75}]\label{thm:chávtal}
	$\lfloor\frac{n}{3}\rfloor$ guards are sufficient and sometimes necessary to cover a domain bounded by a simple closed polygon.
\end{theorem}

\medskip

It is easy to see that at least $\lfloor\frac{n}{3}\rfloor$ guards are required, even if only the vertices of the polygon must be covered:
\begin{figure}[H]
	\begin{center}
		\begin{tikzpicture}[scale=0.4]
			\def \n {4}
			\foreach \x in {1,...,\n}
			\draw[fill=none, cap=round] (3*\x+0.5,6)--(3*\x+1,1)--(3*\x+3,1)--(3*\x+3.5,6);

			\draw[fill=none, cap=round] (3*\n+3.5,6)--(3*\n+4,-1)--(3,-1)--(3.5,6);

			% \foreach \x in {0,...,\n}
			% {
			%   \draw[black, fill=black] (3*\x+3.5,0) circle (0.5ex);
			% }
			%
			% \draw (3.5,0) node[anchor=west] {$g$};
			% \draw (1,5) node[anchor=west] {$P$};
		\end{tikzpicture}
	\end{center}
	\caption{The tips of the tooth are only visible from pairwise disjoint regions; therefore, a guard has to be placed in each tooth}
\end{figure}
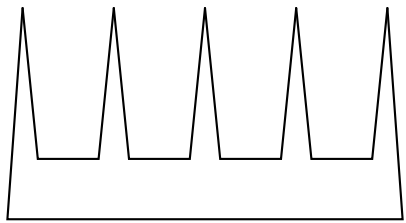

\medskip

The original proof by Chvátal used an inductive partitioning argument with 3 main cases and a couple of small subcases. Not much later, in 1978, Steve Fisk found such a beautiful proof of this theorem, that it is said to be ``from the book'' (see~\cite{MR3288091}).
\begin{proof}[Proof of \Fref{thm:chávtal}~\cite{Fisk78}]
	First, we prove the well-known fact that simple closed polygons can be triangulated, i.e., we can select $n-3$ pairs of vertices of the polygon, such that the line segments spanned by the pairs are in the polygon and these segments may only intersect in their endpoints. The proof is by induction. For $n=3$, the statement is trivial. By sweeping the plane with a line whose slope is different from the slope of every side of the polygon, we can find a convex vertex $v_2$. Let $v_1$ and $v_3$ be its two neighbors.
	\begin{itemize}
		\item If the line segment $\overline{v_1 v_3}$ is intersects the polygon in two points, we found a diagonal. Proceed by induction on the polygon obtained by deleting $v_2$ and adding $\overline{v_1 v_3}$ as a new side.
		\item Otherwise, $L_0=\{v_2\}$ and $L_1=\overline{v_1 v_3}$. For $t\in (0,1)$, let \[ L_t=\{ (1-t)\cdot v_2+t\cdot x\ :\ x\in L_1\},\]
		and take the minimum $t$ for which $L_t$ intersects the polygon in more than 2 points. One of these points must be a vertex $y$ of the polygon, which is not contained in $\overline{v_1 v_2}\cup \overline{v_2 v_3}$. Clearly, $\overline{v_2 y}$ is a diagonal of the polygon, which cuts it into two pieces, say, of $n_1$ and $n_2$ vertices. As $n_1+n_2=n+2$, and $n_1,n_2\ge 3$, we may proceed by induction to obtain $1+(n_1-3)+(n_2-3)=n-3$ diagonals.
	\end{itemize}

	\medskip

	To any triangulation of the (interior) of the polygon there is a corresponding planar graph $G$, whose outer face has exactly $n$ points, but every other face of $G$ is a triangle. Thus, the dual of this graph without the node corresponding to the outer face is a 3-regular tree. If $G$ has 3 vertices, it is trivially 3-colorable. If $G$ has more than 3 vertices, remove a degree 2 vertex (and its edges) of a face which is a leaf in the dual of $G$. By induction, the obtained graph is 3-colorable, and we can easily extend the 3-coloring to the removed degree 2-vertex.

	\medskip

	The smallest color class $A$ has size at most $\lfloor\frac{n}{3}\rfloor$. Clearly, the vertices in $A$ cover the interior of the polygon, as any triangle face has a vertex in $A$.
\end{proof}

Observe, that the proof produces an interesting partition of the domain bounded by the polygon: triangles sharing the same vertex from $A$ form a fan (imagine a handheld one without gaps), which can trivially be covered by one guard.

\medskip

With the original problem of \citeauthor{Chvatal75} solved, interest turned to different variations of the art gallery problem. One such version is when instead of a general polygon, the gallery is assumed to be bounded by an orthogonal polygon.
In 1980, \citeauthor{MR699771} proved that
\begin{theorem}[\citet{MR699771}]\label{thm:K3}
	$\lfloor\frac{n}{4}\rfloor$ guards are sufficient and sometimes necessary to cover a domain bounded by an orthogonal polygon (even on a Riemann surface whose singularities lie outside the polygon).
\end{theorem}
The sharp example is the orthogonal comb:
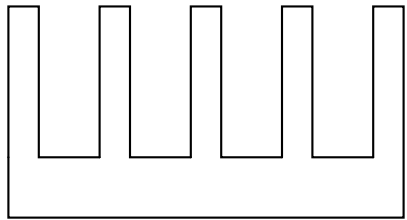
\begin{figure}[H]
	\begin{center}
		\begin{tikzpicture}[scale=0.4]
			\def \n {4}
			\foreach \x in {1,...,\n}
			\draw[cap=round] (3*\x,1)--(3*\x,6)--(3*\x+1,6)--(3*\x+1,1)--(3*\x+3,1);

			\draw[cap=round] (3*\n+3,1)--(3*\n+3,6)--(3*\n+4,6)--(3*\n+4,-1)--(3,-1)--(3,1);

		\end{tikzpicture}
	\end{center}
	\caption{An orthogonal comb; a guard has to be placed for each tooth}
\end{figure}

In fact, the trio proved the following deep geometric lemma.
\begin{lemma}[\cite{MR699771}]\label{lemma:K3}
	Any closed region bounded by a finite number of straight lines, each parallel to one of two orthogonal axes,  has a convex quadrilateralization (even on a Riemann surface whose singularities lie outside the closed region).
\end{lemma}
To prove \Fref{thm:K3} using this lemma, they follow Fisk's argument. We need to assume that the gallery is bounded by an orthogonal polygon (i.e., holes are prohibited), so that that the dual graph of its quadrilateralization (without the outer face) is a tree. Add the two diagonals to each quadrilateral face. Notice, that this graph is 4-colorable, as the degree 2 (3 with the diagonals) vertices of a quadrilateral face which is a leaf in the dual can always be properly colored. The smallest color class covers the gallery.

\medskip

As we noted earlier, the original proof of \Fref{lemma:K3} uses deep geometrical insight, and its proof is about 10 pages long. However, \citet{Lubiw1985} gives a sophisticated and much shorter proof by induction. Moreover, her proof is more general, as it includes certain polygons with holes, and it even leads to an efficient algorithm.

\medskip

Even though \Fref{lemma:K3} applies to even orthogonal polygons with holes, we need simply connectedness to construct the 4-coloring whose existence proves \Fref{thm:K3}.

\medskip

In the first half of the 1980's, Győri and O'Rourke independently gave a simple and short proof of \Fref{thm:K3}.

\begin{theorem}[{\citet{MR844048}~and~\citet[{Thm.~2.5}]{ORourke}}]\label{thm:static}
	Every orthogonal polygon of $n$ vertices can be
	partitioned into $\lfloor  \frac{n}{4} \rfloor$  orthogonal polygons of at most 6 vertices.
\end{theorem}

\Fref{thm:static} is in some aspects a deeper result than that of \citeauthor{MR699771}, as any simple orthogonal polygon of 6 vertices can be covered by a stationary guard.

\medskip

Each proof so far shines light on an interesting phenomenon, which we will refer to as the ``metatheorem'':
\begin{metatheorem}\label{metatheorem}
	Each (orthogonal) art gallery theorem has an underlying
	partition theorem (into simple parts).
\end{metatheorem}

Although both \Fref{thm:K3} and \Fref{thm:static} only apply to simply connected art galleries, Hoffman showed that the same bound holds for any closed region bounded by axis parallel line segments.

\begin{theorem}[\citet{Hoffmann1990}]\label{thm:Hoffmann}
	Any orthogonal polygon with holes of a total of
	$n$ vertices can be partitioned into $\lfloor\frac{n}{4}\rfloor$ rectangular stars of at most 16 vertices.
\end{theorem}

\citeauthor{Hoffmann1990}'s theorem also verifies the metatheorem. Soon after this result, \citet{MR1114588} provided an efficient algorithm to construct such a partition.

\medskip

In \Fref{chap:partition}, we present further evidence that the metatheorem holds,
namely we prove the following partition theorem:

\begin{restatable}[\citet{GyM2016}]{theorem}{thmmobile}\label{thm:mobile}
	Any simple orthogonal polygon of $n$ vertices
	can be partitioned into at most $\lfloor\frac{3n+4}{16}\rfloor$ orthogonal polygons of
	at most 8 vertices.
\end{restatable}

A mobile guard is one who can patrol a line segment in the gallery, and it covers a point $x$ of the gallery if there is a point $y$ on its patrol such that the line segment $[x,y]$ is contained in the gallery. The upper bound of the mobile guard art gallery theorem for orthogonal polygons follows immediately from \Fref{thm:mobile}, as an orthogonal polygon of at most 8 vertices can be covered by a mobile guard.

\begin{theorem}[\citet{Ag}, {\cite[also in][Thm.~3.3]{ORourke}}]\label{thm:mobilecover}
	$\lfloor\frac{3n+4}{16}\rfloor$ mobile guards are sufficient for covering an $n$-vertex simple orthogonal polygon.
\end{theorem}

The lower bound for the previous two theorems is given by stringing together a series of swastikas (\Fref{fig:swastikas}). Observe, that a mobile guard may cover some points of the end of an arm of a swastika for at most one arm. Therefore a mobile guard has to be put in each arm. For $n\not\equiv 0\pmod{16}$, a spiral has to attached to one of the arms.
\begin{figure}[hb]
	\begin{center}
		\begin{tikzpicture}[scale=0.4]
			\def \n {4}
			\foreach \x in {1,...,\n}
			\foreach \y in {0,180}
			\draw[cap=round,shift={(6*\x+\y/30+\y/180,\y/60)},rotate=\y] (0,1.5)--(0,3)--(4,3)--(4,4)--(2,4)--(2,5)--(5,5)--(5,1)--(6,1)--(6,1.5);

			\draw (6,1.5)--(6,0)--(7,0)--(7,1.5);
			\draw (6*\n+6,1.5)--(6*\n+6,3)--(6*\n+7,3)--(6*\n+7,1.5);

			\foreach \x in {1,...,\n}
				{
					\draw[densely dashed,ultra thin,shift={(6*\x,0)}] (2,2) -- (2,3);
					\draw[densely dashed,ultra thin,shift={(6*\x,0)}] (4,0) -- (4,3);
					\draw[densely dashed,ultra thin,shift={(6*\x,0)}] (6,0) -- (6,1);
				}
		\end{tikzpicture}
	\end{center}
	\caption{The dashed lines show a minimum cardinality partition into at most 8-vertex pieces}\label{fig:swastikas}
\end{figure}
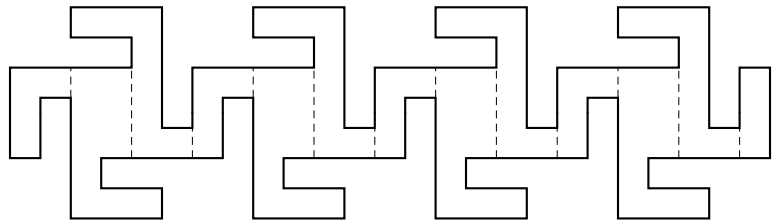

\medskip

\Fref{thm:mobile} is a stronger result than \Fref{thm:mobilecover} and it is
interesting on its own. It fits into the series of  results in
\cites{MR844048}{MR1114588}[Thm.~2.5]{ORourke}{GyH} showing that
orthogonal art gallery theorems are based on theorems on partitions into smaller (``one guardable'') pieces.

\medskip

Moreover, \Fref{thm:mobile} directly implies the following corollary
which strengthens the previous theorem and answers two
questions raised by \citet[Section 3.4]{ORourke}.

\begin{corollary}[\citet{GyM2016}]\label{corollary:mobilecoverstrong}
	$\lfloor\frac{3n+4}{16}\rfloor$ mobile guards are sufficient for covering an $n$-vertex
	simple orthogonal polygon such that the patrols of two guards do
	not pass through one another and visibility is only required at the
	endpoints of the patrols.
\end{corollary}

%Interested readers can find a thorough introduction to the subject of art gallery problems in~\cite{ORourke}.

\begin{table}
	\centering
	\bgroup%
	\def\arraystretch{1.5}
	\begin{tabular}{ c  c  c  c }
		& Point guards &   & Mobile guards \\ \toprule
		Simple polygons            & $\left\lfloor\dfrac{n}{3}\right\rfloor$ &   & $\left\lfloor\dfrac{n}{4}\right\rfloor$     \\ \midrule
		Simple orthogonal polygons & $\left\lfloor\dfrac{n}{4}\right\rfloor$ &   & $\left\lfloor\dfrac{3n+4}{16}\right\rfloor$ \\ \bottomrule
	\end{tabular}
	\egroup%

	\bigskip

	\caption{The extremal number of guards required to cover an $n$-vertex gallery}\label{table:guards}
\end{table}

The results on guarding simple polygons and orthogonal polygons are summarized in \Fref{table:guards}. The proof of the sharp bound on mobile guards in simple polygons due to \citet[Thm.~3.1]{ORourke} also confirms the metatheorem. A combinatorial proof that does not use complex geometric reasoning has already existed for three of the four bounds listed in \Fref{table:guards}. The until recently missing fourth such proof is that of \Fref{thm:mobile}.

\medskip

Joseph O'Rourke pointed out in his 1987 book titled ``Art gallery theorems and algorithms''~\cite{ORourke} that there is a mysterious $4:3$ ratio between the extremal number of point and mobile guards for art galleries given by both  simple polygons and simple orthogonal polygons, as can be seen on \Fref{table:guards}.

\section{Outline of Part~\ref{part:artgalleries}}

After precisely defining the subjects of our study in \Fref{chap:partition}, we introduce the concept of $R$-trees, which is a well-known tool in the literature. The proof of \Fref{thm:mobile} follows.

\medskip

In \Fref{chap:versus} we show that this ratio between the efficacy of point and mobile guards is not only an extremal phenomenon in simple orthogonal polygons appearing for a fixed number of vertices. The magical ratio appears in an upper bound for the ratio of the minimum number of stationary guards covering the gallery and the minimum size of a special, restricted mobile guard cover. The results of \Fref{chap:partition}~and~\ref{chap:versus} have been discovered in collaboration with my supervisor, Ervin Győri.

\medskip

In the last chapter of Part~\ref{part:artgalleries} of this thesis (\Fref{chap:artcomplexity}) we discuss algorithmic versions of our proofs and the computational complexity of orthogonal art gallery problems in general.

\section{Definitions and preliminaries}\label{sec:defs}

Our universe for the study of art galleries is the plane $\mathbb{R}^2$. A \textbf{polygon} is defined by a cyclically ordered list of pairwise distinct vertices in the plane. It is drawn by joining each successive pair of vertices %(including the pair formed by the first and last vertices)
on the list by line segments, that only intersect in vertices of the polygon. The last requirement ensures that the closed domain bounded by the polygon is simply connected (to emphasize this, such polygons are often referred to as simple polygons in the literature). An \textbf{orthogonal polygon} is a polygon such that its line segments are alternatingly parallel to one of the axes of $\mathbb{R}^2$. Consequently, its angles are $\frac12\pi$ (convex) or $\frac32\pi$ (reflex).

\medskip

A \textbf{rectilinear domain} is a closed region of the plane ($\mathbb{R}^2$) whose boundary is an orthogonal polygon, i.e., a closed polygon without self-intersection, so that each segment is parallel to one of the two axes. A \textbf{rectilinear domain with holes} is a rectilinear domain with pairwise disjoint simple rectilinear domain holes. Its boundary is referred to as an \textbf{orthogonal polygon with holes}.

\medskip

The definitions imply that number of vertices of an orthogonal polygon (even with holes) is even. We denote the number of vertices of the polygon by $n(P)$, and define $n(D)=n(P)$, where $D$ is the domain bounded by $P$. Conversely, we write $P=\partial D$. We want to emphasize that in our problems not just the walls, but also the interior of the gallery must be covered. In the proofs of the theorems, therefore, we are working on rectilinear domains, not orthogonal polygons, even though one defines the other uniquely, and vice versa.

\medskip

%Since mostly simply connected domains are discussed in this thesis, we omit the adjective simple in front of orthogonal polygons and rectilinear domains.
Whenever results about objects that are allowed to have holes are mentioned, it is explicitly stated.

\medskip

To avoid confusion, we state that throughout this part, \textbf{vertices} and \textbf{sides} refer to subsets of an orthogonal polygon or a rectilinear domain; whereas any \textbf{graph} will be defined on a set of \textbf{nodes}, of which some pairs are joined by some \textbf{edges}. Given a graph $G$, the edge set $E(G)$ is a subset of the 2-element subsets of the vertices $V(G)$.

\medskip

Unless otherwise noted, we adhere to the same terminology in the subject of art galleries as O'Rourke~\cite{ORourke}.
However, for technical reasons, sometimes we need to assume extra conditions over what is traditionally assumed. In \Fref{lemma:technical}, we prove that we may, without restricting the problem, require the assumptions typeset in \emph{italics} in the following definitions.

\medskip

Two points $x,y$ in a domain $D$ have \textbf{line of sight vision}, \textbf{unrestricted vision}, or simply just \textbf{vision} of each other if the line segment induced by $x$ and $y$ is contained in $D$.

\medskip

A \textbf{point guard} in an art gallery $D$ is a point $y\in D$. It has vision of a point $x\in D$ if the line segment $\overline{xy}$ is a subset of $D$. The term ``stationary guard'' refers to the same meaning, and is used mostly in contrast with ``mobile guards''.

\medskip

A \textbf{mobile guard} is a line segment $L\subset D$. A point $x\in D$ is seen by the guard if there is a point $y\in L$ which has vision of $y$. Intuitively, a mobile guard is a point guard patrolling the line segment $L$.

\medskip

The points \textbf{covered by a guard} is just another name for the set of points of $D$ that are seen by the guard. A \textbf{system of guards} is a set of guards in $D$ which cover $D$, i.e., for any point $x\in D$, there is a guard in the system covering $x$.

\medskip

Two points $x,y$ in a rectilinear domain $D$ have \textbf{\boldmath $r$-vision} of each other (alternatively, $x$ is $r$-visible from $y$) if there exists an axis-aligned \emph{non-degenerate} % amúgy nem kell, ha úgyis átírjuk a polygont
rectangle in $D$ which contains both $x$ and $y$. This vision is natural to use in orthogonal art galleries instead of the more powerful line of sight vision. For example, $r$-vision is invariant on the transformation depicted on \Fref{fig:mgcond1}.

\medskip

A \textbf{point \boldmath $r$-guard} is a point $y\in D$, \emph{such that the two maximal axis-parallel line segments in $D$ containing $y$ do not intersect vertices of $D$}. A set of point guards \hbox{\textbf{\boldmath $r$-cover}} $D$ if any point $x\in D$ is $r$-visible from a member of the set. Such a set is called a \textbf{point \boldmath $r$-guard system}.

\medskip

A \textbf{vertical mobile \boldmath $r$-guard} is a vertical line segment in $D$, \emph{such that the maximal line segment in $D$ containing it does not intersect vertices of~$D$}. \textbf{Horizontal} mobile guards are defined analogously. A \textbf{mobile \boldmath $r$-guard} is either a vertical or a horizontal mobile $r$-guard. A mobile $r$-guard \hbox{\textbf{\boldmath  $r$-covers}} any point $x\in D$ for which there exists a point $y$ on its line segment such that $x$ is $r$-visible from $y$.

\begin{lemma}\label{lemma:technical}
	Any rectilinear domain $D$ can be transformed into another rectilinear domain $D'$ so that the point guard $r$-cover, and the vertical/horizontal mobile guard $r$-cover problems in $D$, without the restrictions typeset in italics, are equivalent to the respective problems, as per our definitions (i.e., with the restrictions), in $D'$.
\end{lemma}
\begin{proof}
	%	Degenerate rectangles only contribute a 0-measure set to the points visible from fixed point of $D$. Furthermore, if a set of guards cover $D$ via $r$-vision (as defined here) except a 0-measure part of it, then actually they cover $D$ entirely; the proof is trivial. Therefore, we may prohibit vision by degenerate rectangles without loss of generality.
	\begin{figure}[ht]
	\centering
	\begin{tikzpicture}
	\def\dist{0.5}
	\begin{scope}[yscale=1.5,shift={(0*\dist,0.33*\dist)},rotate=-90]
		\fill[color=inside]  (0,7*\dist) rectangle (2*\dist,8*\dist) rectangle (\dist,4.5*\dist) rectangle (0,6*\dist) rectangle (\dist,2*\dist) rectangle (2*\dist,3*\dist) rectangle (\dist,-1.5*\dist) rectangle (0,0) rectangle (\dist,-3*\dist);

		\draw (0,0) -- ++(\dist,0) -- ++(0,2*\dist) -- ++(-\dist,0);
		\draw (2*\dist,3*\dist) -- ++(-1*\dist,0) -- ++(0,1.5*\dist) -- ++(\dist,0);
		\draw (0,6*\dist) -- ++(\dist,0) -- ++(0,\dist) -- ++(-\dist,0);
		\draw (0,8*\dist) -- ++(2*\dist,0);
		\draw (0,-3*\dist) -- ++(\dist,0) -- ++(0,1.5*\dist) -- ++(\dist,0);
	\end{scope}

	\draw [very thick,decorate,decoration={coil,aspect=0},->] (8.5*\dist,-\dist) -- (10.5*\dist,-\dist);

	\begin{scope}[shift={(14*\dist,0.5*\dist)}, rotate=-90]
		\fill[color=inside]  (0,7*\dist) rectangle (3*\dist,8*\dist) rectangle (\dist,4.5*\dist) rectangle (0,6*\dist);
		\fill[color=inside] (0,5*\dist) rectangle (2*\dist,2*\dist);
		\fill[color=inside] (3*\dist,3*\dist) rectangle (\dist,-1.5*\dist);
		\fill[color=inside] (0,0) rectangle (2*\dist,-3*\dist);

		\draw (0,0) -- ++(\dist,0) -- ++(0,2*\dist) -- ++(-\dist,0);
		\draw (3*\dist,3*\dist) -- ++(-1*\dist,0) -- ++(0,1.5*\dist) -- ++(\dist,0);
		\draw (0,6*\dist) -- ++(\dist,0) -- ++(0,\dist) -- ++(-\dist,0);
		\draw (0,8*\dist) -- ++(3*\dist,0);
		\draw (0,-3*\dist) -- ++(2*\dist,0) -- ++(0,1.5*\dist) -- ++(\dist,0);
	\end{scope}
\end{tikzpicture}
\caption{After this transformation, those mobile guards whose maximal containing line segment do not intersect vertices of the rectilinear domain, are just as powerful as mobile guards that are not restricted in such a way.}\label{fig:mgcond1}
\end{figure}
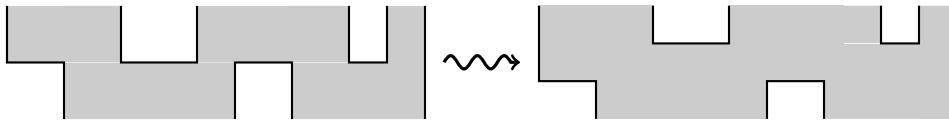
	Let $\varepsilon$ be the minimal distance between any two horizontal line segments of $\partial D$. The transformation depicted in \Fref{fig:mgcond1} in $D$ takes a maximal horizontal line segment $L$ in $D$ which is touched from both above and below by the exterior of $D$, and maps $D$ to \[D'=D\bigcup \left(L+\overline{(0,-\varepsilon/4)(0,\varepsilon/4)}\right),\]
	where addition is taken in the Minkowski sense.
	There is a trivial correspondence between the point and mobile guards of $D$ and $D'$ such that taking this correspondence guard-wise transforms a guarding system of $D$ (guards without the restrictions) into a guarding system of $D'$ (guards with the restrictions), and vice versa.

	\medskip

	After performing this operation at every vertical and horizontal occurrence, we get a rectilinear domain $D''$, in which any vertical or horizontal line segment is contained in a non-degenerate rectangle in $D''$. Therefore, degenerate vision between any two points implies non-degenerate vision between the pair. Furthermore, the line segment of any mobile guard can be translated slightly along its normal (at least in one direction) while staying inside $D''$, and this clearly does not change the set of points $r$-covered by the guard. Similarly, we can perturb the position of a point guard without changing the set of points of $D''$ it $r$-covers.
\end{proof}

%!TEX root = thesis.tex
%*******************************************************************************
%****************************** Second Chapter *********************************
%*******************************************************************************

\chapter{Partitioning orthogonal polygons}\label{chap:partition}

\section{Introduction}

For the sake of completeness, we mention that any orthogonal polygon of $n$ vertices can be partitioned into at most $\lfloor\frac{n}{2}\rfloor -1$ rectangles, and this bound is sharp. In this case, the interesting question is the minimum number of covering rectangles, see \Fref{chap:artcomplexity}.

\medskip

\Fref{thm:mobile} fills in a gap between two already established (sharp) results: in~\cite{MR844048} it is proved that orthogonal polygons can be partitioned into at most $\lfloor\frac{n}{4}\rfloor$ orthogonal polygons of at most 6 vertices, and in~\cite{GyH} it is proved that any orthogonal polygon in general position (an orthogonal polygon without 2-cuts) can be partitioned into $\lfloor\frac{n}{6}\rfloor$ orthogonal polygons of at most $10$ vertices. However, we do not know of a sharp theorem about partitioning orthogonal polygons into orthogonal polygons of at most 12 vertices.

\medskip

Furthermore, for $k\ge 4$, not much is known about partitioning orthogonal polygons with holes into orthogonal polygons of at most $2k$ vertices. Per the ``metatheorem,'' the first step in this direction would be proving that an orthogonal polygon of $n$ vertices with $h$ holes can be partitioned into $\lfloor\frac{3n+4h+4}{16}\rfloor$ orthogonal polygons of at most 8 vertices. This would generalize the corresponding art gallery result in~\cite[Thm.~5.]{GyH}.

\medskip

% To demonstrate the usefulness of $R$-trees, we first reprove \Fref{thm:static}. %Actually, we are reproducing O'Rourke's proof of the theorem, by proving that only straight cuts are needed to obtain the required partition.
%
% \begin{proof}[Proof of \Fref{thm:static}]
%   For $n\le 6$, the theorem is trivial. By induction, it is sufficient to find a cut of $D$ that creates two rectilinear parts of $n_1$ and $n_2$ vertices, such that
%   \begin{equation}
%   \left\lfloor\frac{n_1}{2}\right\rfloor + \left\lfloor\frac{n_2}{2}\right\rfloor \le \left\lfloor\frac{n}{2}\right\rfloor.\label{ineq:nover4}
%   \end{equation}
%   Let $T_H$ be the horizontal $R$-tree of $D$. If any of the edges of $T_H$ correspond to a 2-cut, then $n_1+n_2=n$, which implies \fref{ineq:nover4}.
%   It is sufficient to prove that there is an edge in one of the $R$-trees which, if cut, creates an even number of
%   \begin{description}
%   \item[$\boldmath {\exists h\in T_H\text{ s.t.\ }d_{T_H}(h)=2}$.] We have
%   \end{description}
%   \mynote{Is a proof of Theorem 4 necessary? Definitions would need to be reordered}
% \end{proof}

The proof of \Fref{thm:mobile} is similar to the
proofs of \Fref{thm:static} in that it finds a suitable cut and then
uses induction on the parts created by the cut. However, a cut
along a line segment connecting two reflex vertices is no longer automatically good. We also rely heavily on a tree structure
of the orthogonal polygon (\Fref{sec:treestructure}). However, while O'Rourke's
proof of \Fref{thm:static} only uses straight cuts, in our case this is not sufficient:
\Fref{fig:LisNeccessary} shows an orthogonal polygon of 14 vertices which cannot
be cut into 2 orthogonal polygons of at most 8 vertices using cuts along
straight lines. Therefore, we must consider L-shaped cuts too.

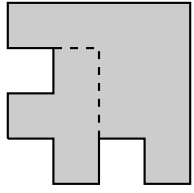
\begin{figure}[ht]
  \centering
  \begin{tikzpicture}[scale=0.6]
  \begin{scope}

  \fill[color=inside] (2,0) rectangle (5,3);
  \fill[color=inside] (1,0) rectangle (2,1);
  \fill[color=inside] (1,2) rectangle (2,3);
  \fill[color=inside] (2,-1) rectangle (3,0);
  \fill[color=inside] (4,-1) rectangle (5,0);

  \draw (1,0)	-- (1,1) -- (2,1) -- (2,2) -- (1,2) -- (1,3) -- (5,3) -- (5,-1) --(4,-1) -- (4,0) -- (3,0) -- (3, -1) -- (2,-1) -- (2,0) -- (1,0);

  \draw[dashed] (3,0) -- (3,2) -- (2,2);

  \end{scope}
  \end{tikzpicture}
  \caption{An L-shaped cut creating a partition into 8-vertex orthogonal polygons}\label{fig:LisNeccessary}
\end{figure}

\section{Definitions and preliminaries}

Let $D,D_1,D_2$ be rectilinear domains of $n,n_1,n_2$ vertices, respectively. If $D=D_1\cup D_2$, $\mathrm{int}(D_1)\cap \mathrm{int}(D_2)=\emptyset$, $0<n_1,n_2$ and $n_1+n_2\le n+2$ are satisfied, we say that $D_1,D_2$ form an \textbf{admissible partition} of $D$, which we denote by $D=D_1\dotcup D_2$.
Also, we call $L=D_1\cap D_2=\partial D_1\cap \partial D_2$ a \textbf{cut} in this case. We may describe this relationship concisely by $L(D_1,D_2)$.
If, say, we have a number of cuts $L_1,L_2$,~etc.,~then we usually write $L_i(D_1^i,D_2^i)$. Generally, if a rectilinear domain is denoted by $D_x^y$, then $y$ refers to a cut and $x\in\{1,2\}$ is the label of the piece in the partition created by said cut.
%, and let a directed cut be a cut plus an order on the parts, which we denote by $\vr L=(D_1,D_2)$.
Furthermore, if
\begin{equation}\label{eq:n1n2}
  \left\lfloor\frac{3n_1+4}{16}\right\rfloor+\left\lfloor\frac{3n_2+4}{16}\right\rfloor\le \left\lfloor\frac{3n+4}{16}\right\rfloor.
\end{equation}
is also satisfied, we say that $D_1,D_2$ form an \textbf{induction-good partition} of $D$, and we call $L$ a \textbf{good~cut}.

\begin{lemma}\label{lemma:tech}
  An admissible partition $D=D_1\dotcup D_2$ is also and induction-good partition if

  \smallskip

  \textbf{\mbox{}\quad (a)}
  \customlabel{lemma:tech:a}{(a)}
  $n_1+n_2= n+2$ and $n_1\equiv 2,8,\text{ or }14 \pmod{16}$, \\
  \mbox{}\qquad\qquad\textbf{or}\\
  \textbf{\mbox{}\quad (b)}
  \customlabel{lemma:tech:b}{(b)}
  $n_1+n_2= n$ and $n_1\equiv 0,2,6,8,12,\text{ or }14 \pmod{16}$, \\
  \mbox{}\qquad\qquad\textbf{or}\\
  \textbf{\mbox{}\quad (c)}
  \customlabel{lemma:tech:c}{(c)}
  $n\not\equiv 14\mod 16$ and \textbf{either}
  \vspace{-12pt}
  \begin{align*}
  n_1+n_2 & =n   &   & \text{ and } &   & n_1\equiv 10 \pmod{16},\textbf{\quad or \quad\qquad} \\
  n_1+n_2 & =n+2 &   & \text{ and } &   & n_1\equiv 12 \pmod{16}.
  \end{align*}
\end{lemma}
\begin{proof}
  Using the fact that the floor function satisfies the triangle inequality, the proof reduces to an easy case-by-case analysis, which we leave to the reader.
\end{proof}

Any cut $L$ falls into one of the following 3 categories (see \Fref{fig:cuts}):
\begin{enumerate}
  \item[\bfseries (a)] \textbf{1-cuts:} $L$ is a line segment, and exactly one of its endpoints is a (reflex) vertex of $D$.

  \item[\bfseries (b)] \textbf{2-cuts:} $L$ is a line segment, and both of its
        endpoints are (reflex) vertices of $D$.

  \item[\bfseries (c)] \textbf{L-cuts:} $L$ consists of two connected line
        segments, and both endpoints of $L$ are (reflex) vertices of $D$.
\end{enumerate}

Note that for 1-cuts and L-cuts the size of the parts satisfy $n_1+n_2=n+2$, while for 2-cuts we have $n_1+n_2=n$.

%\begin{claim}\label{claim:glue}
%	Suppose $D_1$ and $D_2$ are two rectilinear domains, and $L:=D_1\cap D_2$
%	is a polygonal path. Then $D=D_1\dotcup D_2$ is a
%	a rectilinear domain partition which is also defined by $L$ as in
%	\Fref{lemma:polygoncut}.
%\end{claim}
%
%Now we may refer to two part admissible partitions as cuts and
%vice-versa.

\begin{figure}[H]
  \centering
  \begin{subfigure}{.50\textwidth}
  \renewcommand{\thesubfigure}{a}
  \centering
  \begin{tikzpicture}

  \begin{scope}

  \fill[color=inside] (0,0) rectangle (-2,1) rectangle (1,2);
  %\fill[color=white] (-3,0) rectangle (-2,1);

  \draw (0,0) -- (0,1) -- (1,1);

  \draw (-2,0) -- (-2,2);

  \draw[dashed] (0,1) -- (-1,1) node {\ScissorHollowLeft}-- (-2, 1);

  \end{scope}
  \end{tikzpicture}
  \caption{1-cut}
  \end{subfigure}%
  \begin{subfigure}{.50\textwidth}
  \renewcommand{\thesubfigure}{b1}
  \centering
  \begin{tikzpicture}

  \begin{scope}

  \fill[color=inside] (-3,0) rectangle (0,1);
  \fill[color=inside] (-2,1) rectangle (1,2);

  \draw (0,0) -- (0,1) -- (1,1);

  \draw (-3,1) -- (-2,1) -- (-2,2);

  \draw[dashed] (0,1) -- (-1,1) node {\ScissorHollowLeft}-- (-2, 1) ;

  \end{scope}
  \end{tikzpicture}
  \caption{2-cut}
  \end{subfigure}

  \bigskip

  \begin{subfigure}{.50\textwidth}
  \renewcommand{\thesubfigure}{b2}
  \centering
  \begin{tikzpicture}

  \begin{scope}

  \fill[color=inside] (-3,1) rectangle (1,2);
  \fill[color=inside] (-2,0) rectangle (0,1);
  %\fill[color=outside] (0,0) rectangle (1,1);
  %\fill[color=outside] (-3,0) rectangle (-2,1);

  \draw (0,0) -- (0,1) -- (1,1);

  \draw (-3,1) -- (-2,1) -- (-2,0);

  \draw[dashed] (0,1) -- (-1,1) node {\ScissorHollowLeft}-- (-2, 1) ;

  %\path (-1,0.4) node {Corridor};

  \end{scope}
  \end{tikzpicture}
  \caption{2-cut}
  \end{subfigure}%
  \begin{subfigure}{.50\textwidth}
  \renewcommand{\thesubfigure}{c}
  \centering
  \begin{tikzpicture}

  \begin{scope}

  \fill[color=inside] (-3,1) rectangle (-2,2);
  \fill[color=inside] (-2,0.5) rectangle (0,2) rectangle (-1,2.5);

  \draw (-3,1) -- (-2,1) -- (-2,0.5);
  \draw (-1,2.5) -- (-1,2) -- (-3,2);
  %\path (-1,2) node[rectangle,fill,rotate=45] {};

  \draw[dashed] (-2, 1) -- (-1,1) node[rotate = -135] {\ScissorHollowLeft} -- (-1,2);

  \end{scope}
  \end{tikzpicture}
  \caption{L-cut}
  \end{subfigure}
  \caption{Examples for all types of cuts. Light gray areas are subsets of $\mathrm{int}(D)$.}\label{fig:cuts}
\end{figure}
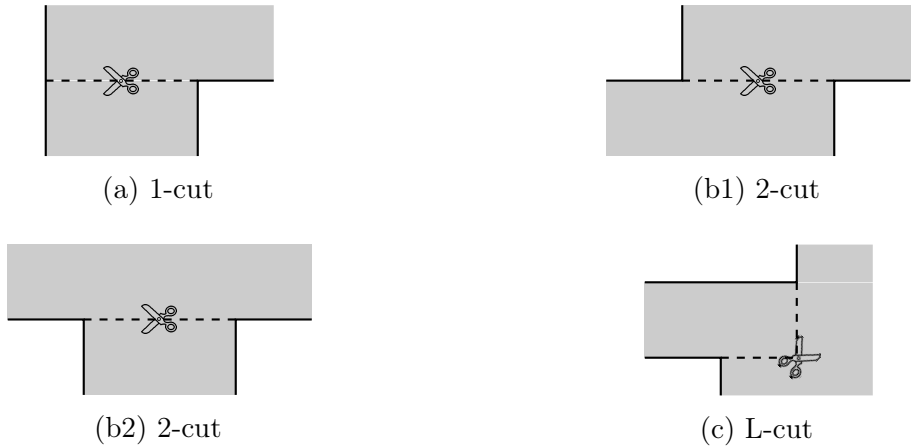

In the proof of \Fref{thm:mobile} we are searching for an induction-good partition of $D$. As a good cut defines an induction-good partition, it is sufficient to find a good cut. We could hope that a good cut of a rectilinear piece of $D$ is extendable to a good cut of $D$, but unfortunately a good cut of a rectilinear piece of $D$ may only be an admissible cut with respect to $D$ (if it is a cut of $D$ at all). \Fref{lemma:tech}, however, allows us to look for cut-systems containing a good cut. % and are preserved under broader conditions.
Fortunately, it is sufficient to consider non-crossing, nested cut-systems of at most 3 cuts, defined as follows.

\begin{definition}[Good cut-system]\label{def:goodcutsystem}
  The cuts $L_1(D_1^1,D_2^1)$, $L_2(D_1^2,D_2^2)$ and $L_3(D_1^3,D_2^3)$ (possibly $L_2=L_3$) constitute a good cut-system if $D_1^1\subset D_1^2\subseteq D_1^3$, and the set
  \[ \left\{n(D_1^i)\ |\ i\in\{1,2,3\}\right\}\cup \left\{n(D_1^i)+2\ |\ i\in\{1,2,3\}\text{ and }L_i\text{ is a 2-cut}\right\} \]
  contains three consecutive even elements modulo 16 (i.e., the union of their residue classes contains a subset of the form $\{a,a+2,a+4\}+16\mathbb{Z}$). If this is the case we also define their kernel as $\ker\{L_1,L_2,L_3\}=(D_1^1-L_1)\cup (D_2^3-L_3)$, which will be used in \Fref{lemma:extendconsecutives}.
  %  Let $\mathcal{L}$ be a set of cuts of $D$, where $|\mathcal{L}|=2$ or $|\mathcal{L}|=3$. Let $I=\{1\le i\le |\mathcal{L}|\}\cap \mathbb{N}$ We say that $\mathcal{L}$ is a good cut-system (of $D$) if $\exists \tau\in \{1,2\}^I$ and $\exists\pi\in \mathrm{Sym}(I)$ (a permutation of $I$) such that $D_{\tau(1)}(L_{\pi(1)})\subseteq D_{\tau(2)}(L_{\pi(2)}) \Big(\subseteq D_{\tau(3)}(L_{\pi(3)})\Big)$, and the set
  %    \begin{align*}
  %      W_{\tau,\pi}(\mathcal{L},D)=\Big\{ n(D_{\tau(i)}(L_{\pi(i)}))\ |\ i\in I\Big\}\cup \Big\{n(D_{\tau(i)}(L_{\pi(1)}))+2\ |\ i\in I\text{ and $L_{\pi(i)}$ is a 2-cut}\Big\}
  %    \end{align*}
  %   contains three consecutive even elements modulo $16$. Lastly, define $$\ker\mathcal{L}:=D_{\tau(1)}(L_{\pi(1)})\cup D_{3-\tau(3)}(L_{\pi(3)})-L_{\pi(1)}-L_{\pi(3)}.$$
\end{definition}

\Fref{lemma:tech:a} and~\ref{lemma:tech:b} immediately yield that any good cut-system contains a good cut.

\begin{remark}\label{rem:invert}
  It is easy to see that if a set of cuts satisfies this definition, then they obviously satisfy it in the reverse order too (the order of the generated parts is also switched).
  %The only non-trivial part in this statement is that $x\mapsto n(D)+2-x$ is a bijection between the sets that contain the three consecutive even elements modulo 16.
  %\end{remark}
  %\begin{remark}\label{rem:cutset}
  Actually, these are exactly the two orders in which they do so. Thus, the kernel is well-defined, and when speaking about a good cut-system it is often enough to specify the set of participating cuts.
\end{remark}

%\begin{claim}\label{claim:goodsystem}
%	Any good cut-system contains at least one good cut.
%\end{claim}
%\begin{proof}
%	Follows immediately from \Fref{lemma:tech}.
%\end{proof}

%The reason for the introduction of the previous definition is that \Fref{lemma:tech} implies that any good cut-system contains a good cut.

\medskip

%\begin{proof}
%%	This follows immediately from \Fref{lemma:tech}.
%	If the modulo 16 remainder of $n(D_1^{L_i,s})$ is one of $2,8,14$ for
%	some $i\in I$, then $L_i$ is a good cut by \Fref{lemma:tech}.
%	Otherwise there exists $i\in I$ s.t. $n(D_1^{L_i,s})+2$ is congruent to one of
%	$2,8,14$ modulo $16$, therefore $n(D_1^{L_i,s})$ is congruent to one of $0,6,12$ modulo $16$, thus $L_i$ is a good 2-cut by	\Fref{lemma:tech:b}.
%\end{proof}

%Furthermore the following lemma, which we can use to change our point of view from $s$ to another point of $D$, follows immediately.
%\begin{claim}\label{claim:invert}
%	Let $\{ L_i\ |\ i\in I\}$ be a good cut-system with respect to $s$.
%	If $s'\in \cap_{i\in I}\mathrm{int}(D_j^{L_i, s})$ for some $j\in\{1,2\}$, then $\{	L_i\ |\ i\in I\}$ is a good cut-system with respect to $s'$ too.
%\end{claim}

\section{Tree structure}\label{sec:treestructure}

Any reflex vertex of a rectilinear domain $D$ defines a (1- or 2-) cut along a horizontal line segment whose interior is contained in $\mathrm{int}(D)$ and whose endpoints are the reflex vertex and another point on the boundary of
$D$. Next, we define a graph structure derived from $D$, which is a standard tool in the literature, for example it is called the $R$-graph of an orthogonal polygon in~\cite{GyH}. A similar structure is used by O'Rourke to prove \Fref{thm:static}, see~\cite[p.~76]{ORourke}.

\medskip

\begin{definition}[Horizontal (vertical) $R$-tree]\label{def:tree}
  The horizontal (vertical) $R$-tree $T$ of a rectilinear domain $D$ (or the orthogonal polygon $\partial D$ bounding it) is obtained as follows. First, partition $D$ into a set of rectangles by cutting along all the horizontal (vertical) cuts of $D$. Let $V(T)$, the vertex set of $T$ be the set of resulting (internally disjoint) rectangles. Two rectangles of $T$ are connected by an edge in $E(T)$ iff their boundaries intersect.
\end{definition}

\medskip

The graph $T$ is indeed a tree as its connectedness is trivial and since any cut creates two internally disjoint rectilinear domains, $T$ is also cycle-free. We can think of $T$ as a sort of dual of the planar graph determined by the union of $\partial D$ and its horizontal cuts. The nodes of $T$ represent rectangles of $D$ and edges of $T$ represent horizontal 1- and 2-cuts. For this reason, we may refer to nodes of $T$ as rectangles. This nomenclature also helps in distinguishing between vertices of $D$ (points) and nodes of $T$. Moreover, for an edge $e\in E(T)$, we may denote the cut represented by $e$ by simply $e$, as the context should make it clear whether we are working in the graph $T$ or in the plane.

\medskip

Note that the vertical sides of rectangles are also edges of the orthogonal polygon bounding $D$.

%\begin{claim}
%
%\end{claim}
%\begin{proof}
%   Given any two rectangles in $V(T)$, there exists a polygonal path
%   $M$ in $\mathrm{int}(D)$ which connects their center points.
%   A corresponding walk in $T$ can be obtained by taking the
%   rectangles and horizontal cuts $M$ successively intersects.
%   Therefore $T$ is connected. Also, cutting at any horizontal cut
%   (edge of $T$) makes $D$ disconnected, so $T$ is cycle-free.
%\end{proof}

\begin{definition}\label{def:tfunc}
  Let $T$ be the horizontal $R$-tree of $D$. Define $t:E(T)\to \mathbb{Z}$ as follows:
  given any edge $\{R_1,R_2\}\in E(T)$, let
  \[ t(\{R_1,R_2\})=n(R_1\cup R_2)-8. \]
\end{definition}

Observe that
\[ t(e)=\left\{
  \begin{array}{rl}
  0,  & \text{ if $e$ represents a 2-cut;} \\
  -2, & \text{ if $e$ represents a 1-cut.} \\
  \end{array}
  \right. \]
The following claim is used throughout the chapter to count the number of vertices of a rectilinear piece of $D$.
\begin{claim}\label{claim:tsize} Let $T$ be the horizontal $R$-tree of $D$. Then
  \[ n(D)=4|V(T)|+\sum_{e\in E(T)} t(e). \]
\end{claim}
\begin{proof}
  The proof is straightforward.
  %	Traverse the simple polygon that is the boundary of $D$. Write $-2$ on points of $\partial D$ which are vertices of one of the rectangles, but not vertices of $\partial D$.
  %
  %	By induction on $|V(T)|$. Let $R_1$ be a leaf in $T$, and let $R_2$ be its neighbor. Then $$D'=\bigcup_{R\in V(T)-R_1} R$$ is a rectilinear domain and $V(T)-R_1$ is one of its horizontal $R$-trees. Applying the induction hypothesis we have
  %\begin{align*}
  %	n(D')&=4(|V(T)|-1)+\sum_{e\in E(T)-\{R_1,R_2\}}t(e)=n(D)-(4+t(e))=\\
  %	&=n(D)-(n(R_1\cup R_2)-4)=n(D)-(n(R_1\cup R_2)-n(R_2))
  %\end{align*}
  %To finish the proof, just verify that $n(D)=n(D')-n(R_2)+n(R_1\cup R_2)$.
\end{proof}
\begin{remark}\label{rem:refine}
  Equality in the previous claim holds even if some of the rectangles of $T$ are cut into several rows (and the corresponding edges, for which the function $t$ takes $-4$, are added to $T$). % of equal width
\end{remark}

\section{Extending cuts and cut-systems}
The following two technical lemmas considerably simplify our analysis in \Fref{sec:proof}, where many cases distinguished by the relative positions of reflex vertices of $D$ on the boundary of a rectangle need to be handled. For a rectangle $R$ let us denote its top left, top right, bottom left, and bottom right vertices with $v_{\textit{TL}}(R)$, $v_{\textit{TR}}(R)$, $v_{\textit{BL}}(R)$, and $v_{\textit{BR}}(R)$, respectively.

\begin{figure}[H]
  \centering
  \begin{subfigure}{.33\textwidth}
  \centering
  \begin{tikzpicture}

  \begin{scope}

  \fill[color=inside] (0,0) rectangle (2,1);
  \fill[color=inside] (0,1) rectangle (1,2);

  \draw (0,0) -- (0,2) -- (1,2) -- (1,1) -- (2,1) -- (2,0) -- (0,0);
  \draw[dashed] (0,1) -- (1,1);

  \path (0.5,0.5) node {$R$} -- (0.5,1.5) node {$Q$};

  \filldraw (0,1) circle (2pt) node[anchor=east] {$v_{\textit{TL}}(R)$};
  \filldraw (2,1) circle (2pt) node[anchor=south] {$v_{\textit{TR}}(R)$};

  \end{scope}
  \end{tikzpicture}
  \caption{$Q\subseteq R_{\textit{TL}}$, and $e_{\textit{TL}}(R)$ is a 1-cut}\label{fig:RTL:a}
  \end{subfigure}%
  \begin{subfigure}{.34\textwidth}
  \centering
  \begin{tikzpicture}

  \begin{scope}

  \fill[color=inside] (0,0) rectangle (1,1);
  \fill[color=inside] (0,1) rectangle (2,2);

  \draw (0,0) -- (0,2) -- (2,2) -- (2,1) -- (1,1) -- (1,0) -- (0,0);
  \draw[dashed] (0,1) -- (1,1);

  \path (0.5,0.5) node {$R$} -- (0.5,1.5) node {$Q$};

  \filldraw (0,1) circle (2pt) node[anchor=east] {$v_{\textit{TL}}(R)$};
  \filldraw (1,1) circle (2pt) node[anchor=north west] {$v_{\textit{TR}}(R)$};

  \end{scope}
  \end{tikzpicture}
  \caption{$Q\subseteq R_{\textit{TL}}$, $e_{\textit{TL}}(R)$ is a 1-cut, and $R_{\textit{TR}}=\emptyset$}\label{fig:RTL:b}
  \end{subfigure}%
  \begin{subfigure}{.33\textwidth}
  \centering
  \begin{tikzpicture}

  \begin{scope}

  \fill[color=inside] (0,0) rectangle (2,1);
  \fill[color=inside] (-1,1) rectangle (1,2);

  \draw (0,0) -- (0,1) -- (-1,1) -- (-1,2) -- (1,2) -- (1,1) -- (2,1) -- (2,0) -- (0,0);
  \draw[dashed] (0,1) -- (1,1);

  \path (0.5,0.5) node {$R$} -- (0.5,1.5) node {$Q$};

  \filldraw (0,1) circle (2pt) node[anchor=north east] {$v_{\textit{TL}}(R)$};
  \filldraw (2,1) circle (2pt) node[anchor=south] {$v_{\textit{TR}}(R)$};

  \end{scope}
  \end{tikzpicture}
  \caption{$Q\subseteq R_{\textit{TL}}$, and $e_{\textit{TL}}(R)$ is a 2-cut}\label{fig:RTL:c}
  \end{subfigure}%

  \bigskip

  \begin{subfigure}{.33\textwidth}
  \centering
  \begin{tikzpicture}

  \begin{scope}

  \fill[color=inside] (0,0) rectangle (1,1);
  \fill[color=inside] (-1,1) rectangle (2,2);

  \draw (0,0) -- (0,1) -- (-1,1) -- (-1,2) -- (2,2) -- (2,1) -- (1,1) -- (1,0) -- (0,0);
  \draw[dashed] (0,1) -- (1,1);

  \path (0.5,0.5) node {$R$} -- (0.5,1.5) node {$Q$};

  \filldraw (0,1) circle (2pt) node[anchor=north east] {$v_{\textit{TL}}(R)$};
  \filldraw (1,1) circle (2pt) node[anchor=north west] {$v_{\textit{TR}}(R)$};

  \end{scope}
  \end{tikzpicture}
  \caption{$R_{\textit{TL}}=R_{\textit{TR}}=\emptyset$, $R$ is either a corridor or a pocket (see \Fref{sec:proof})}\label{fig:RTL:d}
  \end{subfigure}%
  \begin{subfigure}{.33\textwidth}
  \centering
  \begin{tikzpicture}

  \begin{scope}

  \fill[color=inside] (0,0) rectangle (1,1);
  \fill[color=inside] (-1,1) rectangle (1,2);

  \draw (0,0) -- (0,1) -- (-1,1) -- (-1,2) -- (1,2) -- (1,0) -- (0,0);
  \draw[dashed] (0,1) -- (1,1);

  \path (0.5,0.5) node {$R$} -- (0.5,1.5) node {$Q$};

  \filldraw (0,1) circle (2pt) node[anchor=north east] {$v_{\textit{TL}}(R)$};
  \filldraw (1,1) circle (2pt) node[anchor=west] {$v_{\textit{TR}}(R)$};

  \end{scope}
  \end{tikzpicture}
  \caption{$R_{\textit{TL}}=\emptyset$, but $Q\subseteq R_{\textit{TR}}$ \\ \mbox{} \\ \mbox{}}\label{fig:RTL:e}
  \end{subfigure}
  \caption{$R\cup Q$ in all essentially different relative positions of $R,Q\in V(T)$ (up to dilation and contraction of the segments of $R\cup Q$ such that its angles are preserved eventually), where $\{R,Q\}\in E(T)$ and $v_{\textit{TL}}(R)\in Q$}\label{fig:RTL}
\end{figure}
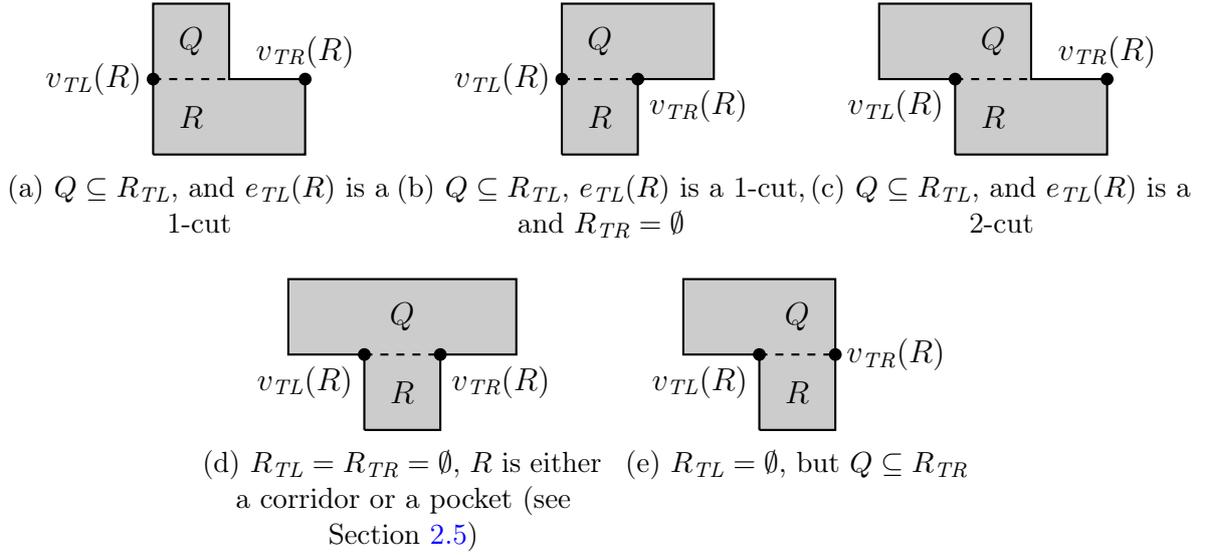

\begin{definition}\label{rectvertexns}
  Let $R,Q\in V(T)$ be arbitrary. We say that $Q$ is \textbf{adjacent} to $R$
  at $v_{\textit{TL}}(R)$, if $v_{\textit{TL}}(R)\in Q$ and $v_{\textit{TL}}(R)$ is not a vertex of the rectilinear domain $R\cup Q$, or $v_{\textit{TR}}(R)\notin Q$. Such situations are depicted on Figures~\ref{fig:RTL:a},~\ref{fig:RTL:b}, and~\ref{fig:RTL:c}.
  However, in the case of \Fref{fig:RTL:d} and~\ref{fig:RTL:e} we have $v_{\textit{TR}}(R)\in Q\not\subseteq R_{\textit{TL}}\;(=\emptyset)$.

  \medskip

  If $Q$ is adjacent to $R$ at $v_{\textit{TL}}(R)$, let $e_{\textit{TL}}(R)=\{R,Q\}$; by cutting $D$ along the dual of $e_{\textit{TL}}(R)$, i.e., $R\cap Q$, we get two rectilinear domains, and we denote the part containing $Q$ by $R_{\textit{TL}}$. If there is no such $Q$, let $e_{\textit{TL}}(R)=\emptyset$ and $R_{\textit{TL}}=\emptyset$.
  These relations can be defined analogously for top right ($R_{\textit{TR}}$, $e_{\textit{TR}}(R)$), bottom left ($R_{\textit{BL}}$, $e_{\textit{BL}}(R)$), and bottom right ($R_{\textit{BR}}$, $e_{\textit{BR}}(R)$).
\end{definition}

\begin{lemma}\label{lemma:extend}
  Let $R$ be an arbitrary rectangle such that $R_{\textit{BL}}\neq\emptyset$. Let
  $U$ be the remaining portion of the rectilinear domain, i.e., $D=R_{\textit{BL}}\dotcup U$ is
  a partition into rectilinear domains. Take an admissible partition $U=U_1\dotcup U_2$ such that $v_{\textit{BL}}(R)\in U_1$.  We can extend this to an admissible partition of $D$ where the two parts are $U_1\cup R_{\textit{BL}}$ and $U_2$.
\end{lemma}
\begin{proof}
  Let $Q_1=R\cap U_1$ and let $Q_2\in V(T)$ be the rectangle which is a subset of $R_{\textit{BL}}$ and adjacent to $R$.

  \medskip

  Observe that $U_1$ and $R_{\textit{BL}}$ only intersect on $R$'s bottom
  side, therefore their intersection is a line segment $L$ and so
  $U_1\cup R_{\textit{BL}}$ is a rectilinear domain. Trivially, $D=(U_1\cup R_{\textit{BL}})\dotcup U_2$ is partition into rectilinear domains, so only admissibility remains to be checked.

  \medskip

  Let the horizontal $R$-tree of $U_1$ and $R_{\textit{BL}}$ be $T_{U_1}$ and $T_{R_{\textit{BL}}}$, respectively.  The horizontal $R$-tree of $U_1\cup R_{\textit{BL}}$ is $T_{U_1}+ T_{R_{\textit{BL}}}+\{Q_1,Q_2\}$, except if $t(\{Q_1,Q_2\})=-4$. Either way, by referring to \Fref{rem:refine} we can use \Fref{claim:tsize} to write that
  \begin{align}\label{eq:newsize}
  n(U_1 & \cup R_{\textit{BL}})+n(U_2)-n(D)=n(U_1)+n(R_{\textit{BL}})+t(\{Q_1,Q_2\})+n(U_2)-n(D)=\nonumber \\
        & =\Big(n(U_1)+n(U_2)-n(U)\Big)+\Big(n(U)+n(R_{\textit{BL}})-n(D)\Big)+t(\{Q_1,Q_2\})=    \\
        & =\Big(n(U_1)+n(U_2)-n(U)\Big)-t(\{R,Q_2\})+t(\{Q_1,Q_2\}).\nonumber
  \end{align}
  Now it is enough to prove that $t(\{Q_1,Q_2\})\le t(\{R,Q_2\})$. If $t(\{R,Q_2\})=0$ this is trivial. The remaining case is when $t(\{R,Q_2\})=-2$. This means that $v_{\textit{BL}}(R)$ is not a vertex of $R\cup Q_2$, therefore it is not a vertex of $Q_1\cup Q_2$ either, implying that $n(Q_1\cup Q_2)<8$.
\end{proof}

\begin{lemma}\label{lemma:extendconsecutives}
  Let $R\in V(T)$ be such that $R_{\textit{BL}}\neq\emptyset$. Let $U$ be the other half of the rectilinear domain, i.e., $D=R_{\textit{BL}}\dotcup U$. If $U$ has a good cut-system $\mathcal{L}$ such that $v_{\textit{BL}}(R)\in\ker\mathcal{L}$, then $D$ also has a good cut-system.
\end{lemma}
\begin{proof}
  Let us enumerate the elements of $\mathcal{L}$ as $L_i$ where $i\in I$. Take $L_i(U_1^i,U_2^i)$ such that $v_{\textit{BL}}(R)\in U_1^i$. Using \Fref{lemma:extend} extend $L_i$ to a cut $L_i'(D_1^i,D_2^i)$ of $D$ such that $U_2^i=D_2^i$.

  \medskip

  \Fref{eq:newsize} and the statement following it implies that
  \[n(D_1^i)+n(D_2^i)=n(D)+2\implies n(U_1^i)+n(U_2^i)=n(U)+2.\]
  In other words, if $L_i$ is a 2-cut then so is $L_i'$. Therefore
  \begin{align*}
  \left\{n(U_2^i)\ |\ i\in I\right\} & \cup \left\{n(U_2^i)+2\ |\ i\in I\text{ and }L_i\text{ is a 2-cut}\right\}\subseteq \\
  & \subseteq \left\{n(D_2^i)\ |\ i\in I\right\}\cup \left\{n(D_2^i)+2\ |\ i\in I\text{ and }L_i'\text{ is a 2-cut}\right\},
  \end{align*}
  and by referring to \Fref{rem:invert}, we get that $\{L_i'\ |\ i\in I\}$ is a good cut-system of $D$.
\end{proof}

\section{Proof of \texorpdfstring{\Fref{thm:mobile}}{Theorem~3}}\label{sec:proof}

Let us recall the theorem to be proved.

\thmmobile*

We will prove \Fref{thm:mobile} by induction on the number of
vertices. For $n\le 8$ the theorem is trivial.

\medskip

For $n>8$, let $D$ be the rectilinear domain bounded by the orthogonal polygon wall of the gallery. We want to partition $D$ into smaller rectilinear domains. It is enough to prove that $D$ has a good cut. The rest of this proof is an extensive case study. Let $T$ be the horizontal $R$-tree of $D$. We need two more definitions.

\begin{itemize}
  \item A \textbf{pocket} in $T$ is a degree-1 rectangle $R$, whose only incident edge in $T$ is a 2-cut of $D$, and this cut covers the entire top or bottom side of $R$.
  \item A \textbf{corridor} in $T$ is a rectangle $R$ of degree $\ge 2$ in $T$, which has an incident edge in $T$ which is a 2-cut of $D$, and this cut covers the entire top or bottom side of $R$.
\end{itemize}

We distinguish 4 cases.
\begin{description}
  \itemsep0em
  \setlength{\baselineskip}{14pt}
  \item[\quad\Fref{case:Tisapath}.] $T$ is a path, \Fref{fig:cases}(a);
  \item[\quad\Fref{case:corridors}.] $T$ has a corridor, \Fref{fig:cases}(b);
  \item[\quad\Fref{case:pockets}.] $T$ does not have a corridor, but it has a pocket, \Fref{fig:cases}(c);
  \item[\quad\Fref{case:twisted}.] None of the previous cases apply, \Fref{fig:cases}(d).
\end{description}

\begin{figure}[H]
  \centering
  \begin{subfigure}{.45\textwidth}
  \centering
  \begin{tikzpicture}

  \begin{scope}

  \fill[color=inside] (2,1) rectangle (3.5,-1);
  \fill[color=inside] (2,0) rectangle (5,1);
  \fill[color=inside] (3,1) rectangle (4,3);
  \fill[color=inside] (4,2) rectangle (5,5);
  \fill[color=inside] (4,5) rectangle (1,4);
  \fill[color=inside] (0,3) rectangle (2,4);
  \fill[color=inside] (0.5,2) rectangle (1.5,3);

  \draw[dashed] (2,0) -- (3.5,0);
  \draw[dashed] (3,1) -- (4,1);
  \draw[dashed] (3,2) -- (4,2);
  \draw[dashed] (4,3) -- (5,3);
  \draw[dashed] (4,4) -- (5,4);
  \draw[dashed] (1,4) -- (2,4);
  \draw[dashed] (0.5,3) -- (1.5,3);

  \draw (2,0) -- (2,1) -- (3,1) -- (3,3) -- (4,3) -- (4,4) -- (2,4) -- (2,3) -- (1.5,3) -- (1.5,2) -- (0.5,2) -- (0.5,3) -- (0,3) -- (0,4) -- (1,4) -- (1,5) -- (5,5) -- (5,2) -- (4,2) -- (4,1) -- (5,1) -- (5,0) -- (3.5,0) -- (3.5,-1) -- (2,-1) -- (2,0);

  %\path (2.5,2.5) node {(a)};

  \end{scope}

  \end{tikzpicture}
  \caption{$T$ is a path.}
  \end{subfigure}%
  \begin{subfigure}{.45\textwidth}
  \centering
  \begin{tikzpicture}[xscale=0.5]

  \begin{scope}

  \fill[color=inside] (0,5) rectangle (2,6);
  \fill[color=inside] (3.5,5) rectangle (5.5,6);
  \fill[color=inside] (8,5) rectangle (9,6);
  \fill[color=inside] (0,4) rectangle (9,5);
  \fill[color=inside] (1,3) rectangle (8,4);
  \fill[color=inside] (1,1) rectangle (4,3);
  \fill[color=inside] (1,0) rectangle (2,1);
  \fill[color=inside] (3,0) rectangle (4,1);
  \fill[color=inside] (5,2) rectangle (6,3);
  \fill[color=inside] (7,2) rectangle (8,3);

  \draw[dashed] (0,5) -- (9,5);
  \draw[dashed] (1,4) -- (8,4);
  \draw[dashed] (1,3) -- (8,3);
  \draw[dashed] (1,1) -- (4,1);

  \draw (1,0) -- (1,4) -- (0,4) -- (0,6) -- (2,6) -- (2,5) -- (3.5,5) -- (3.5,6) -- (5.5,6) -- (5.5,5) -- (8,5) -- (8,6) -- (9,6) -- (9,4) -- (8,4) -- (8,2) -- (7,2) -- (7,3) -- (6,3) -- (6,2) -- (5,2) -- (5,3) -- (4,3) -- (4,0) -- (3,0) -- (3,1) -- (2,1) -- (2,0) -- (1,0);

  \path (4.5,3.5) node {corridor};

  \end{scope}

  \end{tikzpicture}
  \caption{$T$ has a corridor.}
  \end{subfigure}

  \bigskip

  \begin{subfigure}{.45\textwidth}
  \centering
  \begin{tikzpicture}

  \begin{scope}

  \fill[color=inside] (0.5,0) rectangle (1.5,1);
  \fill[color=inside] (3,0) rectangle (5,1);
  \fill[color=inside] (0,1) rectangle (5,2);
  \fill[color=inside] (1,2) rectangle (2,3);
  \fill[color=inside] (2.5,2) rectangle (3.5,3);
  \fill[color=inside] (4,2) rectangle (5,4) rectangle (0,5) rectangle (1,6);
  \fill[color=inside] (2,5) rectangle (3,6);

  \draw[dashed] (0,5) -- (3,5);
  \draw[dashed] (4,4) -- (5,4);
  \draw[dashed] (1,2) -- (2,2);
  \draw[dashed] (2.5,2) -- (3.5,2);
  \draw[dashed] (4,2) -- (5,2);
  \draw[dashed] (0.5,1) -- (1.5,1);
  \draw[dashed] (3,1) -- (5,1);

  \draw (0,1) -- (0,2) -- (1,2) -- (1,3) -- (2,3) -- (2,2) -- (2.5,2) -- (2.5,3) -- (3.5,3) -- (3.5,2) -- (4,2) -- (4,4) -- (0,4) -- (0,6) -- (1,6) -- (1,5) -- (2,5) -- (2,6) -- (3,6) -- (3,5) -- (5,5) -- (5,0) -- (3,0) -- (3,1) -- (1.5,1) -- (1.5,0) -- (0.5,0) -- (0.5,1) -- (0,1);

  \path (1,0.5) node {\footnotesize pocket} -- (1.5,2.5) node {\footnotesize pocket} -- (3,2.5) node {\footnotesize pocket} -- (2.5,5.5) node {\footnotesize pocket};

  \end{scope}

  \end{tikzpicture}
  \caption{$T$ does not have a corridor, but it has a pocket.}
  \end{subfigure}%
  \begin{subfigure}{.45\textwidth}
  \centering
  \begin{tikzpicture}

  \begin{scope}

  \fill[color=inside] (0,0) rectangle (1,0.5);
  \fill[color=inside] (2,0) rectangle (3,0.5);
  \fill[color=inside] (0,0.5) rectangle (3,1);
  \fill[color=inside] (1.5,1) rectangle (4,2);
  \fill[color=inside] (1.5,2) rectangle (2.5,3);
  \fill[color=inside] (3.5,2) rectangle (4,4);
  \fill[color=inside] (3,6) rectangle (4,5) rectangle (0,4) rectangle (1,3);
  \fill[color=inside] (0.5,3) rectangle (0,2);
  \fill[color=inside] (1,6) rectangle (2,5.5) rectangle (0,5);

  \draw[dashed] (0,0.5) -- (1,0.5);
  \draw[dashed] (2,0.5) -- (3,0.5);
  \draw[dashed] (1.5,1) -- (3,1);
  \draw[dashed] (1.5,2) -- (2.5,2);
  \draw[dashed] (3.5,2) -- (4,2);
  \draw[dashed] (3.5,4) -- (4,4);
  \draw[dashed] (3,5) -- (4,5);
  \draw[dashed] (1,5.5) -- (2,5.5);
  \draw[dashed] (0,5) -- (2,5);
  \draw[dashed] (0,4) -- (1,4);
  \draw[dashed] (0,3) -- (0.5,3);

  \draw  (3,1) -- (3,0) -- (2,0) -- (2,0.5) -- (1,0.5) -- (1,0) -- (0,0) -- (0,1) -- (1.5,1) -- (1.5,3) -- (2.5,3) -- (2.5,2) -- (3.5,2) -- (3.5,4) -- (1,4) -- (1,3) -- (0.5,3) -- (0.5,2) -- (0,2) -- (0,5.5) -- (1,5.5) -- (1,6) -- (2,6) -- (2,5) -- (3,5) -- (3,6) -- (4,6) -- (4,1) -- (3,1);

  \end{scope}

  \end{tikzpicture}
  \caption{$T$ does not have a corridor or a pocket, and it is not a path.}
  \end{subfigure}
  \caption{The 4 cases of the proof.}\label{fig:cases}
\end{figure}
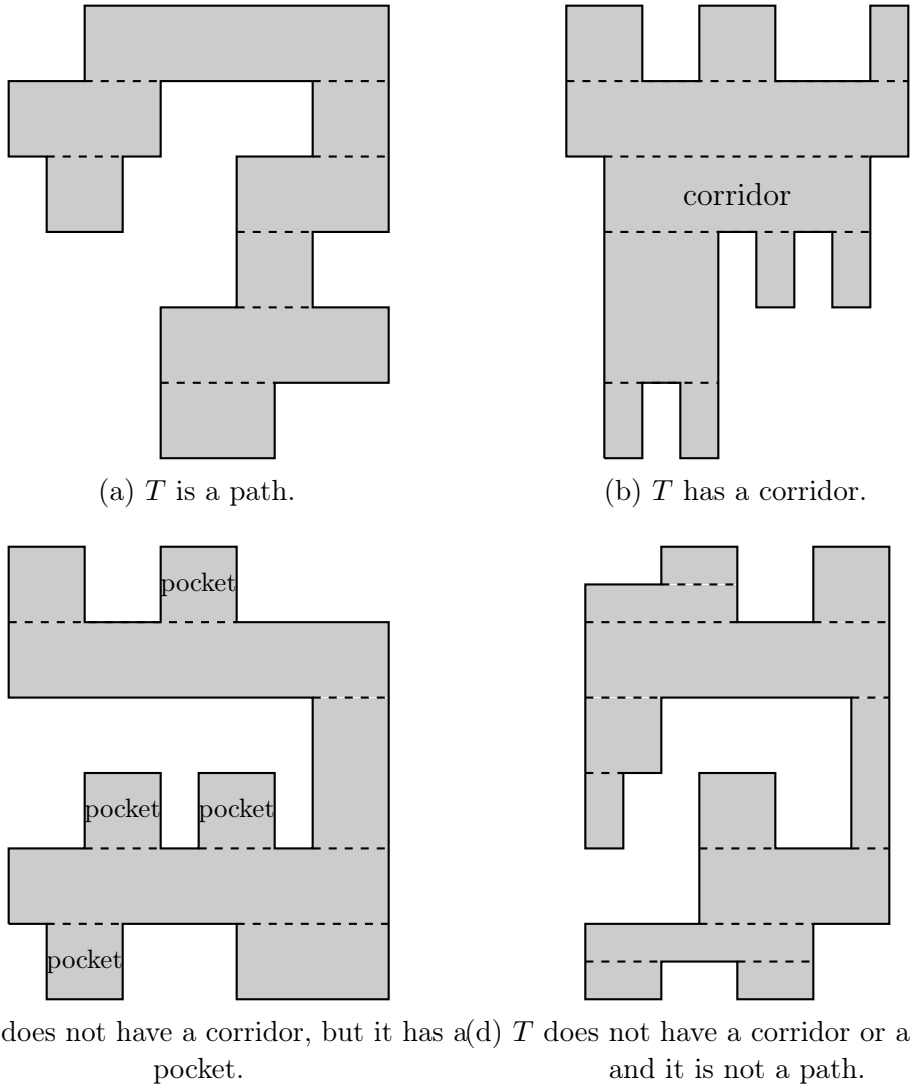

\case{\texorpdfstring{$T$}{T} is a path}\label{case:Tisapath}

\begin{claim}\label{claim:deg2doublecut}
  If an edge incident to a degree-2 vertex $R$ of $T$ is a
  2-cut of $D$, then the incident edges of $R$ form a good cut-system.
\end{claim}
\begin{proof}
  Let the two incident edges of $R$ be $e_1$ and $e_2$. Let their generated partitions be $e_1(D_1^1,D_2^1)$ and $e_2(D_1^2,D_2^2)$, such that $R\subseteq D_2^1\cap D_1^2$. Then $D_1^2=D_1^1\cup R$, so
  \begin{align*}
  n(D_1^2)=n(D_1^1)+n(R)+t(e_1)=n(D_1^1)+4.
  \end{align*}
  \Fref{def:goodcutsystem} is satisfied by $\{e_1,e_2\}$, as $\{n(D_1^1),n(D_1^2)\}\cup \{n(D_1^1)+2\}$ is a set of three consecutive even elements.
\end{proof}

\begin{claim}\label{claim:deg2path}
  If there are two rectangles $R_1$ and $R_2$ which are adjacent degree-2 vertices of $T$, then the union of the set of incident edges of $R_1$ and $R_2$ form a good cut-system.
\end{claim}
\begin{proof}
  Let the two components of $T-R_1-R_2$ be $T_1$ and $T_2$, so that
  $e_1,e_2,f\in E(T)$ joins $T_1$ and $R_1$, $R_1$ and $R_2$, $R_2$ and $T_2$, respectively.
  %	$$T_1\stackrel{e_1}{\longleftrightarrow} R_1
  %	\stackrel{f}{\longleftrightarrow} R_2
  %	\stackrel{e_2}{\longleftrightarrow} T_2.$$
  Obviously, $\cup V(T_1)\subset (\cup V(T_1))\cup R_1 \subset (\cup V(T_1))\cup R_1\cup R_2$. If one of $\{e_1,e_2,f\}$ is a 2-cut, we are done by the previous claim. Otherwise
  \begin{align*}
  n((\cup V(T_1))\cup R_1)=n(\cup V(T_1))+n(R_1)+t(e_1)                     & =n(\cup V(T_1))+2, \\
  n((\cup V(T_1))\cup R_1\cup R_2)=n(\cup V(T_1))+n(R_1)+n(R_2)+t(e_1)+t(f) & =n(\cup V(T_1))+4,
  \end{align*}
  and so $\{n(\cup V(T_1)),n((\cup V(T_1))\cup R_1),n((\cup V(T_1))\cup R_1\cup R_2)\}$ are three consecutive even elements. This concludes the proof that $\{e_1,e_2,f\}$ is a good cut-system of $D$.
\end{proof}

Suppose $T$ is a path. If $T$ is a path of length $\le 3$, such that each edge of it is a 1-cut, then $n(D)\le 8$. Also, if $T$ is path of length 2 and its only edge represents a 2-cut, then $n(D)=8$. Otherwise, either \Fref{claim:deg2doublecut}, or \Fref{claim:deg2path} can be applied to provide a good cut-system~of~$D$.

\case{\texorpdfstring{$T$}{T} has a corridor}\label{case:corridors}

Let	$e=\{R',R\}\in E(T)$ be a horizontal 2-cut such that $R'$
is a wider rectangle than $R$, and $\deg(R)\ge 2$. Let the generated partition be $e(D_1^e,D_2^e)$ such that $R'\subseteq D_1^e$. We can handle all possible cases as follows.
\begin{enumerate}%[(a)]

  \item If $n(D_1^e)\not\equiv 4,10\bmod 16$ or $n(D_2^e)\not\equiv
        4,10\bmod 16$, then $e$ is a good cut by \Fref{lemma:tech:b}.

  \item If $\deg(R)=2$, we find a good cut using
        \Fref{claim:deg2doublecut}.

  \item If $R_{\textit{BL}}=\emptyset$, then $L(D_1^L,D_2^L)$ such that $R'\subseteq D_1^L$ in \Fref{fig:corridor_b} is a good cut, since $n(D_1^L)=n(D_1^e)+4-0\equiv 8,14 \bmod 16$.
        \begin{figure}[H]
        \centering
        \begin{tikzpicture}

        \begin{scope}

        \fill[color=inside] (-2.5,1) rectangle (1.5,2);
        \fill[color=inside] (-2,0) rectangle (1,1);

        \draw (1,0) -- (1,1) -- (1.5,1);

        \draw (-2.5,1) -- (-2,1) -- (-2,0) -- (-1,0) -- (-1,-0.5);

        \draw[dashed] (-1,0) -- (-1,1) node[rotate=45]{\ScissorHollowLeft} -- (1,1);

        \path (-0.5,1.5) node {$R'$} -- (0.3,1) node[anchor=north] {$L$} -- (-0.5,0.5) node {$R$};

        \end{scope}
        \end{tikzpicture}
        \caption{$L$ is a good cut}\label{fig:corridor_b}
        \end{figure}
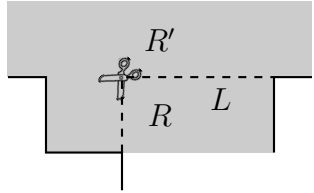

  \item If $R_{\textit{BL}}\neq\emptyset$ and $\deg(R)\ge 3$, then let us consider the following five cuts of $D$ (\Fref{fig:corridor_c}): $L_1(R_{\textit{BL}},R\cup D_1^e)$, $L_2(R_{\textit{BL}}\cup Q_1, Q_2\cup Q_3\cup D_1^e)$, $L_3(R_{\textit{BL}}\cup Q_1\cup Q_2,Q_3\cup D_1^e)$, $L_4(Q_3, R_{\textit{BL}}\cup Q_1\cup Q_2\cup D_1^e)$, and $L_5(Q_3\cup Q_2, R_{\textit{BL}}\cup Q_1\cup D_1^e)$.
        \begin{figure}[H]
        \centering
        \begin{tikzpicture}

        \begin{scope}

        \fill[color=inside] (-3,1) rectangle (2,2);
        \fill[color=inside] (-2,0) rectangle (1,1);
        \fill[color=inside] (-2,0) rectangle (-1,-1);

        \draw (1,0) -- (1,1) -- (2,1);

        \draw (-3,1) -- (-2,1) -- (-2,0);

        \draw (-1,-1) -- (-1,0) -- (0,0) -- (0,-1);

        \draw[dashed] (-2,1) -- (1,1);
        \draw[dashed] (-1,1) -- (-1,0);
        \draw[dashed] (0,1) -- (0,0);
        \draw[dashed] (-2,0) -- (-1,0);

        \path (-0.5,1.5) node {$R'$} -- (-1.5,-0.5) node {$R_{\textit{BL}}$} -- (-1.5,0.5) node{$Q_1$} -- (-0.5,0.5) node{$Q_2$} -- (0.5,0.5) node{$Q_3$};
        \end{scope}
        \end{tikzpicture}
        \caption{$\deg(R)\ge 3$ and $R_{\textit{BL}}\neq\emptyset$}\label{fig:corridor_c}
        \end{figure}
        The first pieces of these partitions have the following number of vertices (respectively).
        \begin{enumerate}%[\bfseries (1)]
        \setlength{\baselineskip}{16pt}
        \item $n(R_{\textit{BL}})$
        \item $n(R_{\textit{BL}}\cup Q_1)=n(R_{\textit{BL}})+n(Q_1)+(t(e_{\textit{BL}}(R))-2)=n(R_{\textit{BL}})+t(e_{\textit{BL}}(R))+2$
        \item $n(R_{\textit{BL}}\cup Q_1\cup Q_2)=n(R_{\textit{BL}})+n(Q_1\cup Q_2)+t(e_{\textit{BL}}(R))=n(R_{\textit{BL}})+t(e_{\textit{BL}}(R))+4$
        \item $n(Q_3)$
        \item $n(Q_3\cup Q_2)=n(Q_3)+n(Q_2)-2=n(Q_3)+2$
        \end{enumerate}
        %Suppose none of the 5 cuts above are good cuts.
        \begin{itemize}
        \item If $t(e_{\textit{BL}}(R))=0$, then $\{L_1,L_2,L_3\}$ is a good cut-system, so one
              of them is a good cut.

        \item If $t(e_{\textit{BL}}(R))=-2$, and none of the 5 cuts above are good cuts, then using \Fref{lemma:tech:b} on
              $L_2$ and $L_3$ gives ${n(R_{\textit{BL}})\equiv 4,10\bmod 16}$.
              The same argument can be used on $L_4$ and $L_5$ to conclude that
              $n(Q_3)\equiv 4,10\bmod 16$.
              However, previously we derived that
              \begin{align*}
                & n(D_2^e)\equiv 4,10\bmod 16,                                          \\
                & n(R_{\textit{BL}}\cup Q_1\cup Q_2\cup Q_3)=n(R_{\textit{BL}}\cup Q_1\cup Q_2)+n(Q_3)-2= \\
                & \qquad=n(R_{\textit{BL}})+n(Q_3)\equiv 4,10 \bmod 16.
              \end{align*}
              This is only possible if $n(R_{\textit{BL}})\equiv n(Q_3)\equiv 10\bmod 16$.
              Let $e_{\textit{BL}}(R)=\{R,S\}$.
              \begin{itemize}
              \item If $\deg(S)=2$, then let $E(T)\ni e'\neq e_{\textit{BL}}(R)$ be the other edge of $S$. Let the partition generated by it be $e'(D_1^{e'},D_2^{e'})$ such that $R'\subseteq D_1^{e'}$. We have
                    \begin{align*}
                    n(R_{\textit{BL}})   & =n(D_2^{e'})+n(S)+t(e')                   \\
                    n(D_2^{e'}) & =n(R_{\textit{BL}})-4-t(e')\equiv 6-t(e') \bmod 16
                    \end{align*}
                    Either $e'$ is a 1-cut, in which case $n(D_2^{e'})\equiv 8\bmod 16$, or $e'$ is a 2-cut, giving $n(D_2^{e'})\equiv 6\bmod 16$. In any case,
                    \Fref{lemma:tech} says that $e'$ is a good cut.

              \item If $\deg(S)=3$, then we can partition $D$ as in \Fref{fig:corridor_e}. Since $n(Q_5\cup Q_6)=4+n(Q_6)-2$, by \Fref{lemma:tech:a} the only case when neither
                    \begin{align*}
                      & L_6(Q_5\cup Q_6,Q_4\cup R\cup D_1^e)\text{, nor} \\
                      & L_7(Q_6,Q_4\cup Q_5\cup R\cup D_1^e)
                    \end{align*}
                    is a good cut of $D$ is when $n(Q_6)\equiv 4,10\bmod 16$. Also,
                    \begin{align*} 10\equiv n(Q_4\cup Q_5\cup Q_6) & =n(Q_4)+n(Q_5)+n(Q_6)-2-2\equiv \\
                                                    & \equiv n(Q_4)+n(Q_6)\bmod 16.
                    \end{align*}
                    \begin{itemize}
                    \item If $n(Q_6)\equiv 10\bmod 16$, then $n(Q_4)\equiv 0\bmod 16$, hence
                          \begin{align*}
                          n(Q_4\cup Q_5\cup Q_{11}\cup Q_{12}) & =                                 \\
                          =n(Q_4\cup Q_5)+4-4                  & =n(Q_4)+n(Q_5)-2\equiv 2\bmod 16,
                          \end{align*}
                          showing that $L_8(Q_4\cup Q_5\cup Q_{11}\cup Q_{12},Q_6\cup Q_{13}\cup Q_2\cup Q_3\cup D_1^e)$ is a good cut.
                    \item If $n(Q_6)\equiv 4\bmod 16$,
                          \begin{align*}
                          n(Q_6\cup Q_{13}\cup Q_2\cup Q_3) & =n(Q_6)+n(Q_{13}\cup Q_2)+n(Q_3)-2-2\equiv \\
                                                            & \equiv n(Q_6)+10\equiv 14\mod 16,
                          \end{align*}
                          therefore $L_9(Q_6\cup Q_{13}\cup Q_2\cup Q_3,Q_4\cup Q_5\cup Q_{11}\cup Q_{12}\cup D_1^e)$ is a good cut.
                    \end{itemize}
                    \begin{figure}[H]
                    \centering
                    \begin{tikzpicture}

                    \begin{scope}

                    \fill[color=inside] (-3,1) rectangle (4,2);
                    \fill[color=inside] (3,1) rectangle (-2,0) rectangle (1,-1);
                    \fill[color=inside] (-2,-1) rectangle (-1,-2);

                    \draw (3,0) -- (3,1) -- (4,1);

                    \draw (-3,1) -- (-2,1) -- (-2,-1);

                    \draw (1,-1) -- (1,0) -- (2,0) -- (2,-1);
                    \draw (-1,-2) -- (-1,-1) -- (0,-1) -- (0,-2);

                    \draw[dashed] (-2,1) -- (3,1);
                    \draw[dashed] (-1,-1) -- (-1,1);
                    \draw[dashed] (0,-1) -- (0,1);
                    \draw[dashed] (-2,0) -- (1,0);
                    \draw[dashed] (2,0) -- (2,1);
                    \draw[dashed] (1,0) -- (1,1);

                    \path (0.5,1.5) node {$R'$} -- (-1.5,0.5) node{$Q_{11}$} -- (-0.5,0.5) node {$Q_{12}$} -- (0.5,0.5) node {$Q_{13}$} -- (1.5,0.5) node {$Q_2$} -- (2.5,0.5) node {$Q_3$} -- (-1.5,-0.5) node {$Q_4$} -- (-0.5,-0.5) node {$Q_5$} -- (0.5,-0.5) node {$Q_6$};

                    \end{scope}
                    \end{tikzpicture}
                    \hspace{24pt}
                    \begin{tikzpicture}

                    \begin{scope}

                    \fill[color=inside] (-3,1) rectangle (4,2);
                    \fill[color=inside] (3,1) rectangle (-2,0) rectangle (1,-1);
                    \fill[color=inside] (-2,-1) rectangle (-1,-2);

                    \draw (3,0) -- (3,1) -- (4,1);

                    \draw (-3,1) -- (-2,1) -- (-2,-1);

                    \draw (1,-1) -- (1,0) -- (2,0) -- (2,-1);
                    \draw (-1,-2) -- (-1,-1) -- (0,-1) -- (0,-2);

                    \draw[dashed] (0,-1) -- (0,1) -- (3,1);

                    \path (0.12,0) node[rotate=90]{\ScissorHollowLeft$_{L{_9}}$};

                    \path (0.5,1.5) node {$R'$};

                    \end{scope}
                    \end{tikzpicture}
                    \caption{$\deg(Q)\ge 3$ and $Q_{\textit{BL}}\neq\emptyset$}\label{fig:corridor_e}
                    \end{figure}
                    In each of the above subcases we found a good cut.
              \end{itemize}
        \end{itemize}

\end{enumerate}

\case{There are no corridors in \texorpdfstring{$T$}{T}, but there is a pocket}\label{case:pockets}

Let $S$ be a (horizontal) pocket. Also, let $R$ be the neighbor of $S$ in $T$. If $\deg(R)=2$, then \Fref{claim:deg2doublecut} provides a good cut-system of $D$. However, if $\deg(R)\ge 3$, we have two cases.

\subcase{If \texorpdfstring{$R$}{R} is adjacent to at least two pockets}
Let $U$ be the union of $R$ and its adjacent pockets, and let $T_U$ be its \textbf{vertical} $R$-tree. It contains at least $4$ reflex vertices, therefore $|V(T_U)|\ge 3$.
\begin{itemize}
  \item If $V(T_U)=3$, then $|E(T_U)|=2$. Thus $t(e)=0$ for any $e\in E(T_U)$, and \Fref{claim:deg2doublecut} gives a good cut-system $\mathcal{L}$ of $U$ such that all 4 vertices of $R$ are contained in $\ker\mathcal{L}$.
  \item If $V(T_U)\ge 4$, then \Fref{claim:deg2path} gives a good cut-system $\mathcal{L}$ of $U$ such that all 4 vertices of $R$ are contained in $\ker\mathcal{L}$.
\end{itemize}
Since there are no corridors in $D$, we have
\[ D=\Big(\big((U\cup R_{\textit{BL}})\cup R_{\textit{TL}}\big)\cup R_{\textit{BR}}\Big)\cup R_{\textit{TR}}. \]
By applying \Fref{lemma:extendconsecutives} repeatedly, the good cut-system $\mathcal{L}$ can be extended to a good cut-system of $D$.

\subcase{If \texorpdfstring{$S$}{S} is the only pocket adjacent to \texorpdfstring{$R$}{R}}
We may assume without loss of generality that $S$ intersects the top side of $R$. Again, define $U$ as the union of $R$ and its adjacent pockets.

\begin{itemize}
  \item If $R_{\textit{TL}}\neq\emptyset$, let $V=U\dotcup R_{\textit{TL}}$. The cut-system $\{L_1,L_2,L_3\}$ in \Fref{fig:pockets1} is a good cut-system of $V$, and all 4 vertices of $R$ are contained in $\ker \{L_1,L_2,L_3\}$.
        By applying \Fref{lemma:extendconsecutives} repeatedly, we get a good cut-system of $D$, for example, see \Fref{fig:pockets2}.
        \begin{figure}[H]
        \centering
        \begin{subfigure}{0.55\textwidth}
        \centering
        \begin{tikzpicture}

        \begin{scope}

        \fill[color=insidelight] (0,0) rectangle (5,2);
        \fill[color=inside] (0,2) rectangle (2,3);
        \fill[color=inside] (3,2) rectangle (4,3);

        \draw[very thin] (0,2) -- (4,2);

        \draw (0,2) -- (0,0) -- (5,0) -- (5,2) -- (4,2);
        \draw (2,3) -- (2,2) -- (3,2) -- (3,3) -- (4,3) -- (4,2);

        \draw[dashed] (2,2) -- (2,0);
        \draw[dashed] (3,2) -- (3,0);
        \draw[dashed] (4,2) -- (4,0);

        %\draw[dashed] (0,2) -- (2,2);
        %\draw[dashed] (3,2) -- (4,2);
        \path (1,1) node {$R$} -- (1,2.5) node {$R_{\textit{TL}}$} -- (3.5,2.5) node {$S$};

        \path	(4-0.09,1) node[rotate=-90]{\ScissorHollowLeft$_{L_1}$};
        \path (3-0.09,1) node[rotate=-90]{\ScissorHollowLeft$_{L_2}$};
        \path (2-0.09,1) node[rotate=-90]{\ScissorHollowLeft$_{L_3}$};
        \filldraw (5,2) circle (2pt) node[anchor=west] {$v_{\textit{TR}}(R)$};

        \end{scope}
        \end{tikzpicture}
        \caption{$\{L_1,L_2,L_3\}$ is a good cut-system of $V=R\cup R_{\textit{TL}}\cup S$}\label{fig:pockets1}
        \end{subfigure}
        \begin{subfigure}{0.35\textwidth}
        \centering
        \begin{tikzpicture}

        \begin{scope}

        \fill[color=inside] (0,0) rectangle (5,2);
        \fill[color=inside] (0,2) rectangle (2,3);
        \fill[color=inside] (3,2) rectangle (4,3);
        \fill[color=inside] (3.5,0) rectangle (5,-1);
        \fill[color=inside] (0,0) rectangle (1,-1);

        \draw (0,2) -- (0,0);
        \draw (1,-1) -- (1,0) -- (3.5,0) -- (3.5,-1);
        \draw (5,0) -- (5,2) -- (4,2);
        \draw (2,3) -- (2,2) -- (3,2) -- (3,3) -- (4,3) -- (4,2);

        \draw[dashed] (2,2) -- (2,0);
        \draw[dashed] (3,2) -- (3,0);
        \draw[dashed] (4,2) -- (4,0) -- (3.5,0);

        %\draw[dashed] (0,2) -- (2,2);
        %\draw[dashed] (3,2) -- (4,2);

        \path	(4-0.1,1) node[rotate=-90]{\ScissorHollowLeft$_{L_1}$};
        \path (3-0.1,1) node[rotate=-90]{\ScissorHollowLeft$_{L_2}$};
        \path (2-0.1,1) node[rotate=-90]{\ScissorHollowLeft$_{L_3}$};

        \end{scope}
        \end{tikzpicture}
        \caption{The extended cuts}\label{fig:pockets2}
        \end{subfigure}
        \caption{$R$ has one pocket}
        \end{figure}
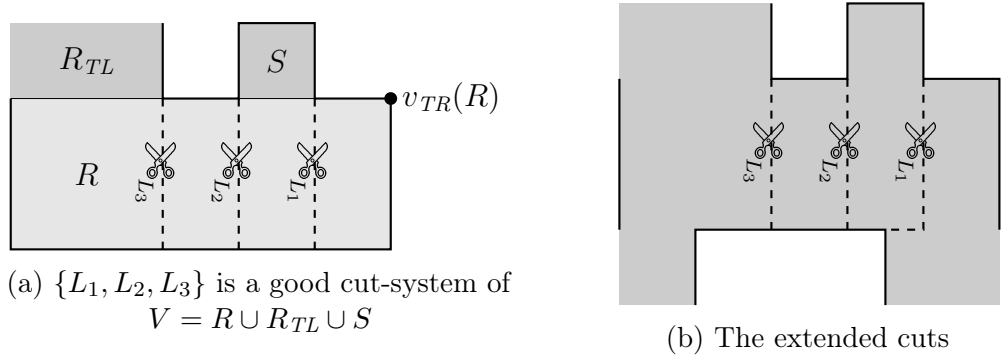
  \item If $R_{\textit{TR}}\neq\emptyset$, the case can be solved analogously to the previous case.
  \item Otherwise $R_{\textit{BL}}\neq\emptyset$ and $R_{\textit{BR}}\neq\emptyset$. Let $L_1(U_1^1,U_2^1)$ and $L_2(U_1^2,U_2^2)$ be the vertical cuts (from right to left) defined by the two reflex vertices of $U$, such that $v_{\textit{BR}}(R)\in U_1^1\subset U_1^2$. Let $V=R_{\textit{BL}}\dotcup U$.
        As before, $L_1$ and $L_2$ can be extended to cuts of $V$, say $L_1'(U_1^1,V_2^1)$, $L_2'(U_1^2,V_2^2)$. We claim that together with $e_{\textit{BL}}(R)(U,V_2^3)$, they form a good cut-system $\mathcal{L}$ of $V$. This is obvious, as $\{n(U_1^1),n(U_1^2),n(U)\}=\{4,6,8\}$.
        Since $v_{\textit{BR}}(R)\in \ker\mathcal{L}$, $D$ also has a good cut-system by \Fref{lemma:extendconsecutives}.
\end{itemize}

\case{\texorpdfstring{$T$}{T} is not a path and it does not contain either corridors or pockets}\label{case:twisted}

By the assumptions of this case, any two adjacent rectangles are adjacent at one of their vertices, so the maximum degree in $T$ is 3 or 4. We distinguish between several subcases.

\medskip

\begin{description}
  \itemsep0em
  \item[\Fref{case:toporbotcontained}.] There exists a rectangle of degree \texorpdfstring{$\ge 3$}{\textge 3} such that its top or bottom side is entirely contained in one of its neighboring rectangles;
  \item[\Fref{case:last}.] Every rectangle of degree \texorpdfstring{$\ge 3$}{\textge 3} is such that its top and bottom sides are not entirely contained in any of their neighboring rectangles;
  \vspace{-\topsep}
  \begin{description}
  \item[\Fref{case:last1}.] There exist at least two rectangles of degree \texorpdfstring{$\ge 3$}{\textge 3};
  \item[\Fref{case:last2}.] There is exactly one rectangle of degree \texorpdfstring{$\ge 3$}{\textge 3}.
  \end{description}
\end{description}

\medskip

\subcase{There exists a rectangle of degree \texorpdfstring{$\ge 3$}{\textge 3} such that its top or bottom side is entirely contained in one of its neighboring rectangles}\label{case:toporbotcontained}
Let $R$ be a rectangle and $R'$ its neighbor, such that the top or bottom side of $R$ is a subset of $\partial R'$. Moreover, choose $R$ such that if we partition $D$ by cutting $e=\{R,R'\}$, the part containing $R$ is minimal (in the set theoretic sense).

\medskip

Without loss of generality, the top side of $R$ is contained entirely by a neighboring rectangle $R'$, and $R_{\textit{TL}}=\emptyset$.
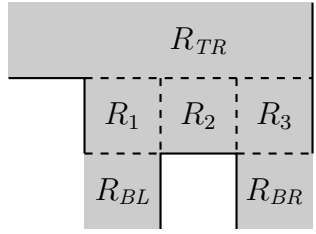
\begin{figure}
  \centering
  \begin{tikzpicture}

  \begin{scope}

  \fill[color=inside] (-3,1) rectangle (1,2);
  \fill[color=inside] (-2,0) rectangle (1,1);
  \fill[color=inside] (-2,0) rectangle (-1,-1);
  \fill[color=inside] (0,0) rectangle (1,-1);
  \fill[color=white] (1,0) rectangle (2,1);

  \draw (1,0) -- (1,2);

  \draw (-3,1) -- (-2,1) -- (-2,0);

  \draw (-1,-1) -- (-1,0) -- (0,0) -- (0,-1);

  \draw[dashed] (-2,1) -- (1,1);
  \draw[dashed] (-1,1) -- (-1,0);
  \draw[dashed] (0,1) -- (0,0);
  \draw[dashed] (-2,0) -- (-1,0);
  \draw[dashed] (0,0) -- (1,0);

  \path (-0.5,1.5) node {$R_{\textit{TR}}$} -- (-1.5,-0.5) node {$R_{\textit{BL}}$} -- (0.5,-0.5) node {$R_{\textit{BR}}$} -- (-1.5,0.5) node {$R_1$} -- (-0.5,0.5) node {$R_2$} -- (0.5,0.5) node {$R_3$};
  \end{scope}
  \end{tikzpicture}
  \caption{The top side of $R=R_1\cup R_2\cup R_3$ is contained entirely by a neighboring rectangle.}\label{fig:topside}
\end{figure}
This is pictured in \Fref{fig:topside}, where $R=R_1\cup R_2\cup R_3$. We can cut off $R_{\textit{BL}}$, $R_{\textit{BL}}\cup R_1$, and $R_{\textit{BL}}\cup R_1\cup R_2$, whose number of vertices are respectively

\medskip

\begin{enumerate}%[\qquad\bfseries (1)]
  \item $n(R_{\textit{BL}})$,
  \item $n(R_{\textit{BL}}\cup R_1)=n(R_{\textit{BL}})+n(R_1)+(t(e_{\textit{BL}}(R))-2)=n(R_{\textit{BL}})+t(e_{\textit{BL}}(R))+2$,
  \item $n(R_{\textit{BL}}\cup R_1\cup R_2)=n(R_{\textit{BL}})+n(R_1\cup R_2)+t(e_{\textit{BL}}(R))=n(R_{\textit{BL}})+t(e_{\textit{BL}}(R))+4$.
\end{enumerate}
If $t(e_{\textit{BL}}(R))=0$, then one of the 3 cuts is a good cut by \Fref{lemma:tech:a}.

\medskip

Otherwise $t(e_{\textit{BL}}(R))=-2$, thus, one of the 3 cuts is a good cut, or $n(R_{\textit{BL}})\equiv 4,10\bmod 16$. Let $S$ be the rectangle for which $e_{\textit{BL}}(R)=\{R,S\}$. Since $e_{\textit{BL}}(R)$ is a 1-cut containing the top side of $S$, we cannot have $\deg(S)=3$, as it contradicts the choice of $R$. We distinguish between two cases.
\subsubcase{$\deg(S)=1$}
Let $U=R'\cup R\cup R_{\textit{BL}}\cup R_{\textit{BR}}$, which is depicted on \Fref{fig:topside3cuts:before}. It is easy to see that $L_1(Q_1,U_2^1)$, $L_2(Q_1\cup Q_2,U_2^2)$, and $L_3(Q_1\cup Q_2\cup Q_3,U_2^3)$ in \Fref{fig:topside3cuts:after} is a good cut-system of $U$.
\begin{figure}[H]
  \centering
  \begin{subfigure}{.40\textwidth}
  \centering
  \begin{tikzpicture}

  \begin{scope}

  \fill[color=inside] (-3,1) rectangle (1,3);
  \fill[color=inside] (-2,0) rectangle (1,1);
  \fill[color=inside] (-2,0) rectangle (-1,-1);
  \fill[color=inside] (0,0) rectangle (1,-1);

  \draw (1,0) -- (1,3) -- (-3,3) -- (-3,1) -- (-2,1) -- (-2,-1) -- (-1,-1) -- (-1,0) -- (0,0) -- (0,-1);

  \draw[dashed] (-2,1) -- (1,1);
  \draw[dashed] (-1,0) -- (-2,0);
  \draw[dashed] (0,0) -- (1,0);

  \path (-1.5,-0.5) node {$S$} -- (0.5,-0.5) node {$R_{\textit{BR}}$} -- (-0.5,0.5) node {$R$} -- (-0.5,2.0) node {$R'$};
  \end{scope}

  \end{tikzpicture}
  \caption{}\label{fig:topside3cuts:before}
  \end{subfigure}%
  \begin{subfigure}{.40\textwidth}
  \centering
  \begin{tikzpicture}

  \begin{scope}

  \fill[color=inside] (-3,1) rectangle (1,3);
  \fill[color=inside] (-2,0) rectangle (1,1);
  \fill[color=inside] (-2,0) rectangle (-1,-1);
  \fill[color=inside] (0,0) rectangle (1,-1);

  \draw (1,0) -- (1,3) -- (-3,3) -- (-3,1) -- (-2,1) -- (-2,-1) -- (-1,-1) -- (-1,0) -- (0,0) -- (0,-1);

  \draw[dashed] (-2,3) -- (-2,1);
  \draw[dashed] (-1,3) -- (-1,0);
  \draw[dashed] (0,3) -- (0,0) -- (1,0);

  \path (0.5,-0.5) node {$R_{\textit{BR}}$} -- (-1.5,1) node {$Q_2$} -- (-0.5,1) node {$Q_3$} -- (0.5,1) node {$Q_4$} -- (-2.5,2.5) node {$Q_1$};

  \path (-2+0.11,2) node[rotate=+90]{\ScissorHollowLeft$_{L_1}$};
  \path (-1+0.11,2) node[rotate=+90]{\ScissorHollowLeft$_{L_2}$};
  \path (0+0.11,2) node[rotate=+90]{\ScissorHollowLeft$_{L_3}$};

  \end{scope}
  \end{tikzpicture}
  \caption{}\label{fig:topside3cuts:after}
  \end{subfigure}
  \caption{The rectilinear domain $U$ is shown in~(a). The cuts $L_1$, $L_2$, $L_3$, shown in~(b), form a good-cut system of $U$.}\label{fig:topside3cuts}
\end{figure}
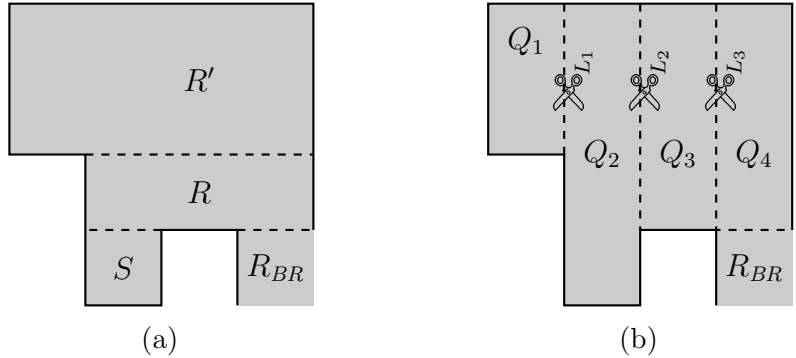
As all 4 vertices of $S$ are contained in $\ker\{L_1,L_2,L_3\}$, we can extend this good cut-system to $D$ by reattaching $S_{\textit{TL}}$, $S_{\textit{BL}}$, $S_{\textit{TR}}$ (if non-empty) via \Fref{lemma:extendconsecutives}. Therefore, $D$ has a good cut.

\subsubcase{$\deg(S)=2$}
Let $f$ be the edge of $S$ which is different from $e_{\textit{BL}}(R)=e_{\textit{TL}}(S)$. Let the partition generated by it be $f(D_1^f,D_2^f)$, where $S\subseteq D_2^f$. We have $n(D_1^f)=n(R_{\textit{BL}})-n(S)-t(f)$.
\begin{itemize}
  \item If $t(f)=-2$, then $n(D_1^f)\equiv 2,8\bmod 16$, so $f$ is a good cut by \Fref{lemma:tech:a}.
  \item If $t(f)=0$, then $n(D_1^f)\equiv 0,6\bmod 16$, so $f$ is a good cut by \Fref{lemma:tech:b}.
\end{itemize}

\subcase{Every rectangle of degree \texorpdfstring{$\ge 3$}{\textge 3} is such that its top and bottom sides are not entirely contained in any of their neighboring rectangles}\label{case:last}

Let $R$ be a rectangle of degree $\ge 3$ and $e=\{R,S\}$ be one of its edges. Let the partition generated by $e$ be $e(D_1^e,D_2^e)$, where $R\subset D_1^e$ and $S\subseteq D_2^e$. If $e$ is a 1-cut, then by the assumptions of this case $\deg(S)\le 2$.
\begin{itemize}
  \item If $\deg(S)=1$ and $t(e)=-2$, then $n(D_2^e)+t(e)=2$.
  \item If $\deg(S)=1$ and $t(e)=0$, then $n(D_2^e)+t(e)=4$.
  \item If $\deg(S)=2$ and one of the edges of $S$ is a 0-cut, then $D$ has a good cut by \Fref{claim:deg2doublecut}.
  \item If $\deg(S)=2$ and both edges of $S$, $e$ and (say) $f$ are 1-cuts:
        Let the partition generated by $f$ be $D=D_1^f\dotcup D_2^f$, such that $S\in D_1^f$. Then
        $n(D_2^e)=n(D_2^f)+n(S)+t(f)=n(D_2^f)+2$. Either one of $e$ and $f$ is a good cut, or by \Fref{lemma:tech:a} we have $n(D_2^f)\equiv 4,10\mod 16$. In other words,
        $n(D_2^e)+t(e)\equiv 4,10\mod 16$. Similarly, $n(D_1^f)=n(D_1^e)+4-2=n(D_1^e)+2$, so $n(D_1^e)\equiv 4,10\mod 16$.
  \item If $\deg(S)\ge 3$, then $t(e)=0$. Either $e$ is a good cut, or by \Fref{lemma:tech:b} we have $n(D_2^e)+t(e)\equiv 4,10\mod 16$. \Fref{lemma:tech:b} also implies $n(D_1^e)\equiv 4,10\mod 16$.
\end{itemize}
From now on, we assume that none of the edges of the neighbors of a degree $\ge 3$ rectangle represent a good cut, so in particular, we have
\[ n(D_2^e)+t(e)\equiv 2,4,\text{ or }10\mod 16.\]

In addition to the simple analysis we have just conducted, we deduce an easy claim to be used in the following subcases.

\begin{claim}\label{claim:opposite}
  Let $R\in V(T)$ be of degree $\ge 3$ and suppose both $R_{\textit{BR}}\neq\emptyset$ and $R_{\textit{TR}}\neq\emptyset$. Then $D$ has two admissible cuts $L_1$ and $L_2$ such that they form a good cut-system or
  \begin{center}
  \begin{tabular}{rl}
  \textbf{(i)}  & \text{one of the parts generated by $L_1$ has size }                                                                                                            \\
                & $\Big(n(R_{\textit{BR}})+t(e_{\textit{BR}}(R))\Big)+\Big(n(R_{\textit{TR}})+t(e_{\textit{TR}}(R))\Big)+2$,\makeatletter\def\@currentlabel{\arabic{theorem}(i)}\label{claim:opposite:2}\makeatother  \\
  \mbox{}       & \textbf{\quad and }                                                                                                                                             \\
  \textbf{(ii)} & \text{one of the parts generated by $L_2$ has size }                                                                                                            \\
                & $\Big(n(R_{\textit{BR}})+t(e_{\textit{BR}}(R))\Big)+\Big(n(R_{\textit{TR}})+t(e_{\textit{TR}}(R))\Big)+4$.\makeatletter\def\@currentlabel{\arabic{theorem}(ii)}\label{claim:opposite:4}\makeatother
  \end{tabular}
  \end{center}
\end{claim}
\begin{proof}
  Let $U=R\cup R_{\textit{BL}}\cup R_{\textit{BR}}$. Let $L_1(U_1^1,U_2^1)$ and $L_2(U_1^2,U_2^2)$ be the vertical cuts of $U$
  defined by the two reflex vertices of $U$ that are on the boundary of $R$, such that $v_{\textit{BR}}(R)\in U_1^1\subset U_1^2$.
  By \Fref{lemma:extend}, $L_1$ and $L_2$ can be extended to cuts of $V=R\cup R_{\textit{BL}}\cup R_{\textit{BR}}\cup R_{\textit{TR}}$, say $L_1'(V_1^1,U_2^1)$, $L_2'(V_1^2,U_2^2)$.
  If one of $L_1'$ or $L_2'$ is a 2-cut, then similarly to \Fref{claim:deg2doublecut}, one can verify they form a good cut-system of $V$, which we can extend to $D$. Otherwise
  \begin{align*}
  n(V_1^1) & =n(R\cap U_1^1)+n(R_{\textit{BR}})+n(R_{\textit{TR}})+(t(e_{\textit{BR}}(R))-2)+t(e_{\textit{TR}}(R))=    \\
           & =\Big(n(R_{\textit{BR}})+t(e_{\textit{BR}}(R))\Big)+\Big(n(R_{\textit{TR}})+t(e_{\textit{TR}}(R))\Big)+2, \\
  n(V_1^2) & =n(R\cap U_1^2)+n(R_{\textit{BR}})+n(R_{\textit{TR}})+t(e_{\textit{BR}}(R))+t(e_{\textit{TR}}(R))=        \\
           & =\Big(n(R_{\textit{BR}})+t(e_{\textit{BR}}(R))\Big)+\Big(n(R_{\textit{TR}})+t(e_{\textit{TR}}(R))\Big)+4.
  \end{align*}
  Lastly, we extend $L_1'$ and $L_2'$ to $D$ by reattaching $R_{\textit{TL}}$ using \Fref{lemma:extend}. This step does not affect the parts $V_1^1$ and $V_1^2$, so we are done.
\end{proof}

\subsubcase{There exist at least two rectangles of degree \texorpdfstring{$\ge 3$}{\textge 3}}\label{case:last1}
In the subgraph $T'$ of $T$ which is the union of all paths of $T$ which connect two degree $\ge 3$ rectangles, let $R$ be a leaf and $e=\{R,S\}$ its edge in the subgraph. As defined in the beginning of \Fref{case:last}, the set of incident edges of $R$ (in $T$) is $\{e_i\ |\ 1\le i\le \deg(R)\}$, and without loss of generality we may suppose that $e=e_{\deg(R)}$.
The analysis also implies that for all $1\le i\le \deg(R)-1$, we have $n(D_2^{e_i})+t(e_i)=2,4$.

\medskip

By the assumptions of this case $\deg(S)\ge 2$, therefore $n(D_1^e)\equiv 4\text{ or }10\mod 16$. If $\deg(S)\ge 3$, let $Q=S$. Otherwise $\deg(S)=2$, and let $Q$ be the second neighbor of $R$ in $T'$. The degree of $Q$ cannot be 1 by its choice. If $\deg(Q)=2$, then we find a good cut using \Fref{claim:deg2path}. In any case, we may suppose from now on that $\deg(Q)\ge 3$.

\medskip

Let $\{f_i\ |\ 1\le i\le \deg(Q)\}$ be the set of incident edges of $Q$ such that they generate the partitions $D=D_1^{f_i}\dotcup D_2^{f_i}$ where $R\subset D_1^{f_i}$ and $Q\subset D_2^{f_1}$. We have
\begin{align*}
  n(D_1^{e})=n(R)+\sum_{i=1}^{\deg(R)-1}\Big(n(D_2^{e_i})+t(e_i)\Big) \in & 4+\{2,4\}+\{2,4\}+\{0,2,4\}= \\
                                                                          & =\{8,10,12,14,16\},
\end{align*}
so the only possibility is $n(D_1^e)=10$.
\begin{itemize}
  \item If $\deg(S)\ge 3$, $e$ is a 2-cut (by the assumption of \Fref{case:last}), so by \Fref{lemma:tech:c}, either $e$ is a good cut or $n(D)\equiv 14\bmod 16$. Since $Q=S$ and $e=f_1$, we have
        \begin{align*}
        &n(D_1^{f_1})+t(f_1) =n(D_1^e)+t(e)=10, \\
        &n(D_2^{f_1})=n(D)-n(D_1^{f_1})-t(f_1)\equiv 14-10\equiv 4\bmod 16.
        \end{align*}
  \item If $\deg(S)=2$, either $e$ is a 1-cut or we find a good cut using \Fref{claim:deg2doublecut}. Also, $f_1=\{S,Q\}$ is a 1-cut too (otherwise apply \Fref{claim:deg2doublecut}), so
  \[ n(D_1^{f_1})+t(f_1)=n(D_1^e)+n(S)+t(e)+t(f_1)=10. \]
  By \Fref{lemma:tech:c}, either $f_1$ is a good cut (as $n(D_1^{f_1})=12$) or $n(D)\equiv 14\bmod 16$. Thus
  \[ n(D_2^{f_1})=n(D)-n(D_1^{f_1})-t(f_1)\equiv 14-12+2\equiv 4\bmod 16. \]
\end{itemize}
We have
\begin{align*}
  n(D)&=n(Q)+\Big(n(D_1^{f_1})+t(f_1)\Big)+\sum_{i=2}^{\deg(Q)}\Big(n(D_2^{f_i})+t(f_i)\Big)\in \\ &\in 14+\{2,4,10\}+\{2,4,10\}+\{0,2,4,10\} \mod 16.
\end{align*}
The only way we can get $14\bmod 16$ on the right-hand side is when $\deg(Q)=4$ and out of
\[ n(Q_{\textit{BL}}),n(Q_{\textit{BR}}),n(Q_{\textit{TL}}),n(Q_{\textit{TR}})\mod 16, \]
one is $2$, another is $4$, and two are $10 \bmod 16$.

\medskip

The last step in this case is to apply \Fref{claim:opposite} to $Q$. If it does not give a good cut-system, then it gives an admissible cut where one of the parts has size congruent to $2+10+2=14$ or $2+4+2=8$ modulo~16, therefore we find a good cut anyway.

\subsubcase{There is exactly one rectangle of degree \texorpdfstring{$\ge 3$}{\textge 3}}\label{case:last2}
Let $R$ be the rectangle of degree $\ge 3$ in $T$, and let $\{e_i \ |\ 1\le i\le \deg(R)\}$ be the edges of $R$, which generate the partitions $D=D_1^{e_i}\dotcup D_2^{e_i}$ where $R\subset D_1^{e_i}$. Then $D_2^{e_i}$ is path for all $i$.
If either \Fref{claim:deg2doublecut} or \Fref{claim:deg2path} can be applied, $D$ has a good cut.
The remaining possibilities can be categorized into 3 types:

\[
  \begin{array}{lllll}
  \text{\bfseries Type 1: } & t(e_i)=-2, & n(D_2^{e_i})=4,       & \text{ and } & n(D_2^{e_i})+t(e_i)=2; \\
  \text{\bfseries Type 2: } & t(e_i)=-2, & n(D_2^{e_i})=4+4-2=6, & \text{ and } & n(D_2^{e_i})+t(e_i)=4; \\
  \text{\bfseries Type 3: } & t(e_i)=0,  & n(D_2^{e_i})=4,       & \text{ and } & n(D_2^{e_i})+t(e_i)=4.
  \end{array}
\]

Without loss of generality $R_{\textit{BR}}\neq\emptyset$ and $R_{\textit{TR}}\neq\emptyset$. We will now use \Fref{claim:opposite}. If it gives a good cut-system, we are done. Otherwise
\begin{itemize}
  \item If exactly one of $e_{\textit{BR}}(R)$ and $e_{\textit{TR}}(R)$ is of {\bfseries type~1}, apply \Fref{claim:opposite:2}: it gives an admissible cut which cuts off a rectilinear domain of size $2+4+2=8$, so $D$ has a good cut.
  \item If both $e_{\textit{BR}}(R)$ and $e_{\textit{TR}}(R)$ are of {\bfseries type~1}, apply \Fref{claim:opposite:4}: it gives an admissible cut which cuts off a rectilinear domain of size $2+2+4=8$, so $D$ has a good cut.
  \item If none of $e_{\textit{BR}}(R)$ and $e_{\textit{TR}}(R)$ are of {\bfseries type~1} and $n(D)\not\equiv 14 \pmod{16}$, apply  \Fref{claim:opposite:2}: it gives an admissible cut which cuts off a rectilinear domain of size $4+4+4=12$, which is a good cut by \Fref{lemma:tech:c}.
\end{itemize}
Now we only need to deal with the case where $n(D)=14$ and neither $e_{\textit{BR}}(R)$ nor $e_{\textit{TR}}(R)$ is of {\bfseries type~1}.
%The proof \Fref{thm:mobile} is completed in the following paragraphs.

\medskip

If $R$ still has two edges of {\bfseries type~1}, again \Fref{claim:opposite:2} gives a good cut of $D$. If $R$ has at most one edge of {\bfseries type~1}, we have
\[ 14=n(D)=n(R)+\sum_{i=1}^{\deg(R)}\Big(n(D_2^{e_i})+t(e_i)\Big)\in 4+\{2,4\}+\{4\}+\{4\}+\{0,4\}=\{14,16,18,20\}, \]
and the only way we can get $14$ on the right hand is when $\deg(R)=3$ and both $e_{\textit{BR}}(R)$ and $e_{\textit{TR}}(R)$ are of {\bfseries type~2~or~3} while the third incident edge of $R$ is of {\bfseries type~1}. We may  assume without loss of generality that the cut represented by $e_{\textit{TR}}(R)$ is longer than the cut represented by $e_{\textit{BR}}(R)$.
\begin{itemize}
  \item If $D$ is vertically convex, its vertical $R$-tree is a path, so it has a good cut as deduced in \Fref{case:Tisapath}.
  \item If $D$ is not vertically convex, but $e_{\textit{TR}}(R)$ is a \textbf{type~2} edge of $R$, such that the only horizontal cut of $R_{\textit{TR}}$ is shorter than the cut represented by $e_{\textit{TR}}(R)$,
        then $D'=D-R_{\textit{BR}}$ is vertically convex, and has $(10 - 4) / 2 = 3$ reflex vertices.
        By \Fref{claim:deg2doublecut} or \Fref{claim:deg2path}, $D'$ has a good cut-system such that its kernel contains $v_{\textit{BR}}(R)$, since its $x$-coordinate is maximal in $D'$. \Fref{lemma:extendconsecutives} states that $D$ also has a good cut-system.
  \item Otherwise we find that the top right part of $D$ looks like to one of the cases in \Fref{fig:44}. It is easy to see that in all three pictures $L$ is an admissible cut which generates two rectilinear domains of 8 vertices.
        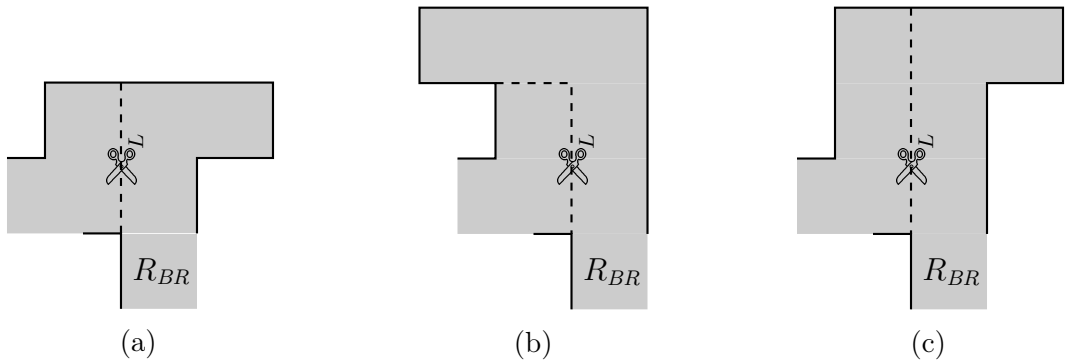
\begin{figure}[ht]
        \centering
        \begin{subfigure}{.33\textwidth}
        \centering
        \begin{tikzpicture}

        \begin{scope}

        \fill[color=white] (1,2) rectangle ++(1,1);
        \fill[color=inside] (-1,1) rectangle (2,2);
        \fill[color=inside] (-1.5,1) rectangle (1,0) rectangle (0,-1);

        \draw (1,0) -- (1,1) -- (2,1) -- (2,2) -- (-1,2) -- (-1,1) -- (-1.5,1);

        \draw (-0.5,0) -- (0,0) -- (0,-1);

        \draw[dashed] (0,2) -- (0,0);

        \path (0,-0.5) node[right] {$R_{\textit{BR}}$} ;%-- (-1, 0.5) node {$R$};

        \path (0.05,1) node[rotate=90]{\ScissorHollowLeft$_L$};

        \end{scope}
        \end{tikzpicture}
        \caption{}\label{fig:44:a}
        \end{subfigure}%
        \begin{subfigure}{.33\textwidth}
        \centering
        \begin{tikzpicture}

        \begin{scope}

        \fill[color=inside] (-2,3) rectangle (1,2) rectangle (-1,1);
        \fill[color=inside] (-1.5,1) rectangle (1,0) rectangle (0,-1);

        \draw (1,0) -- (1,3) -- (-2,3) -- (-2,2) -- (-1,2) -- (-1,1) -- (-1.5,1);

        \draw (-0.5,0) -- (0,0) -- (0,-1);

        \draw[dashed] (-1,2) -- (0,2) -- (0,0);

        \path (0,-0.5) node[right] {$R_{\textit{BR}}$} ;%-- (-1, 0.5) node {$R$};

        \path (0.05,1) node[rotate=90]{\ScissorHollowLeft$_L$};

        \end{scope}
        \end{tikzpicture}
        \caption{}\label{fig:44:b}
        \end{subfigure}%
        \begin{subfigure}{.33\textwidth}
        \centering
        \begin{tikzpicture}

        \begin{scope}

        \fill[color=inside] (2,3) rectangle (-1,2) rectangle (1,1);
        \fill[color=inside] (-1.5,1) rectangle (1,0) rectangle (0,-1);

        \draw (1,0) -- (1,2) -- (2,2) -- (2,3) -- (-1,3) -- (-1,1) -- (-1.5,1);

        \draw (-0.5,0) -- (0,0) -- (0,-1);

        \draw[dashed] (0,3) -- (0,0);

        \path (0,-0.5) node[right] {$R_{\textit{BR}}$} ;%-- (-1, 0.5) node {$R$};

        \path (0.05,1) node[rotate=90]{\ScissorHollowLeft$_L$};

        \end{scope}
        \end{tikzpicture}
        \caption{}\label{fig:44:c}
        \end{subfigure}%

        \caption{The last 3 cases of the proof. Since $e_{\textit{BR}}(R)$ is a \textbf{type 2 or 3} edge, cutting the rectilinear domain at $L$ creates two rectilinear domains of 8-vertices.}\label{fig:44}
        \end{figure}
\end{itemize}

The proof of \Fref{thm:mobile} is complete. To complement the formal proof, we now demonstrate the algorithm on \Fref{fig:partition}.

\begin{figure}[ht]
  \centering
  \begin{tikzpicture}

  \begin{scope}[shift={(-6,4)}] % figure b

  \fill[color=inside] (0,5) rectangle (2,6);
  \fill[color=inside] (3.5,5) rectangle (5.5,6);
  %\fill[color=inside] (8,5) rectangle (9,6);
  \fill[color=inside] (0,4) rectangle (11,5);
  \fill[color=inside] (1,3) rectangle (10,4);
  \fill[color=inside] (1,1) rectangle (4,3);
  \fill[color=inside] (1,0) rectangle (2,1);
  \fill[color=inside] (3,0) rectangle (4,1);
  \fill[color=inside] (5,2) rectangle (6,3);
  \fill[color=inside] (7,2) rectangle (10,3);

  \fill[color=inside] (3.5,-1) rectangle (4,0);

  \draw (1,0) -- (1,4) -- (0,4) -- (0,6) -- (2,6) -- (2,5) -- (3.5,5) -- (3.5,6) -- (5.5,6) -- (5.5,5) -- (11,5) -- (11,4) -- (10,4) -- (10,2);
  \draw (7,2) -- (7,3) -- (6,3) -- (6,2) -- (5,2) -- (5,3) -- (4,3) -- (4,-1) -- (3.5,-1) -- (3.5,0) -- (3,0) -- (3,1) -- (2,1) -- (2,0) -- (1,0);

  \end{scope}

  \begin{scope}[shift={(2,0)}] % figure d

  \fill[color=inside] (0,0) rectangle (1,0.5);
  \fill[color=inside] (2,0) rectangle (3,0.5);
  \fill[color=inside] (0,0.5) rectangle (3,1);
  \fill[color=inside] (1.5,1) rectangle (4,2);
  \fill[color=inside] (1.5,2) rectangle (2.5,3);
  \fill[color=inside] (3.5,2) rectangle (4,4);
  \fill[color=inside] (3,6) rectangle (4,5) rectangle (0,4) rectangle (1,3);
  \fill[color=inside] (-0.5,3) rectangle (1,2);
  \fill[color=inside] (2,6) rectangle (-1,5.5);
  \fill[color=inside] (2,5.5) rectangle (0,5);

  \draw  (3,1) -- (3,0) -- (2,0) -- (2,0.5) -- (1,0.5) -- (1,0) -- (0,0) -- (0,1) -- (1.5,1) -- (1.5,3) -- (2.5,3) -- (2.5,2) -- (3.5,2) -- (3.5,4) -- (1,4) -- (1,2) -- (-0.5,2) -- (-0.5,3) -- (0,3) -- (0,5.5) -- (-1,5.5) -- (-1,6);
  \draw (2,6) -- (2,5) -- (3,5) -- (3,6) -- (4,6) -- (4,1) -- (3,1);

  \end{scope}

  % n=52

  \begin{scope} % partition

  % corridor
  \draw[dashed] (-2,7) -- (-2,8) -- (4,8);
  \node[draw,thin,circle,inner sep=1pt,anchor=west,outer sep=3pt] at (-2,7.5) {1};
  \node[draw,thin,circle,inner sep=1pt,anchor=south,outer sep=3pt] at (2,8) {1};

  % we created a corridor
  \draw[dotted] (-4,5) -- (-4,8) -- (-2,8);
  \node[draw,thin,circle,inner sep=1pt,anchor=west,outer sep=3pt] at (-4,6.5) {2};
  \node[draw,thin,circle,inner sep=1pt,anchor=north,outer sep=3pt] at (-3,8) {2};

  % 2 pockets
  \draw[dashed] (-4,8) -- (-4,9);
  \node[draw,thin,circle,inner sep=1pt,anchor=east,outer sep=3pt] at (-4,8.5) {3};

  % 1 pocket
  \draw[dotted] (1,7) -- (1,8);
  \node[draw,thin,circle,inner sep=1pt,anchor=west,outer sep=3pt] at (1,7.5) {4};

  % no corridors or pockets anymore, n=28
  \draw[dotted] (5.5,4) -- (6,4);
  \node[draw,thin,circle,inner sep=1pt,anchor=south,outer sep=3pt] at (5.75,4) {5};

  \draw[dashed] (3.5,1) -- (5,1);
  \node[draw,thin,circle,inner sep=1pt,anchor=south,outer sep=3pt] at (4.25,1) {6};

  % 14-vertex piece: 2 type 2, 1 type 1, one rectangle
  \draw[dotted] (3,4)--(3,8);
  \node[draw,thin,circle,inner sep=1pt,anchor=west,outer sep=3pt] at (3,6) {7};

  \end{scope}

  \end{tikzpicture}
  \caption{The output of the algorithm on a rectilinear domain of $52$ vertices.}\label{fig:partition}
\end{figure}
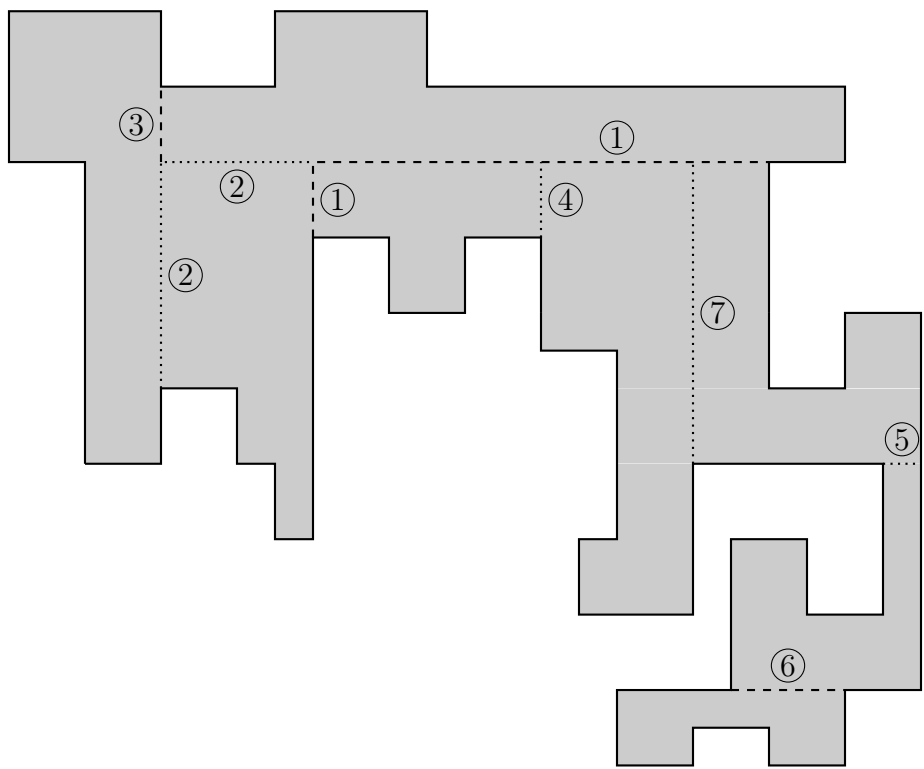

First, we resolve a corridor via the $L$-cut \circled{1}, which creates two pieces of 20 and 34 ($\equiv 2\bmod 16$) vertices. Because of this cut, a new corridor emerges in the 20-vertex piece, so we cut the rectilinear domain at \circled{2}, cutting off a piece of 8 vertices. The other piece of 14 vertices containing two pockets is further divided by \circled{3} into two pieces of 6 and 8 vertices. Another pocket is dealt with by cut \circled{4}, which divides the rectilinear domain into 8- and 28-vertex pieces. To the larger piece, \Fref{case:last} applies, and we find cut \circled{5}, which produces 16- and 14-vertex pieces. The 16-vertex piece is cut into two 8-vertex pieces by \circled{6}. Lastly, \Fref{fig:44:c} of \Fref{case:last2} applies to the 14-vertex piece, so cut \circled{7} divides it into two 8-vertex pieces.

\medskip

We got lucky with cuts \circled{2} and \circled{6} in the sense that they both satisfy the inequality in (\ref{eq:n1n2}) strictly. Hence, only 8 pieces are needed to partition the rectilinear domain of \Fref{fig:partition} into rectilinear domains of $\le 8$-vertex pieces, instead of the extremal upper bound of $(3\cdot 52+4)/16=10$.

\chapter{Mobile vs.~point guards}\label{chap:versus}

\section{Introduction}
% The number of mobile and point guards required to control the interior of a general or an orthogonal polygon has been well-studied as a function of the number of vertices of the polygon. \citeauthor{MR699771} in 1980~\cite{MR699771}, and a few years later \citet{MR844048}, and \citet{ORourke} proved that $\lfloor\frac{n}{4}\rfloor$ point guards are sufficient and sometimes necessary to cover the interior of an orthogonal polygon of $n$ vertices. Aggarwal proved in his thesis~\cite{Ag} that any $n$-vertex orthogonal polygon can be covered by at most $\lfloor\frac{3n+4}{16}\rfloor$ mobile guards, and a strengthening of this result has been shown in~\cite{GyM2016}. These estimates are also shown to be sharp as extremal results.
% These theorems imply that --- from an extremal point of view --- only $\frac43$'s as many point guards as mobile guards are needed. However, it was not  much studied if we can say something about the ratio of these optima.
% % In \cite[Table~3.1]{ORourke} it is shown that as a function of the number of vertices of an art gallery which is bounded by a general or an orthogonal polygon, at most $\frac43$'s as many point guards as mobile guards are needed.
%
% \medskip

The main goal of this chapter is to explore the ratio between the numbers of mobile guards and points guards required to control an orthogonal polygon  without holes. At first, this appears to be hopeless, as \Fref{fig:comb} shows a comb, which can be guarded by one mobile guard (whose patrol is shown by a dotted horizontal line). However, to cover the comb using point guards, one has to be placed for each tooth, so ten point guards are needed (marked by solid disks). Combs with arbitrarily high number of teeth clearly demonstrate that the minimum number of points guards required to control an orthogonal polygon cannot be bounded by the minimum size of a mobile guard system covering the comb.
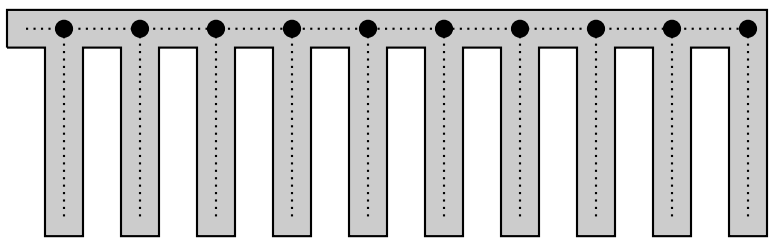
\begin{figure}[h]
	\centering
	\begin{tikzpicture}
	\def\n {10}
	\def\l {5}

	\begin{scope}

	\filldraw [fill=inside][turtle={home,forward,right}]
	\foreach \i in {1,...,\n}
		{[turtle={fd,fd}]}

	[turtle=right,fd]
	\foreach \i in {1,...,\n}
	{
		\foreach \j in {1,...,\l}
			{[turtle=fd]}
		[turtle=rt,fd,rt]
		\foreach \j in {1,...,\l}
			{[turtle=fd]}
		[turtle=lt,fd,lt]
	};

	\end{scope}

	\begin{scope}[\mgsize]
		\draw[\hpat] (\dist*0.5,\dist*0.5) -- (\dist*2*\n-\dist*0.5,\dist*0.5);

		\foreach \x in {1,...,\n}
		{
			\draw[\vpat] (\dist*2*\x-\dist*0.5,\dist*0.5) -- (\dist*2*\x-\dist*0.5,-\dist*\l+\dist*0.5);
			\filldraw[black] (\dist*2*\x-\dist*0.5,\dist*0.5) circle ( \pgsize );
		}
	\end{scope}
\end{tikzpicture}
	\caption{A comb with 10 teeth}\label{fig:comb}
\end{figure}

\medskip

%Notice, however, that the minimum size of a vertical mobile guard system of the comb in \Fref{fig:comb} is the same as the minimum size of a point guard system of it.

\citet{KM11} defined and studied the notion of ``horizontal sliding cameras''. This notion is identical to what we call horizontal mobile $r$-guard (a horizontal mobile guard with rectangular vision). The main result of this chapter, \Fref{thm:main}, shows that a constant factor times the sum of the minimum sizes of a horizontal and a vertical mobile $r$-guard system can be used to estimate the minimum size of a point $r$-guard system. It is surprising to have such a result after encountering the comb, but it is similarly unexpected that even this ratio cannot be bounded if the region may contain holes.

\medskip

Take, for example, \Fref{fig:holes}, which generally contains $3k^2+4k+1$ square holes (in the figure $k=4$). The regions covered by line of sight vision by the black dots are pairwise disjoint, because the distance between adjacent square holes is less than half of the length of a square hole's side.
Therefore, no two of the black dots can be covered by one point guard,
%(even if it uses line of sight vision)
so at least $k^2$ point guards are necessary to control gallery. However, $2k+2$ horizontal mobile guards can easily cover the polygon, and the same holds for vertical mobile guards.
\begin{figure}
	\centering
	\begin{tikzpicture}[scale=0.3]
	\def \n {4}
	\filldraw[fill=inside] (-0.25,-0.25) rectangle (4*\n+2.75,4*\n+2.75);

	\foreach \x in {1,...,\n}
	{
		\foreach \y in {1,...,\n}
		{
			\filldraw[fill=white] (4*\x-2,4*\y-2) rectangle ++(-1.5,-1.5);
			\filldraw[fill=white] (4*\x-2,4*\y) rectangle ++(-1.5,-1.5);
			\filldraw[fill=white] (4*\x,4*\y-2) rectangle ++(-1.5,-1.5);

			\draw[black,fill=black] (4*\x-0.75,4*\y-0.75) circle (0.5ex);
		}
	}

	\foreach \x in {1,...,\n}
	{
			\filldraw[fill=white] (4*\x-2,4*\n+2) rectangle ++(-1.5,-1.5);
			\filldraw[fill=white] (4*\x,4*\n+2) rectangle ++(-1.5,-1.5);
			\filldraw[fill=white] (4*\n+2,4*\x-2) rectangle ++(-1.5,-1.5);
			\filldraw[fill=white] (4*\n+2,4*\x) rectangle ++(-1.5,-1.5);
	}
	\filldraw[fill=white] (4*\n+2,4*\n+2) rectangle ++(-1.5,-1.5);
\end{tikzpicture}
	\caption{A polygon with holes --- unlimited ratio.}\label{fig:holes}
\end{figure}
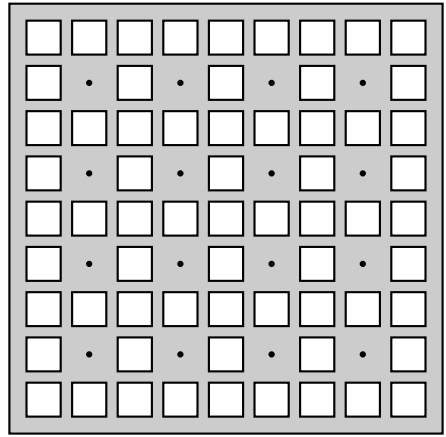

\medskip

In the next chapter, we show that a minimum size horizontal mobile $r$-guard system can be found in linear time (\Fref{thm:mgalg}). This improves the result in~\cite{KM11}, where it is shown that this problem can be solved in polynomial time.

%\Fref{fig:comb} shows a comb, which can be guarded by 1 mobile guard (whose patrol is shown in \textbf{\textcolor{red}{red}}). Thus $m=1$ and $m_H=1$, but to cover the comb using only vertical mobile guards (shown in \textbf{\textcolor{blue}{blue}}), one has to be placed in each tooth, so $m_V=10$. Similarly, $p=10$, as one point guard (shown in \textbf{\textcolor{green}{green}}) needs to placed in each tooth. Combs with arbitrarily high number of teeth clearly demonstrate that there exists no univariate function of either $m$, $m_V$, or $m_H$ which bounds $p$ from above for every orthogonal polygon. \figcap{comb}{A comb with 10 teeth}

%\subsubsection{Our results.} We study the problem of $r$-guarding orthogonal art galleries with vertical mobile guards (alternatively, horizontal) and point guards. We prove a sharp bound on the minimum number of point guards required to cover the gallery in terms of the minimum number of vertical mobile guards and the minimum number of horizontal mobile guards required to cover the gallery (\Fref{thm:main}).

\begin{theorem}[\citet{GyM2017}]\label{thm:main}
	Given a rectilinear domain $D$ let $m_V$ be the minimum size of a vertical mobile $r$-guard system of $D$, let $m_H$ be defined analogously for horizontal mobile $r$-guard systems, and finally let $p$ be the minimum size of a point $r$-guard system of $D$. Then
	\[ \left\lfloor\frac{4(m_V+m_H-1)}{3}\right\rfloor\ge p. \]
\end{theorem}

\medskip

In case it is not confusing, the prefix ``$r$-'' is omitted from now on. Before moving onto the proof of \Fref{thm:main}, we discuss the aspects of its sharpness.

\medskip

For $m_V+m_H\le 6$, sharpness of the theorem is shown by the examples in \Fref{fig:sharpness}.
The polygon in \Fref{fig:sharpness6} can be easily generalized to one satisfying $m_V+m_H=3k+1$ and $p=4k$. For $m_V+m_H=3k+2$ and $m_V+m_H=3k+3$, we can attach 1 or 2 plus signs to the previously constructed polygons, as shown in \Fref{fig:sharpness4} and~\ref{fig:sharpness5}.
Thus \Fref{thm:main} is sharp for any fixed value of $m_V+m_H$.

\medskip

By stringing together a number of copies of the polygons in \Fref{fig:sharpness1} and~\ref{fig:sharpness3} in an L-shape (\Fref{fig:sharpness6} is a special case of this), we can construct rectilinear domains for any $(m_H,m_V)$ pair satisfying $m_V\le 2(m_H-1)$ and $m_H\le 2(m_V-1)$, such that the polygon satisfies \Fref{thm:main} sharply.
The analysis in \Fref{sec:translating} immediately yields that if $m_V=1$ or $m_H=1$, then $m_V+m_H-1$ is an upper bound for the minimum size of a point guard system (see~\Fref{prop:star}), whose sharpness is shown by combs (\Fref{fig:comb}).

%\medskip
%
%Let $m$ be the minimum size of a mixed set of vertical and horizontal mobile $r$-guard system of $D$. Furthermore, for any integer choice of $m,m_V,m_H\ge 2$, where $m\ge m_V$ and $m\ge m_H$, we can construct a polygon whose mobile guard parameters take the previous values and $\frac43\cdot(m_V+m_H-1)$ point $r$-guards are required to cover it. This is achieved by stringing together a number of copies of the polygons in \Fref{fig:sharpness1} and~\ref{fig:sharpness3} in a \emph{staircase} shape, and attaching to it zero, one, or two plus signs.

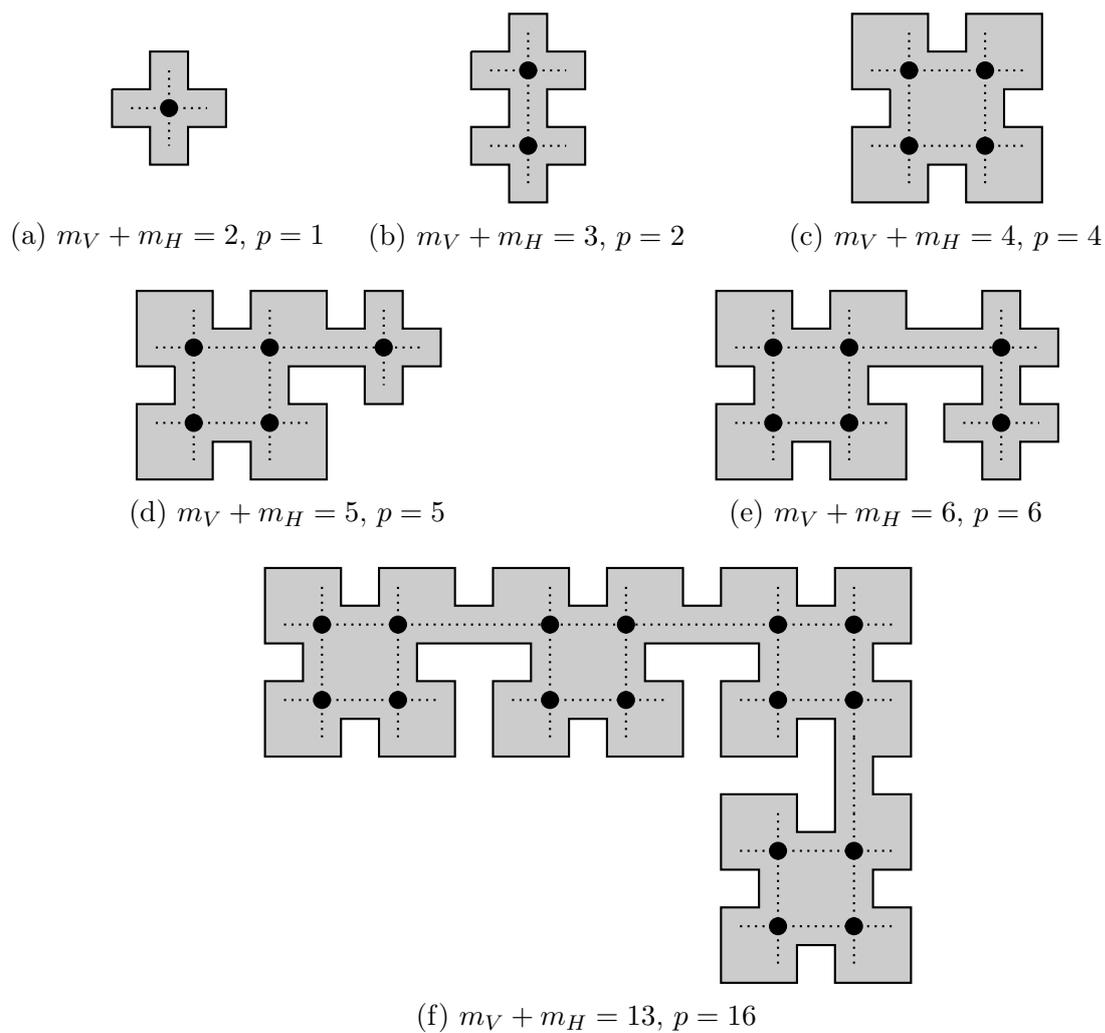
\begin{figure}
	\centering
	\begin{subfigure}{.3\textwidth}
		\centering
		\begin{tikzpicture}
	\begin{scope}

		\fill[white] (\dist,\dist) rectangle ++(\dist,\dist);
		\fill[white] (\dist,\dist*-3) rectangle ++(\dist,\dist);
		\filldraw [fill=inside][turtle=home]
		\foreach \i in {1,...,4}
		{
			[turtle={right,forward,left,forward,right,forward}]
		};
	\end{scope}

	\begin{scope}[\mgsize]
		\draw[\hpat] (\dist*0.5,\dist*-0.5) -- (\dist*2.5,\dist*-0.5);
		\draw[\vpat] (\dist*1.5,\dist*0.5) -- (\dist*1.5,\dist*-1.5);
		\filldraw[black] (\dist*1.5,\dist*-0.5) circle ( \pgsize );
	\end{scope}
\end{tikzpicture}
		\caption{\small $m_V+m_H=2$, $p=1$}\label{fig:sharpness1}
	\end{subfigure}%
	\begin{subfigure}{.3\textwidth}
		\centering
		\begin{tikzpicture}
	\begin{scope}

		\filldraw [fill=inside][turtle=home]
		\foreach \i in {1,2}
		{
			[turtle={right,forward,left,forward,right,forward}]
		}
		[turtle={right,forward,left,left}]
		\foreach \i in {1,2,3}
		{
			[turtle={right,forward,left,forward,right,forward}]
		}
		[turtle={right,forward,left,left}]
		[turtle={right,forward,left,forward,right,forward}]
		;
	\end{scope}

	\begin{scope}[\mgsize]
		\draw[\hpat] (\dist*0.5,\dist*-0.5) -- (\dist*2.5,\dist*-0.5);
		\draw[\hpat] (\dist*0.5,\dist*-2.5) -- (\dist*2.5,\dist*-2.5);
		\draw[\vpat] (\dist*1.5,\dist*0.5) -- (\dist*1.5,\dist*-3.5);
		\filldraw[black] (\dist*1.5,\dist*-0.5) circle ( \pgsize );
		\filldraw[black] (\dist*1.5,\dist*-2.5) circle ( \pgsize );

	\end{scope}
\end{tikzpicture}
		\caption{\small $m_V+m_H=3$, $p=2$}\label{fig:sharpness2}
	\end{subfigure}%
	\begin{subfigure}{.40\textwidth}
		\centering
		\begin{tikzpicture}
	\begin{scope}

	\end{scope}

	\begin{scope} % shiftelni
		\filldraw [fill=inside][turtle=home]
		\foreach \i in {1,2,3,4}
		{
			[turtle={left,forward,right,forward,forward,right,forward,forward,right,forward,left,forward}]
		};

	\end{scope}

	\begin{scope}[\mgsize]

	\foreach \x in {\dist*0.5,\dist*2.5}
		\draw[\vpat] (\x,\dist*1.5) -- (\x,\dist*-2.5);

	\foreach \y in {\dist*0.5,\dist*-1.5}
		\draw[\hpat] (\dist*-0.5,\y) -- (\dist*3.5,\y);

	\foreach \x in {\dist*0.5,\dist*2.5}
	\foreach \y in {\dist*0.5,\dist*-1.5}
		\filldraw[black] (\x,\y) circle ( \pgsize );

	\end{scope}
\end{tikzpicture}
		\caption{\small $m_V+m_H=4$, $p=4$}\label{fig:sharpness3}
	\end{subfigure}

	\bigskip

	\begin{subfigure}{.5\textwidth}
		\centering
		\begin{tikzpicture}
	\begin{scope}

	\end{scope}

	\begin{scope} % shiftelni\\
		\filldraw [fill=inside][turtle=home]			[turtle={left,forward,right,forward,forward,right,forward,forward,right,forward,left,forward}]
		[turtle={left,forward,right,forward,forward,right,forward,left}]

		\foreach \i in {1,...,3}
		{
			[turtle={forward,left,forward,right,forward,right}]
		}
		[turtle={forward,left,forward,forward,left,forward}]

		\foreach \i in {1,2}
		{
			[turtle={left,forward,right,forward,forward,right,forward,forward,right,forward,left,forward}]
		};

	\end{scope}

	\begin{scope}[\mgsize]

	\foreach \x in {\dist*0.5,\dist*2.5}
	\draw[\vpat] (\x,\dist*1.5) -- (\x,\dist*-2.5);

	\draw[\hpat] (\dist*-0.5,\dist*0.5) -- (\dist*6.5,\dist*0.5);
	\draw[\vpat] (\dist*5.5,\dist*1.5) -- (\dist*5.5,\dist*-0.5);
	\filldraw[black] (\dist*5.5,\dist*0.5) circle ( \pgsize );

	\foreach \y in {\dist*-1.5}
		\draw[\hpat] (\dist*-0.5,\y) -- (\dist*3.5,\y);

	\foreach \x in {\dist*0.5,\dist*2.5}
	\foreach \y in {\dist*0.5,\dist*-1.5}
	\filldraw[black] (\x,\y) circle ( \pgsize );

	\end{scope}

\end{tikzpicture}
		\caption{\small $m_V+m_H=5$, $p=5$}\label{fig:sharpness4}
	\end{subfigure}%
	\begin{subfigure}{.5\textwidth}
		\centering
		\begin{tikzpicture}
	\begin{scope}

	\end{scope}

	\begin{scope} % shiftelni
		\filldraw [fill=inside][turtle=home]			[turtle={left,forward,right,forward,forward,right,forward,forward,right,forward,left,forward}]
		[turtle={left,forward,right,forward,forward,right,forward,left,forward}]

		\foreach \i in {1,2}
		{
			[turtle={forward,left,forward,right,forward,right}]
		}
		[turtle={forward,left,left}]
		\foreach \i in {1,2,3}
		{
			[turtle={right,forward,left,forward,right,forward}]
		}
		[turtle={right,forward,left,left}]
		[turtle={right,forward,left,forward,forward,forward,left,forward}]

		\foreach \i in {1,2}
		{
			[turtle={left,forward,right,forward,forward,right,forward,forward,right,forward,left,forward}]
		};

	\end{scope}

	\begin{scope}[\mgsize]

	\foreach \x in {\dist*0.5,\dist*2.5}
	\draw[\vpat] (\x,\dist*1.5) -- (\x,\dist*-2.5);

	\draw[\hpat] (\dist*-0.5,\dist*0.5) -- (\dist*7.5,\dist*0.5);
	\draw[\hpat] (\dist*5.5,\dist*-1.5) -- (\dist*7.5,\dist*-1.5);
	\draw[\vpat] (\dist*6.5,\dist*1.5) -- (\dist*6.5,\dist*-2.5);
	\filldraw[black] (\dist*6.5,\dist*0.5) circle ( \pgsize );
	\filldraw[black] (\dist*6.5,\dist*-1.5) circle ( \pgsize );

	\foreach \y in {\dist*-1.5}
	\draw[\hpat] (\dist*-0.5,\y) -- (\dist*3.5,\y);

	\foreach \x in {\dist*0.5,\dist*2.5}
	\foreach \y in {\dist*0.5,\dist*-1.5}
	\filldraw[black] (\x,\y) circle ( \pgsize );

	\end{scope}
\end{tikzpicture}
		\caption{\small $m_V+m_H=6$, $p=6$}\label{fig:sharpness5}
	\end{subfigure}

	\bigskip

	\begin{subfigure}{\textwidth}
		\centering
		\begin{tikzpicture}

	\def\n{2}
	\begin{scope}

		\filldraw [fill=inside][turtle=home]

		% felül előre
		\foreach \i in {1,...,5} %2n+1
		{
			[turtle={fd,rt,fd,fd,rt,fd,lt,fd,lt}]
		}

		[turtle={fd,rt,fd}]

		\foreach \i in {1,...,3}
		{
			[turtle={fd,rt,fd,fd,rt,fd,lt,fd,lt}]
		}

		\foreach \i in {1,2,3}
		{
			[turtle={fd,rt,fd,fd,rt,fd,fd,rt,fd,lt,fd,lt}]
		}

		[turtle={fd,fd,fd,lt,fd,lt,fd,rt,fd,fd,rt,fd,fd,rt,fd,lt,fd,lt,fd}]

		\foreach \i in {1,2}
		{
			[turtle=fd]

			\foreach \j in {1,2}
			{
					[turtle={fd,lt,fd,lt,fd,rt,fd,fd,rt,fd,fd,rt}]
			}

			[turtle={fd,lt,fd,lt,fd}]

		}

		[turtle={rt,fd}]

		;
	\end{scope}

	\begin{scope}[\mgsize]

		\draw[\hpat] (\dist*0.5,\dist*-0.5) -- (\dist*16.5,\dist*-0.5);

		\foreach \z in {0,...,\n}
		{
			\begin{scope}[shift={(\z*\dist*6,0)}]
				\foreach \x in {\dist*1.5,\dist*3.5}
				\draw[\vpat] (\x,\dist*0.5) -- (\x,\dist*-3.5);

				\draw[\hpat] (\dist*0.5,\dist*-2.5) -- (\dist*4.5,\dist*-2.5);

				\foreach \x in {\dist*1.5,\dist*3.5}
				\foreach \y in {\dist*-0.5,\dist*-2.5}
				\filldraw[black] (\x,\y) circle ( \pgsize );
			\end{scope}
		}

		\begin{scope}[shift={(\n*\dist*6,-\dist*6)}]
			\foreach \x in {\dist*1.5,\dist*3.5}
				\draw[\vpat] (\x,\dist*0.5) -- (\x,\dist*-3.5);

			\draw[\vpat] (\dist*3.5,\dist*2.5) -- (\dist*3.5,\dist*0.5);

			\foreach \y in {\dist*-0.5,\dist*-2.5}
				\draw[\hpat] (\dist*0.5,\y) -- (\dist*4.5,\y);

			\foreach \x in {\dist*1.5,\dist*3.5}
			\foreach \y in {\dist*-0.5,\dist*-2.5}
			\filldraw[black] (\x,\y) circle ( \pgsize );
		\end{scope}
	\end{scope}
\end{tikzpicture}
		\caption{\small $m_V+m_H=13$, $p=16$}\label{fig:sharpness6}
	\end{subfigure}

	\caption{\small Vertical dotted lines: a minimum size vertical mobile guard system; \\ Horizontal dotted lines: a minimum size horizontal mobile guard system; \\ Solid disks: a minimum size point guard system.}\label{fig:sharpness}
\end{figure}

\section{Translating the problem into the language of graphs}\label{sec:translating}

For graph theoretical notation and theorems used in this chapter (say, the block decomposition of graphs), the reader is referred to~\cite{Diestel}.

\begin{definition}[Chordal bipartite or bichordal graph,~\cite{GolumbicGoss78}]
	A graph $G$ is chordal bipartite iff any cycle $C$ of $\ge 6$ vertices of $G$ has a chord (that is $E(G[C])\supsetneqq E(C)$).
\end{definition}

Let $S_V$ be the set of internally disjoint rectangles we obtain by cutting vertically at each reflex vertex of a rectilinear domain $D$. Similarly, let $S_H$ be defined analogously for horizontal cuts of $D$. We may refer to the elements of these sets as {\bfseries vertical and horizontal slices}, respectively. Let $G$ be the intersection graph of $S_H$ and $S_V$, i.e.,
\[ G=\left(S_H\cup S_V,\left\{\{h,v\}\ :\ h\in S_H,\ v\in S_V,\ \mathrm{int}(h)\cap \mathrm{int}(v)\neq\emptyset\right\}\right).\]
In other words, a horizontal and a vertical slice are joined by an edge iff their interiors intersect; see \Fref{fig:pixelation}. We may also refer to $G$ as the \textbf{pixelation graph} of $D$.
Clearly, the {\bfseries set of  pixels} $\{\cap e\ |\ e\in E(G)\}$ is a cover of $D$. Let us define $c(e)$ as the center of gravity of $\cap e$ (the pixel determined by $e$).
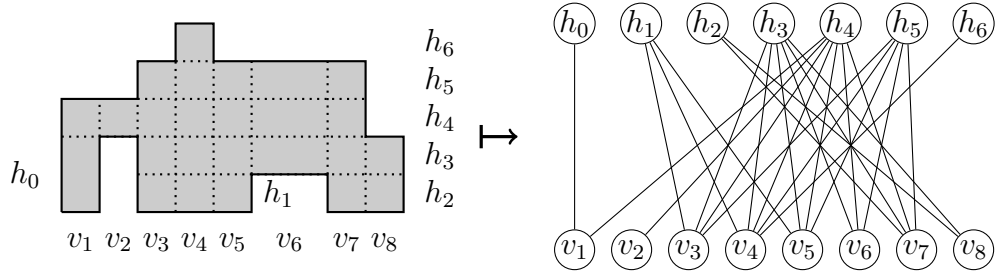
\begin{figure}[h]
	\centering
	\begin{tikzpicture}
\begin{scope}
	\filldraw[fill=inside][turtle=home]
	[turtle=fd,fd,fd,rt,fd,fd,lt,fd,rt,fd,lt,fd,rt,fd,rt,fd,lt,fd,fd,fd,fd,rt,fd,fd,lt,fd,rt,fd,fd,rt,fd,fd,rt,fd,lt,fd,fd,lt,fd,rt,fd,fd,fd,rt,fd,fd,lt,fd,lt,fd,fd,rt,fd];

	\draw[dotted] (3*\dist,4*\dist) -- ++(\dist,0);
	\draw[dotted] (2*\dist,3*\dist) -- ++(6*\dist,0);
	\draw[dotted] (0*\dist,2*\dist) -- ++(8*\dist,0);
	\draw[dotted] (2*\dist,1*\dist) -- ++(7*\dist,0);

	\draw[dotted] (1*\dist,2*\dist) -- ++(0,1*\dist);
	\draw[dotted] (2*\dist,2*\dist) -- ++(0,1*\dist);
	\draw[dotted] (3*\dist,0*\dist) -- ++(0,4*\dist);
	\draw[dotted] (4*\dist,0*\dist) -- ++(0,4*\dist);
	\draw[dotted] (5*\dist,1*\dist) -- ++(0,3*\dist);
	\draw[dotted] (7*\dist,1*\dist) -- ++(0,3*\dist);
	\draw[dotted] (8*\dist,0*\dist) -- ++(0,2*\dist);

	\foreach \x in {1,...,5}
		\draw (\dist*\x-0.5*\dist,-0.25*\dist) node[anchor=north] {$v_\x$};

	\draw (\dist*6.5-0.5*\dist,-0.25*\dist) node[anchor=north] {$v_6$};

	\foreach \x in {7,8}
		\draw (\dist*\x+0.5*\dist,-0.25*\dist) node[anchor=north] {$v_\x$};

	\foreach \y in {2,...,6}
		\draw (9.25*\dist,\dist*\y-1.5*\dist) node[anchor=west] {$h_\y$};

	\draw (-0.25*\dist,1*\dist) node[anchor=east] {$h_0$};
	\draw (5*\dist,0.5*\dist) node[anchor=west] {$h_1$};
\end{scope}

\draw [very thick,|->] (11*\dist,2*\dist) -- ++(1*\dist,0);

\begin{scope}[shift={(12*\dist,0)}]]
	\foreach \x in {0,...,6}
		\node[draw,thin,circle,inner sep=0pt,minimum size=15pt](h_\x) at (\x*7*1.5/6*\dist+1.5*\dist,5*\dist) {$h_\x$};

	\foreach \x in {1,...,8}
		\node[draw,thin,circle,inner sep=0pt,minimum size=15pt](v_\x) at
		(\x*1.5*\dist,-1*\dist) {$v_\x$};

	\foreach \v/\h in {1/0,3/1,4/1,5/1,1/4,2/4,3/3,3/4,3/5,4/3,4/4,4/5,4/6,5/3,5/4,5/5,6/3,6/4,6/5,7/2,7/3,7/4,7/5,8/2,8/3}
		\draw[thin] (v_\v) -- (h_\h); % edge[bend left=-4*\v+4*\h-6]
\end{scope}
\end{tikzpicture}
	\caption{A rectilinear domain and its associated pixelation graph}\label{fig:pixelation}
\end{figure}

\medskip

The horizontal $R$-tree $T_H$ of $D$ defined in \Fref{sec:treestructure} is equal to
\[ T_H=\left(S_H,\Big\{\{h_1,h_2\}\subseteq S_H\ :\ h_1\neq h_2,\ h_1\cap h_2\neq\emptyset\Big\}\right), \]
i.e., $T_H$ is the intersection graph of the horizontal slices of $D$. Similarly, $T_V$ is the intersection graph of the vertical slices of $D$.

% \vbox{\begin{claim}\label{claim:cut-tree}
% 	Both $T_H$ and $T_V$ are trees.
% \end{claim}
% \begin{proof}
% 	Connectedness of the trees follows from the connectedness of $D$. Furthermore, given an edge $e=\{h_1,h_2\}\in E(T)$, the set $D-\cap e=D- h_1\cap h_2=D-\partial h_1\cap \partial h_2$ has two components, therefore $T_H-e$ must have two components as well.
% \end{proof}}
%
% \medskip

\begin{lemma}\label{lemma:chordal}
	$G$ is a connected chordal bipartite graph.
\end{lemma}
\begin{proof}
	Connectedness of $D$ immediately yields that $G$ is connected too. Suppose $C$ is a cycle of $\ge 6$ vertices in $G$. For each node of the cycle $C$, connect the centers of gravity of its two incident edges with a line segment. This way we get an orthogonal polygon $P$ in $D$.

	\medskip

	If $P$ is self-intersecting, then the vertices which are represented by the two intersecting line segments are intersecting. This clearly corresponds to a chord of $C$ in $G$.

	\medskip

	If $P$ is simple, then the number of its vertices is $|V(C)|$, thus one of them is a reflex vertex, say $c(v_1\cap h_1)$ is one. As $P$ lives in $D$, its interior is a subset of $D$ as well (here we use that $D$ is simply connected). The simpleness of $P$ also implies that the vertical line segment intersecting $c(v_1\cap h_1)$ after entering the interior of $P$ at $c(v_1\cap h_1)$, intersects $P$ at least once more when it emerges, say at $c(v_1\cap h_2)$. As this is not an intersection of the line segments corresponding to two vertices of $D$, the edge $\{v_1,h_2\}$ is a chord of $C$.
\end{proof}

It is worth mentioning that even if $D$ is a rectilinear domain with rectilinear hole(s), $G$ may still be chordal bipartite. Take, for example, ${[0,3]}^2\setminus {(1,2)}^2$; the graph associated to it has only one cycle, which is of length 4.

\medskip

We will use the following technical claim to translate $r$-vision of points of $D$ into relations in $G$.

\begin{claim}\label{claim:rectvision}
	Let $e_1,e_2\in E(G)$, where $e_1=\{v_1,h_1\}$, $e_2=\{v_2,h_2\}$, $v_1,v_2\in S_V$, and $h_1,h_2\in S_H$. The points $p_1\in \mathrm{int}(\cap e_1)$ and $p_2\in \mathrm{int}(\cap e_2)$ have $r$-vision of each other in $D$ iff $e_1\cap e_2\neq\emptyset$ or $e_1\cup e_2$ induces a 4-cycle in $G$.
\end{claim}
\begin{proof}
	If $v_1\in e_1\cap e_2$, then $p_1,p_2\in v_1$, therefore $p_1$ and $p_2$ have $r$-vision of each other. If $h_1\in e_1\cap e_2$, the same holds. If $\{v_1,h_1,v_2,h_2\}$ induces a 4-cycle, then
	\[ \mathrm{Conv}((v_1\cap h_1)\cup (v_1\cap h_2))\subseteq v_1\subseteq D \]
	by $v_1$'s convexity.	Moreover,
	\begin{align*}
		B=   & \mathrm{Conv}((v_1\cap h_1)\cup (v_1\cap h_2))\cup \mathrm{Conv}((v_1\cap h_2)\cup (v_2\cap h_2))\cup \\
		\cup & \mathrm{Conv}((v_2\cap h_2)\cup (v_2\cap h_1))\cup \mathrm{Conv}((v_2\cap h_1)\cup (v_1\cap h_1))
	\end{align*}
	is contained in $D$. Since $D$ is simply connected, we have $\mathrm{Conv}(B)\subseteq D$, which is a rectangle containing both $p_1$ and $p_2$.

	\medskip

	In the other direction, suppose $e_1\cap e_2=\emptyset$. If $R$ is an axis-aligned rectangle which contains both $p_1$ and $p_2$, then $R$ clearly intersects the interiors of each element of $e_1\cup e_2$, which implies that $\mathrm{int}(v_2)\cap \mathrm{int}(h_1)\neq\emptyset$ and $\mathrm{int}(v_1)\cap \mathrm{int}(h_2)\neq\emptyset$. Thus $e_1\cup e_2$ induces a cycle in $G$.
\end{proof}
This easily implies the following claim.
\begin{claim}\label{claim:rectvision2}
	Two points $p_1,p_2\in D$ have $r$-vision of each other iff $\exists e_1,e_2\in E(G)$ such that $p_1\in \cap e_1$, $p_2\in \cap e_2$, and either $e_1\cap e_2\neq\emptyset$ or $e_1\cup e_2$ induces a 4-cycle in $G$.
\end{claim}

These claims motivate the following definition.

\begin{definition}[$r$-vision of edges]\label{def:rvisionedge}
	For any $e_1,e_2\in E(G)$ we say that $e_1$ and $e_2$ have $r$-vision of each other iff $e_1\cap e_2\neq\emptyset$ or there exists a $C_4$ in $G$ which contains both $e_1$ and $e_2$.
\end{definition}

Let $Z\subseteq E(G)$ be such that for any $e_0\in E(G)$ there exists an $e_1\in Z$ so that $e_1$ has $r$-vision of $e_0$.
According \Fref{claim:rectvision2}, if we choose a point from $\mathrm{int}(\cap e_1)$ for each $e_1\in Z$, then we get a point $r$-guard system of $D$.
%Therefore in the proof of \Fref{thm:main} we only need to find a suitable $Z$, a problem which is defined conveniently in terms of graph theoretic concepts.

\medskip

Observe that any vertical mobile $r$-guard is contained in $\mathrm{int}(v)$ for some $v\in S_V$ (except $\le 2$ points of the patrol).
Extending the line segment the mobile guard patrols increases the area that it covers, therefore we may assume that this line segment intersects each element of $\{\mathrm{int}(\cap e)\ |\ v\in e\in E(G)\}$, which only depends on some $v\in S_V$.
Using \Fref{claim:rectvision2}, we conclude that the set which such a mobile guard covers with $r$-vision is exactly $\cup\{ h\in S_H\ |\ \{h,v\}\in E(G) \}$. The analogous statement holds for horizontal mobile guards as well.

\medskip

Thus, a set of mobile guards of $D$ can be represented by a set $M_V\subseteq S_V$. Clearly, $M_V$ covers $D$ if and only if
\[ D = \bigcup_{v\in M_V}\left(\bigcup N_G(v)\right),\text{ which holds iff } S_H=\bigcup_{v\in M_V} N_G(v), \]
or in other words, $M_V$ dominates each element of $S_H$ in $G$. Similarly, a horizontal mobile guard system has a representative set $M_H\subseteq S_H$, which dominates $S_V$ in $G$. Equivalently, $M_H\cup M_V$ is a totally dominating set of $G$, i.e., a subset of $V(G)$ that dominates every node of $G$ (even the nodes of $M_H\cup M_V$).

\medskip

The same arguments imply that a mixed set of vertical and horizontal mobile $r$-guards is represented by a set of vertices of $S\subseteq V(G)$. The set of guards is a covering system of guards of $D$ if and only if every node $V(G)\setminus S$ has neighbor in $S$, i.e., $S$ is a dominating set in $G$.
\Fref{table:translation} is the dictionary that lists the main notions of the original problem and their corresponding phrasing in the pixelation graph.

\begin{table}
	\centering
	\bgroup%
	\def\arraystretch{1.5}
	\small
	\begin{tabular}{ c  c }
		\textbf{Orthogonal polygon}      & \textbf{Pixelation graph}                               \\ \midrule\midrule
		Mobile guard                     & Vertex                                                  \\ \toprule
		Point guard                      & Edge                                                    \\ \midrule
		Simply connected                 & Chordal bipartite ($\Rightarrow$, but $\not\Leftarrow$) \\ \midrule
		$r$-vision of two points         & $e_1\cap e_2\neq\emptyset$ or $G[e_1\cup e_2]\cong C_4$ \\ \midrule
		Horiz.~mobile guard cover        & $M_H\subseteq S_H$ dominating $S_V$                     \\ \midrule
		Covering system of mobile guards & Dominating set                                          \\ \bottomrule
	\end{tabular}
	\egroup%

	\bigskip

	\caption{Translating the orthogonal art gallery problem to the pixelation graph}\label{table:translation}
\end{table}

\medskip

As promised, the following claim has a very short proof using the definitions and claims of this section.

\begin{proposition}\label{prop:star}
	If $m_V=1$ or $m_H=1$, then $p\le m_V+m_H-1$.
\end{proposition}
\begin{proof}
	Let $Z$ be the set of edges of $G$ induced by $M_H\cup M_V$. Clearly, $G[M_H\cup M_V]$ is a star, thus $|Z|=|M_H|+|M_V|-1$.

	\medskip

	We claim that $Z$ covers $E(G)$. There exist two slices, $h_1\in M_H$ and $v_1\in M_V$, which are joined by an edge to $v_0$ and $h_0$, respectively.
	Since $G[M_H\cup M_V]$ is a star, $\{v_1,h_1\}\in Z$. This edge has $r$-vision of $e_0$, as either $\{v_1,h_1\}$ intersects $e_0$, or $\{v_0,h_0,v_1,h_1\}$ induces a $C_4$ in $Z$.
\end{proof}

\medskip

Finally, we can state \Fref{thm:main} in a stronger form, conveniently via graph theoretic concepts.

\begin{thmbis}{thm:main}\label{thm:mainprime} Let $A_V$ be a set of internally disjoint axis-parallel rectangles of a rectilinear domain $D$ (we call them vertical slices). Similarly, let $A_H$ be another set with the same property, whose elements we call the horizontal slices.
	Also, suppose that for any $v\in A_V$, its top and bottom sides are a subset of $\partial D$, and for any $h\in A_H$, its left and right sides are a subset of $\partial D$.
	Furthermore, suppose that their intersection graph \[ G=\left(A_H\cup A_V,\big\{\{h,v\}\subseteq A_V\cup A_H\ :\ \mathrm{int}(v)\cap\mathrm{int}(h)\neq\emptyset\big\}\right) \]
	is connected.

	\medskip

	If $M_V\subseteq A_V$ dominates $A_H$ in $G$, and $M_H\subseteq A_H$ dominates $A_V$ in $G$, then there exists a set of edges $Z\subseteq E(G)$ such that any element of $E(G)$ is $r$-visible from some element of $Z$, and
	\[ |Z|\le \frac43\cdot(|M_V|+|M_H|-1). \]
\end{thmbis}

\medskip

Now we are ready to prove the main theorem of this chapter.

\section{Proof of \texorpdfstring{\Fref{thm:mainprime}}{Theorem~\ref{thm:main}'}}
Both $A_H$ and $A_V$ can be extended to a partition of $D$ (while preserving the assumptions of the theorem on them), so $G$ is a subgraph induced by $A_H\cup A_V$ in a chordal bipartite graph (see \Fref{lemma:chordal}), thus $G$ is chordal bipartite as well. Let $M=G[M_V\cup M_H]$ be the subgraph induced by the dominating sets. Notice, that the bichordality of $G$ is inherited by $M$.

\medskip

\begin{claim}\label{claim:conn_rvis}
	If $M$ is connected, then any edge $e_0=\{h_0,v_0\}\in E(G)$ is $r$-visible from some edge of $M$.
\end{claim}
\begin{proof}
	As $N_G(M_V\cup M_H)=V(G)$, there exists two vertices, $v_1\in M_V$ and $h_1\in M_H$, such that $\{v_1,h_0\},\{v_0,h_1\}\in E(G)$.

	\medskip

	If $v_0\in M_V$ or $h_0\in M_H$, then $\{v_0,h_1\}$ or $\{v_1,h_0\}$ is in $E(M)$.

	\medskip

	Otherwise, there exists a path in $M$, whose endpoints are $v_1$ and $h_1$, and this path and the edges $\{v_1,h_0\}$,$\{h_0,v_0\}$,$\{v_0,h_1\}$ form a cycle in $G$. By the bichordality of $G$, there exists a $C_4$ in $G$ which contains an edge of $M$ and $e_0$.
\end{proof}

\medskip

We distinguish 3 cases based on the level connectivity of $M$.

\subsection{\texorpdfstring{\boldmath $M$}{𝑀} is 2-connected}\label{case:2conn}
The $\frac43$ constant in the statement of \Fref{thm:mainprime} is determined by this case. Knowing this, it is not surprising that this is the longest and most complex case of the proof.

\medskip

If $E(M)$ consists of a single edge $e$, then $Z=\{e\}$ is clearly a point guard system of $G$ by~\Fref{claim:conn_rvis}.

\medskip

Suppose now, that $M$ has more than two vertices. Any edge of $M$ is contained in a cycle of $M$, and by the bichordality property, there is such a cycle of length 4. It is easy to see that the convex hull of the pixels determined by the edges of a $C_4$ is a rectangle. Define
\[ D_M=\bigcup\limits_{\{e_1,e_2,e_3,e_4\}\text{ is a $C_4$ in }M}\mathrm{Conv}\left(\bigcup\limits_{i=1}^4 \cap e_i\right).\]
The simply connectedness of $D$ implies that $D_M\subseteq D$.

\begin{claim}\label{claim:D}
	For any slice $s\in V(M)$ the intersection of $s$ and $D_M$ is connected.
\end{claim}
\begin{proof}
	Suppose that $e_1,e_2\in E(M)$ are such that $\cap e_1$ and $\cap e_2$ are in two different components of $s\cap D_M$. Since $M$ is 2-connected, there is a path connecting $e_1\setminus \{s\}$ and $e_2\setminus\{s\}$ in $M-s$.

	\medskip

	Take the shortest cycle in $M$ containing $e_1$ and $e_2$.
	If this cycle contains 4 edges, then the convex hull of their pixels is in $D_M$, which is a contradiction.
	Similarly, if the cycle contains more than 4 edges, the bichordality of $M$ implies that $s$ is joined to every second node of the cycle, which contradicts our assumption that $s\cap D_M$ is disconnected.
\end{proof}

\medskip

\begin{claim}\label{claim:DG}
	For any slice $s\in V(G)$, the intersection of $\,\mathrm{int}(s)$ and $D_M$ is connected.
\end{claim}
\begin{proof}
	If $s\in V(M)$, we are done by \Fref{claim:D}. If $s\in V(G)\setminus V(M)$, let $e_1,e_2\in E(G)$ be the two edges such that $e_1\cap e_2=\{s\}$, $\partial(\cap e_1)\bigcap \partial D_M\neq\emptyset$, $\partial(\cap e_2)\bigcap \partial D_M\neq\emptyset$.
	Then, we must have $\left(e_1\bigcup e_2\right)\setminus \{s\}\subseteq V(M)$. Take the shortest path in $M$ joining $e_1\setminus \{ s \}$ to $e_2\setminus \{ s \}$.
	The proof can be finished as that of the previous claim.
\end{proof}

\medskip

Let $B_H\subset M_H$ be the set of those slices whose top and bottom sides both intersect $\partial D_M$ in an uncountable number of points of $\mathbb{R}^2$.

\medskip

For technical reasons, we split each element of $h\in B_H$ horizontally through $c(h)$ to get two isometric rectangles in $\mathbb{R}^2$; let the set of the resulting refined horizontal slices be $B'_H$. Replace the elements of $A_H$ and $M_H$ contained in $B_H$ with their corresponding two halves in $B'_H$ to get
\[ A'_H=B'_H\bigcup A_H\setminus B_H\quad\text{and}\quad M'_H=B'_H\bigcup M_H\setminus B_H,\]
respectively. Let $A'_V=A_V$, $M'_V=M_V$. Let $\tau$ be the function which maps $h\in B'_H$ to the $\tau(h)\in A_H$ for which $h\subseteq \tau(h)$ holds, and let $\tau$ be the identity function on $A'_V\cup A'_H\setminus B'_H$.

\medskip

Let $G'$ be the intersection graph of $A'_H$ and $A'_V$ (as in the statement of \Fref{thm:mainprime}). Also, let $M'=G'[M'_H\cup M'_V]=\tau^{-1}(M)$. Observe that $\tau$ naturally defines a graph homomorphism $\tau:G'\to G$ (edges are mapped vertex-wise).

\medskip

\begin{claim}\label{claim:tau}
	In $G'$, the set $M'_H$ dominates $A'_V$, and $M'_V$ dominates $A'_H$.
	Furthermore, if $Z'\subseteq E(M')$ is a point guard system of $G'$, then $Z=\tau(Z')\subseteq E(M)$ is a point guard system of $G$.
\end{claim}
\begin{proof}
	The first statement of this claim holds, since $\tau$ maps non-edges to non-edges, and both $M'_H=\tau^{-1}(M_H)$ and $M'_V=\tau^{-1}(M_V)$ by definition.
	As $\tau$ is a graph homomorphism, it preserves $r$-visibility, which implies the second statement of this claim.
\end{proof}

\medskip

Notice, that $M'$ is 2-connected and $D_M=D_{M'}$.
An edge $e\in E(M')$ falls into one of the following 4 categories:
\begin{figure}[H]
	\centering
	\begin{tikzpicture}[scale=0.75]
	\begin{scope}
		\draw[very thick] (0,3) -- (0,10) -- (7,10) -- (7,7) -- (11,7) -- (11,10) -- (15,10) -- (15,5) -- (12,5) -- (12,0) -- (3,0) -- (3,3) -- (0,3);

		\def\hpat{north east lines}
		\def\vpat{north west lines}
		\filldraw[thin,pattern=\hpat] (0,3) rectangle (1,10);
		\filldraw[thin,pattern=\hpat] (3,0) rectangle (4,10);
		\filldraw[thin,pattern=\hpat] (5,0) rectangle (7,10);
		\filldraw[thin,pattern=\hpat] (8.5,0) rectangle (9.5,7);
		\filldraw[thin,pattern=\hpat] (11,0) rectangle (12,10);
		\filldraw[thin,pattern=\hpat] (14,5) rectangle (15,10);

		\filldraw[thin,pattern=\vpat] (3,0) rectangle (12,1);
		\filldraw[thin,pattern=\vpat] (0,3) rectangle (12,4);
		\filldraw[thin,pattern=\vpat] (0,5) rectangle (15,7);
		\draw[thin] (0,6) -- (15,6);
		\filldraw[thin,pattern=\vpat] (0,9) rectangle (7,10);
		\filldraw[thin,pattern=\vpat] (11,9) rectangle (15,10);

		\foreach \x/\y/\l in {1/10/1,4/10/2,6.5/10/3,9.5/7/4,12/10/5,15/10/6}
			\draw (\x-0.5,\y+0.5) node[anchor=south] {$v_\l$};

		\draw [decorate,thick,decoration={brace,amplitude=10pt}]
		(-0.1,5) -- (-0.1,7);

		\foreach \x/\y/\l in {3/0/1,0/3/2,0/5.5/3,0/9/4,16.5/9/5}
			\draw (\x-0.5,\y+0.5) node[anchor=east] {$h_\l$};

		\draw (16,5+0.5) node[anchor=east] {$h'_3$};
		\draw (16,6+0.5) node[anchor=east] {$h''_3$};

		\foreach \x/\y in {1/4,1/10,6.5/10,4/1,12/1,12/10,15/10,15/6}
			\draw (\x-0.5,\y-0.5) node[rectangle,rounded corners=3pt,fill=white,text opacity=1,fill opacity=1,inner sep=2pt] {\tiny convex};

		\foreach \x/\y in {1/6,1/7,15/7,4/10,12/4,9.5/7,9.5/1,6.5/1}
			\draw (\x-0.5,\y-0.5) node[rectangle,rounded corners=3pt,fill=white,text opacity=1,fill opacity=1,inner sep=2pt] {\tiny side};

		\foreach \x/\y in {4/4,6.5/7,12/7,12/6}
			\draw (\x-0.5,\y-0.5) node[rectangle,rounded corners=3pt,fill=white,text opacity=1,fill opacity=1,inner sep=2pt] {\tiny reflex};

		\foreach \x/\y in {4/6,4/7,6.5/4,6.5/6,9.5/4,9.5/6}
			\draw (\x-0.5,\y-0.5) node[rectangle,rounded corners=3pt,fill=white,text opacity=1,fill 	opacity=1,inner sep=2pt] {\tiny internal};

	\end{scope}
\end{tikzpicture}
\caption{We have $M_H=\{h_1,h_2,h_3,h_4,h_5\}$, $M'_H=M_H-h_3+h_3'+h_3''$, and $M_V=M'_V=\{v_1,v_2,v_3,v_4,v_5,v_6\}$. The thick line is the boundary of $D_M$. Each rectangle pixel is labeled according to the type of its corresponding edge of $M'$.}\label{fig:types}
\end{figure}
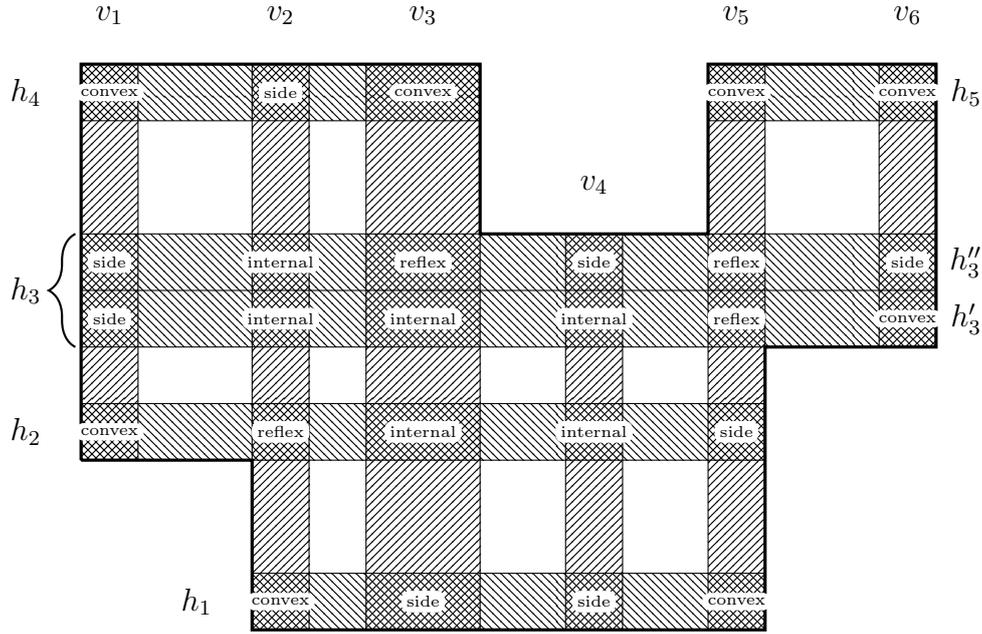

\begin{description}\setlength{\baselineskip}{1em}
	\item[Convex edge:] 3 vertices of $\cap e$ fall on $\partial D_M$, e.g., the edge $\{h_2,v_1\}$ on \Fref{fig:types};

	\item[Reflex edge:] exactly 1 vertex of $\cap e$ falls on $\partial D_M$, e.g., $\{h_3'',v_3\}$ on \Fref{fig:types};

	\item[Side edge:] two neighboring vertices of $\cap e$ fall on $\partial D_M$,  e.g., $\{h_1,v_4\}$ on \Fref{fig:types};

	\item[Internal edge:] zero vertices of $\cap e$ fall on $D_M$,  e.g., $\{h_2,v_3\}$ on \Fref{fig:types}.
\end{description}

Notice that on \Fref{fig:types}, the edge $\{h_3,v_5\}$ falls into neither of the previous categories, as two non-neighboring (diagonally opposite) vertices of pixel $h_3\cap v_5$ fall on $D_M$. This clearly cannot happen with edges of $G'$, but $G$ may contain edges of this type.

\medskip

Observe that $\tau$ maps convex edges to convex edges, and side edges to side edges. Conversely, the preimages of a convex edge are a convex edge and a side edge ($M'$ is 2-connected), the preimages of a side edge are two side edges, and the preimages of a reflex edge are a reflex edge and an internal edge.

\medskip

The following definition allow us to break our proof into smaller, transparent parts, which ultimately boils down to presenting a precise proof. It captures a condition which in certain circumstances allows us to conclude that a guard $e_1$ can be replaced by $e_2$ such that we still have complete coverage of $G'$.

\begin{definition}\label{def:dominance} For any two edges $e_1,e_2\in E(M')$, where $e_1=\{v_1,h_1\}$ and $e_2=\{v_2,h_2\}$, we write $e_2 \covers e_1$ ($e_2$ dominates $e_1$) iff either
	\begin{itemize}
		\item $e_1\cap e_2\subset A'_H$, and $\exists h_3,h_4\in M'_H$ such that $\{v_1,v_2,h_3,h_4\}$ induces a $C_4$ in $M'$, and $h_1=h_2$ is between $h_3$ and $h_4$; or

		\item $e_1\cap e_2\subset A'_V$, and $\exists v_3,v_4\in M'_V$ such that $\{v_3,v_4,h_1,h_2\}$ induces a $C_4$ in $M'$, and $v_1=v_2$ is between $v_3$ and $v_4$; or

		\item $e_1\cap e_2=\emptyset$, and $\exists v_3\in M'_V$ and $h_3\in M'_H$ such that both $\{v_1,h_2,v_2,h_3\}$ and $\{h_1,v_3,h_2,v_2\}$ induces a $C_4$ in $M'$;
		      furthermore, $v_1$ is between $v_2$ and $v_3$, and $h_1$ is between $h_2$ and $h_3$.
	\end{itemize}
	We write $e_2 \rela e_1$ iff both $e_2 \covers e_1$ and $e_1 \covers e_2$ hold. Note that $\rela$ is a symmetric, but generally intransitive relation.
\end{definition}

For example, on \Fref{fig:types}, $\{h_1,v_3\}\rela \{h_3'',v_3\}$, and $\{h_1,v_2\}\covers \{h_3'',v_3\}$.
Also, $\{h_3'',v_3\}\rela \{h_3'',v_1\}$, but $\{h_3'',v_3\}\not\rightarrow \{h_3',v_1\}$. This is a technicality which makes the proofs easier, but does not cause any issues in the end, as $\tau(\{h_3'',v_1\})=\tau(\{h_3',v_1\})$.

\medskip

We will search for a point guard system of $M'$ with very specific properties, which are described by the following definition.

\begin{definition}
	Suppose $Z'\subseteq E(M')$ is such, that
	\begin{enumerate}
		\item $Z'$ contains every convex edge of $M$,
		\item for any non-internal edge $e_1\in E(M)\setminus Z'$, there exists some $e_2\in Z'$ for which $e_2\covers e_1$, and
		\item\label{item:internal} for each $h_0\in A'_H$ for which $\mathrm{int}(h_0)\cap D_M\neq\emptyset$ holds, $\exists \{h_2,v_2\}\in Z'$ such that $\{h_0,v_2\}\in E(M')$ and $N_{M'}(h_2)\supseteq N_{G'}(h_0)\bigcap M'_V$.
	\end{enumerate}
	If these three properties hold, we call $Z'$ a \textbf{hyperguard} of $M'$.
\end{definition}

\begin{lemma}\label{lemma:zprime}
	Any hyperguard $Z'$ of $M'$ is a point guard system of $G'$, i.e., any edge of $G'$ is $r$-visible from some element of $Z'$.
\end{lemma}
\begin{proof}
	Let $e_0=\{v_0,h_0\}\in E(G')$ be an arbitrary edge. By~\Fref{claim:conn_rvis}, there exists an edge
	$e_1\in E(M')$ which has $r$-vision of $e_0$, and we also suppose that $e_1$ is chosen so that Euclidean distance $\mathrm{dist}(\cap e_0,\cap e_1)$ is minimal.

	\medskip

	Trivially, if $e_1\in Z'$ (for example, if $e_1$ is a \textbf{convex edge} of $M'$), then $e_0$ is $r$-visible from $e_1$. Assume now, that $e_1\notin Z'$.
	\begin{itemize}
		\item If $e_1$ is a \textbf{reflex or side edge} of $M'$, then $\exists e_2\in Z'$ so that $e_2\covers e_1$. We claim that $e_2$ has $r$-vision of $e_0$ in $G'$ (this is the main motivation for \Fref{def:dominance}).
		      \begin{enumerate}
			      \item If $e_1\cap e_2\subset A'_H$: by the choice of $e_1$ and $e_2$, $v_1$ is joined to $h_0,h_3,h_4$ in $G'$. The choice of $e_1$ guarantees that $v_1\cap h_0$ is between $v_1\cap h_3$ and $v_1\cap h_4$.
			            Therefore $\mathrm{int}(v_2\cap h_0)\neq\emptyset$, so $\{v_0,h_0,v_2,h_1(=h_2)\}$ induces a $C_4$ in $G'$.

			      \item If $e_1\cap e_2\subset A'_V$: the proof proceeds analogously to the previous case.

			      \item If $e_1\cap e_2=\emptyset$: by the choice of $e_1$ and $e_2$, $v_1$ is joined to $h_0,h_3,h_2$ in $G'$, and $v_1$ is joined to $v_0,v_3,v_2$ in $G'$.
			            The choice of $e_1$ guarantees that $v_1\cap h_0$ is between $v_1\cap h_3$ and $v_1\cap h_2$, and that $v_0\cap h_1$ is between $v_3\cap h_1$ and $v_2\cap h_1$.
			            Therefore $\mathrm{int}(v_2\cap h_0)\neq\emptyset$ and $\mathrm{int}(v_0\cap h_2)\neq\emptyset$, so $\{v_0,h_0,v_2,h_2\}$ induces a $C_4$ in $G'$.
		      \end{enumerate}
		      In any of the three cases, $e_0$ is $r$-visible from $e_2$ in $G'$.

		\item If $e_1$ is an \textbf{internal edge} of $M'$, then $\cap e_0\subset D_M$. %int(D_M) holds as well
		      By the~\ref{item:internal}\textsuperscript{rd} property of hyperguards, there  $\exists \{h_2,v_2\}\in Z'$ such that $\{h_0,v_2\}\in E(M')$ and $N_{M'}(h_2)\supseteq N_{G'}(h_0)\bigcap M'_V$.
		      An easy argument (use that $D_M\subset D$ are both simply connected) gives that $\{v_0,h_2\}\in E(G')$.
		      Thus $\{v_0,h_2,v_2,h_0\}$ induces a $C_4$ in $G'$, so $e_0$ is $r$-visible from $\{v_2,h_2\}\in Z'$.
	\end{itemize}
	We have verified the statement in every case, so the proof of this lemma is complete.
\end{proof}

Notice, that the set of all convex, reflex, and side edges of $E(M')$ form a hyperguard of $M'$. By \Fref{lemma:zprime}, this set is a point guard system of $G'$, and \Fref{claim:tau} implies that its $\tau$-image is a point guard system of $G$. The cardinality of the $\tau$-image of this hyperguard is bounded by $2|V(M)|-4$ (we will see this shortly), which is already a magnitude lower than what the trivial choice of $E(M)$ would give (generally, $|E(M)|$ can be equal to $\Omega(|V(M)|^2)$).

\medskip

Let the number of convex, side, and reflex edges in $M'$ be $c'$, $s'$, and $r'$, respectively. \Fref{claim:D} and~\ref{claim:DG} allow us to count these objects.

\begin{enumerate}
	\item The number of reflex vertices of $D_M$ is equal to $r'$: any reflex vertex is a vertex of a reflex edge, and the way $M'$ and $D_M$ is constructed guarantees that exactly one vertex of the pixel of a reflex edge is a reflex vertex of $D_M$.

	\item The number of convex vertices of $D_M$ is equal to $c'$: any convex vertex is a vertex of the pixel of a convex edge, and the way $D_M$ is constructed guarantees that exactly one vertex of the pixel of a convex edge is a convex vertex.

	\item The cardinality of $V(M')$ is $c'+\frac12s'$: the first and last edge incident to any element of $V(M')$ ordered from left-to-right (for elements of $M'_H$) or from top-to-bottom (for elements of $M'_V$) is a convex or a side edge. Conversely, any convex edge is the first or last incident edge of exactly one element of $M'_H$ and one element of $M'_V$. A side edge is the first or last incident edge of exactly one element of $V(M')$.

	\item For any reflex edge $e_1=\{v_1,h_1\}\in E(M')$, there is exactly one reflex or side edge in $E(M')$ which contains $v_1$ and is in the $\rela$ relation with $e_1$, and the same can be said about $h_1$.

	\item Any side edge $e_1\in E(M')$ is in $\rela$ relation with exactly one reflex or side edge which it intersects. The intersection is the slice in $V(M')$ on which $e_1$ is a boundary edge.
\end{enumerate}

%Observe, that except the first statement of the previous list, the statements hold even for $\tau$ images.

We can now compute the size of the set of all convex, reflex, and side edges of $E(M')$:
\[ c'+r'+s'=2c'-4+s'=2|V(M')|-4.\]
Furthermore, it is clear that taking the $\tau$-image of this set decreases its by cardinality by $2|B_H|$ (new reflex and side edges are created at both ends of slices in $B_H$ when splitting them), and $2|V(M')|-4-2|B_H|=2|V(M)|-4$, proving the claim from the previous page. Readers who are only interested in a result which is sharp up to a constant factor, may skip to \Fref{case:conn}. Further analysis of $M'$ allows us to lower the coefficient $2$ to $\frac43$. 

\medskip

Define the {\bfseries auxiliary graph \boldmath $X$} as follows: let $V(X)$ be the set of reflex and side edges of $M'$, and let
\[ E(X)=\Big\{\{e,f\}:\ e\neq f,\ e\cap f\neq\emptyset,\ e\rela f\Big\}. \]
By our observations, $X$ is the disjoint union of some cycles and $\frac12s'$ paths. This structure allows us to select a hyperguard which contains a subset of the reflex and side edges of $M'$, instead of the whole set.

\subsubsection{Constructing a hyperguard \texorpdfstring{\boldmath $Z'$ of $M'$}{𝑍' of 𝑀'}.}  We will define  ${(Z'_j)}_{j=0}^\infty$, a sequence of (set theoretically) increasing sequence of subsets of $E(M')$, and ${(X_j)}_{j=0}^\infty$, a decreasing sequence of induced subgraphs of $X$.

\medskip

Additionally, we will define a function $w_j:V(X)\to \{0,1,2\}$, and extend its domain to any subgraph $H\subseteq X$ by defining $w_j(H)=\sum_{e\in V(H)}w_j(e)$. The purpose of $w_j$, very vaguely, is that as $Z'$ will contain every third node of $X$, we need to keep count of the modulo 3 remainders. Furthermore, $w_j$ serves as buffer in a(n implicitly defined) weight function (see \fref{ineq:recursion}).

\medskip

For a set $E_0\subseteq E(X)$, let the indicator function of $E_0$ be
\[
	\mathds{1}_{E_0}(e)=\left\{
	\begin{array}{ll}
		1,\quad & \text{if }e\in E_0,               \\
		0,\quad & \text{if }e\in E(X)\setminus E_0.
	\end{array}
	\right.
\]

\medskip

Let $Z'_0=\emptyset$ and $X_0=X$. By our previous observations, $X$ does not contain isolated nodes. Define $w_0:V(X)\to \{0,1,2\}$ such that
\[
	w_0(e)=\left\{
	\begin{array}{ll}
		1,\quad & \text{if }d_{X_0}(e)=1, \\
		0,\quad & \text{if }d_{X_0}(e)=0\text{ or }2.
	\end{array}
	\right.
\]

\medskip

In each step, we will define $Z'_j$, $X_j$, and $w_j$ so that
\begin{itemize}
	\item $Z'_{j-1}\subseteq Z'_{j}$, $X_{j}\subseteq X_{j-1}$,
	\item $\{e\in V(X_j)\ |\ d_{X_j}(e)=1\}\subseteq w^{-1}_j(1)$,
	\item $\{e\in V(X_j)\ |\ d_{X_j}(e)=0\}=w^{-1}_j(2)$, and
	\item $\forall e_0\in V(X)\setminus V(X_j)$, either $e_0\in Z'_j$, or $\exists e_1\in Z'_j$ so that $e_1\covers e_0$.
\end{itemize}
If these hold, then for any path component $P_j$ in $X_j$, we have $w_j(P_j)\ge 2$.

\phase{}\label{phase:initial} Let the set of convex edges of $M'$ be $C'$.
Let
\begin{align*}
	S' & =\Big\{ e\in V(X):\ \tau(e)\text{ is a side edge, }\exists f\in V(X)\  e\rela f,\ e\cap f\subseteq B'_H\Big\},                          \\
	T' & =\Big\{ f\in V(X):\ \exists e\in S'\ f\rela e,\ \tau^{-1}(\tau(f))\setminus \{f\}\covers N_X(\tau^{-1}(\tau(e)))\setminus \{f\} \Big\}, \\
	U' & =\tau^{-1}(\tau(C'))\setminus C',                                                                                                       \\
	Q' & = \bigcup_{\substack{e_1,e_4\in S'\cup U'                                                                                               \\ e_2,e_3\in V(X) \\ e_1\rela e_2, e_2\rela e_3, e_3\rela e_4}} \{e_1,e_2,e_3,e_4\}.
\end{align*}
Take
\begin{align*}
	Z'_{1} & =\tau^{-1}\left(\tau(C')\bigcup \tau(T')\right),                                                                                                                         \\
	X_1    & =X -T'-N_X(T') -U'-N_X(U'),                                                                                                                                              \\
	w_1    & =w_0-\mathds{1}_{S'}-\mathds{1}_{U'}+\sum_{f\in T'}\mathds{1}_{N_X(N_X(f))\setminus \{f\}\setminus Q'}+\sum_{e\in U'}\mathds{1}_{N_X(N_X(e))\setminus\{e\}\setminus Q'}.
\end{align*}

\medskip

\phase{}\label{phase:cycle} Take a cycle $e_1,e_2,\ldots,e_{2k_j}$ in $X_j$ ($k_j\ge 2$, $j\ge 1$). This set of nodes of $X_j$ is the edge set of a cycle of length $2k_j$ in $M'$.
\begin{itemize}
	\item If $2k_j=4$, observe that $e_1\rela e_2$, $e_1\rela e_4$, $e_2\rela e_3$, $e_4\rela e_3$ together imply that $e_1\rela e_3$. Take
	      \begin{align*}
		      Z'_{j+1} & =\{e_1\}\bigcup Z'_j,     \\
		      X_{j+1}  & =X_j-\{e_1,e_2,e_3,e_4\}, \\
		      w_{j+1}  & = w_j.
	      \end{align*}
	\item If $2k_j\ge 6$, the chordal bipartiteness of $M'$ implies that without loss of generality there is a chord $f\in E(M')$ which forms a cycle with $e_1,e_2,e_3$ in $M'$. Take
	      \begin{align*}
		      Z'_{j+1} & =\{f\}\bigcup Z'_j,                            \\
		      X_{j+1}  & =X_j-\{e_{2k_j},e_1,e_2,e_3,e_4\},             \\
		      w_{j+1}  & =w_j+\mathds{1}_{e_5}+\mathds{1}_{e_{2k_j-1}}.
	      \end{align*}
\end{itemize}

Iterate this step until $X_{j_1}$ is cycle-free.

\medskip

\phase{}\label{phase:selfintersect} Take a path $e_1,e_2,\ldots,e_{k}$ in $X_j$ (for $j\ge j_1$), such that
\[ E\left(M'\left[\bigcup_{i=2}^{k-1} e_i\right]\right)\setminus \{e_2,\ldots e_{k-1}\}\neq\emptyset.\]
Using the bichordality of $M'$, there exists a chord $f\in E(M')$ which forms a $C_4$ with $\{e_{l-1},e_l,e_{l+1}\}$, where $3\le l\le k-2$.
It is easy to see that $e_{l-2}\rela e_{l-1}$ implies $f\covers e_{l-2}$ and $f\covers e_{l-1}$. Similarly, we have that $f\covers e_{l+1}$ and $f\covers e_{l+2}$. Also, $f\covers e_{l-1}$ and $f\covers e_{l+1}$ together imply $f\covers e_l$.
Therefore, we take
\begin{align*}
	Z'_{j+1} & =\{f\}\bigcup Z'_j,                                   \\
	X_{j+1}  & =X_j-\{e_{l-2},e_{l-1},e_{l},e_{l+1},e_{l+2}\},       \\
	w_{j+1}  & =w_j+\mathds{1}_{\{\mathrm{dist}_X(\bullet,e_l)=3\}}.
\end{align*}
Iterate this step until $X_{j_2}$ is free of the above defined paths.

\medskip

\phase{}\label{phase:3rdproperty} The set $A'_H$ is the subset of the nodes of a horizontal $R$-tree of $D$. Let $h_\text{root}\in A'_H$ be an arbitrarily chosen node serving as the root of the horizontal $R$-tree. Process the elements of $A'_H$ in decreasing distance (measured in the horizontal $R$-tree) from $h_\text{root}$.

\medskip

Suppose $h_0\in A'_H$ is the next horizontal slice to be processed. If $\mathrm{int}(h_0)\cap D_M=\emptyset$ or $h_0\in M'_H$, then move on to the next slice of $A'_H$, as the~\ref{item:internal}\textsuperscript{rd} property of hyperguards for $Z'$ is satisfied by any edge of $M'$ incident to $h_0$.

\medskip

Suppose now, that $h_0\notin M'_H$. It is easy to see that there exists a $C_4$ in $M'$ whose edge set $\{e_1,e_2,e_3,e_4\}$ satisfies
\[ h_0\cap D_M\subset \mathrm{Conv}\left(\bigcup_{i=1}^4 \cap e_i\right). \]
Without loss of generality, we may suppose that we choose the $C_4$ so that the convex hull of the pixels of its edges is minimal. Then $e_i$ (for $i=1,2,3,4$) is not an internal-edge of $M'$, as this would contradict the choice of the $C_4$.

\medskip

If $e_i$ is a convex edge of $M'$, then it is already contained in $Z'_1\subset Z'$, so it satisfies the~\ref{item:internal}\textsuperscript{rd} property of hyperguards for $Z'$ and $h_0$, and we may skip to processing the next slice. If $e_i$ is a side edge of $M'$, then for any edge $f\in Z'$ which satisfies $f\covers e_i$, we have
$\cap f\subset\mathrm{Conv}\left(\bigcup_{i=1}^4 \cap e_i\right)$,
so $f$ satisfies the~\ref{item:internal}\textsuperscript{rd} property of hyperguards for $Z'$ and $h_0$, and again, we may skip to processing the next slice.

\medskip

Suppose now, that each $e_i$ (for $i=1,2,3,4$) is a reflex edge of $M'$. Let
\[\{h_1,h_2\}=M'_H\bigcap \bigcup_{i=1}^4 e_i\quad\text{and}\quad\{v_1,v_2\}=M'_V\bigcap \bigcup_{i=1}^4 e_i.\]
The minimality of the chosen $C_4$ implies that $\{h_1,v_1\}\rela \{h_1,v_2\}$ and $\{h_2,v_1\}\rela \{h_2,v_2\}$.

\medskip

If $\{h_1,v_1\},\{h_1,v_2\}$ were removed in \Fref{phase:cycle}~or~\Fref{phase:selfintersect} in one step, then the edge by which $Z'$ is extended in the same step satisfies the~\ref{item:internal}\textsuperscript{rd} property of hyperguards for $Z'$ and $h_0$. The same holds for $\{h_2,v_1\},\{h_2,v_2\}$.
In both cases, we may skip to the next slice to be processed.

\medskip

Without loss of generality, we may suppose that $h_1$ is farther away from the root of the horizontal $R$-tree than $h_2$.

\medskip

If $\{\{h_1,v_1\},\{h_1,v_2\}\}\cap V(X_j)$ is non-empty, take the path component $P_j$ of $X_j$ containing this set; otherwise let $P_j$ be the empty graph.
Observe, that \Fref{claim:DG} implies that as a result of \Fref{phase:selfintersect}, for any node $e\in V(P)$, its horizontal slice $e\cap M_H$ is at least as far away from the root as $h_1$.

\medskip

Split the path $P_j$ into two components $P_{j,1}$ and $P_{j,2}$ by deleting $\{h_1,v_1\}$ and $\{h_1,v_2\}$ (if one of them is not in $E(X_j)$, then one of the components is empty), so that $\{h_1,v_1\}\notin V(P_{j,2})$ and $\{h_1,v_2\}\notin V(P_{j,1})$.

\begin{itemize}
	\item
	      If $|V(P_{j,1})|\not\equiv 0\pmod 3$ or $|V(P_{j,2})|\not\equiv 0\pmod 3$, then let $Y_j$ be a dominating set of $P_j$ which contains $\{h_1,v_1\}$ or $\{h_1,v_2\}$, and is minimal with respect to these conditions. Set
	      \begin{align*}
		      Z'_{j+1}   & =Y_j\bigcup Z'_j, \\
		      X_{j+1}    & =X_j-P_j,         \\
		      w_{j+1}(e) & =\left\{
		      \begin{array}{ll}
			      0,\quad     & \text{if }e\in V(P_j),    \\
			      w_j(e)\quad & \text{if }e\notin V(P_j).
		      \end{array}
		      \right.
	      \end{align*}
	      Clearly, one of $\{h_1,v_1\}$ and $\{h_1,v_2\}$ is contained in $Y_j\subset Z'_{j+1}\subseteq Z'$, and it satisfies the~\ref{item:internal}\textsuperscript{rd} property of hyperguards for $Z'$ and $h_0$.
	\item
	      If $|V(P_{j,1})|\equiv |V(P_{j,2})|\equiv 0\pmod 3$, then let $Y_j$ be a minimal dominating set of $P_j$. Moreover, if $\{\{h_2,v_1\},\{h_2,v_2\}\}\bigcap (V(X_j)\bigcup Z'_j)$ is non-empty, let $f_j$ be an element of it, otherwise set $f_j=\{h_2,v_1\}$.
	      Take
	      \begin{align*}
		      Z'_{j+1}   & =Y_j\bigcup \{f_j\}\bigcup Z'_j,   \\
		      X_{j+1}    & =X_j-P_j-\{f_j\}-N_{X_j}(\{f_j\}), \\
		      w_{j+1}(e) & =\left\{
		      \begin{array}{ll}
			      0,\quad        & \text{if }e\in V(P_j)\bigcup \big\{\{h_2,v_1\},\{h_2,v_2\}\big\}, \\
			      w_j(e)+1,\quad & \text{if }\mathrm{dist}_{X}(e,f_j)=2,                             \\
			      w_j(e)\quad    & \text{otherwise.}
		      \end{array}
		      \right.
	      \end{align*}
	      Observe, that $f_j$ satisfies the~\ref{item:internal}\textsuperscript{rd} property of hyperguards for $Z'$ and $h_0$.
\end{itemize}

In any case, some element of $Z'_{j+1}\subseteq Z'$ satisfies the~\ref{item:internal}\textsuperscript{rd} property of hyperguards for $Z'$ and $h_0$. Furthermore, this holds for any slice of $A'_H$ between $h_1$ and $h_2$, so we skip processing these elements.

\medskip

\phase{}\label{phase:path} Lastly, we get $X_{j_3}$ which is the disjoint union of paths and isolated nodes (or it is an empty graph). Take a component $P_j$ of $X_j$ (for some $j\ge j_3$). Let $Y_j$ be a dominating set of $P_j$ (if $|V(P_j)|=1$, then $Y_j=V(P_j)$). Take
\begin{align*}
	Z'_{j+1}   & =Y_j\bigcup Z'_j, \\
	X_{j+1}    & =X_j-P_j,         \\
	w_{j+1}(e) & =\left\{
	\begin{array}{ll}
		0,\quad     & \text{if }e\in V(P_j),    \\
		w_j(e)\quad & \text{if }e\notin V(P_j).
	\end{array}
	\right.
\end{align*}

\medskip

By repeating this procedure, eventually $X_{j_4}$ is the empty graph for some $j_4\ge j_3$.

\medskip

Let $Z'=Z'_{j_4}$. This procedure is orchestrated in a way to guarantee that $Z'$ is a hyperguard of $M'$, so only an upper estimate on the cardinality of $\tau(Z')$ needs to be calculated to complete the proof of \Fref{case:2conn}.

\subsubsection{Estimating the size of \texorpdfstring{\boldmath $Z=\tau(Z')$}{𝑍=𝜏(𝑍')}.}
We have
\begin{align*}
	|V(X_0)| =r'+s',\ w_0(X) =s',\ |B'_H|=|T'|+|U'|.
\end{align*}

By definition, $|Z'_1|=c'+|U'|+2|T'|$ and $ |\tau(Z'_1)|=|Z'_1|-|B'_H|$. It is easy to check that
\[ |V(X_1)|+w_1(X)+2|U'|+5|T'|\le |V(X_0)|+w_0(X). \]
Therefore, we have
\begin{align}
	|Z'_1| & + \frac{|V(X_1)|+w_1(X)}{3}\le c'+|U'|+2|T'|+\frac{|V(X_1)|+w_1(X)}{3}\le \nonumber \\
	       & \le c'+|B'_H|+\frac{|V(X_0)|+ w_0(X)-2|U'|-2|T'|}{3}\le \label{ineq:Z1}             \\
	       & \le c'+|B'_H|+\frac{r'+2s'-2|B'_H|}{3}.\nonumber
\end{align}

We now show that
\begin{align}
	|Z'_{j+1}| & +\frac{|V(X_{j+1})|+w_{j+1}(X)}{3}\le
	|Z'_{j}|+\frac{|V(X_{j})|+w_{j}(X)}{3}.\label{ineq:recursion}
\end{align}
holds for any $j\ge 1$.

\medskip

In \Fref{phase:cycle}, we choose a node from each cycle of $X_1$. \Fref{ineq:recursion} is preserved, since
\begin{align*}
	|Z'_{j+1}|   & =|Z'_j|+1,                                  \\
	|V(X_{j+1})| & =|V(X_j)|-5+\mathds{1}_{\{4\}}(k_j),        \\
	w_{j+1}(X)   & \le w_j(X)+2-2\cdot\mathds{1}_{\{4\}}(k_j).
\end{align*}

In \Fref{phase:selfintersect}, for every $j_2>j\ge j_1$, we have
\begin{align*}
	|Z'_{j+1}|   & =|Z'_j|+1,       \\
	|V(X_{j+1})| & =|V(X_{j_1})|-5, \\
	w_{j+1}(X)   & \le w_{j}(X)+2.
\end{align*}

Let $j_3>j\ge j_2$. If $|V(P_{j,1})|\not\equiv 0\pmod 3$ and $|V(P_{j,2})|\not\equiv 2\pmod 3$, then take a dominating set of $P_j$ containing $\{v_1,h_1\}$. We have
\begin{align*}
	|Y_j| & \le 1+\left\lceil\frac{|V(P_{j,1})|-2}{3}\right\rceil+\left\lceil\frac{|V(P_{j,2})|-1}{3}\right\rceil \le \\
	      & \le 1+\frac{|V(P_{j,1})|-1}{3}+\frac{|V(P_{j,2})|}{3}= \frac{|V(P_j)|+2}{3}.
\end{align*}
Similarly, if $|V(P_{j,1})|\not\equiv 2\pmod 3$ and $|V(P_{j,2})|\not\equiv 0\pmod 3$, then there is a small dominating set of $P_j$ containing $\{h_1,v_2\}$. Also, if both $|V(P_{j,1})|\equiv 2\pmod 3$ and $|V(P_{j,2})|\equiv 2\pmod 3$, then there is a small dominating set of $P_j$ containing $\{h_1,v_2\}$.
Thus, if $|V(P_{j,1})|\not\equiv 0\pmod 3$ or $|V(P_{j,2})|\not\equiv 0\pmod 3$, then
\begin{align*}
	|Z'_{j+1}|   & =|Z'_j|+|Y_j|\le |Z'_j|+\frac{|V(P_j)|+2}{3}, \\
	|V(X_{j+1})| & =|V(X_{j_1})|-|V(P_j)|,                       \\
	w_{j+1}(X)   & \le w_{j}(X)-2.
\end{align*}

\medskip

If both $|V(P_{j,1})|\equiv 0\pmod 3$ and $|V(P_{j,2})|\equiv 0\pmod 3$, then
$|Y_j|=\frac{|V(P_j)|}{3}$. Observe, that
\[ \{h_1,v_1\},\{h_1,v_2\},\{h_2,v_1\},\{h_2,v_2\}\notin V(P_k)\text{ for any }k<j. \]
If both $\{h_1,v_1\}\notin Z'_j$ and $\{h_1,v_2\}\notin Z'_j$, but were removed in different steps, then when $\{h_1,v_1\}$ is removed in step $k$ we must have set $w_k(\{h_1,v_2\})=1$, which is the consequence of the previous observation. Thus, $w_j(\{h_1,v_2\})=1$.
Similarly, we must have $w_j(\{h_1,v_1\})=1$.
This reasoning holds for $\{h_2,v_1\}$ and $\{h_2,v_2\}$, as well.

\medskip

If $P_j$ is not the empty graph or $f_j\in Z(X_j)$, then \fref{ineq:recursion} trivially holds.
If $P_j$ is the empty graph, then $w_j(\{h_1,v_1\})=w_j(\{h_1,v_2\})=1$. If $f_j\in V(X_j)$, these 2 extra weights can be used to compensate for the new degree 1 vertices of $X_{j+1}$. If $f_j\notin Z(X_j)\bigcup V(X_j)$, then even $w_j(\{h_2,v_1\})=w_j(\{h_2,v_2\})=1$, and in total the 4 extra weights compensate for adding $f_j$ to $Z'_{j+1}$.

\medskip

In any case, \fref{ineq:recursion} holds for $j_3>j\ge j_2$.

\medskip

For any $j_4>j\ge j_3$, we have
\[ |Y_j|\le \left\lceil\frac{|V(P_j)|}{3}\right\rceil\le \frac{|V(P_j)|+2}{3} \] and $w_j(P_j)=2$, so \fref{ineq:recursion} holds for $j$.

\subsubsection{Summing it all up.}
By definition, we have
\[ |Z'|=|Z'_{j_4}|,\ X_{j_4}=\emptyset,\ 0\le w_{j_4}(X). \]
\Fref{ineq:recursion} is preserved from \Fref{phase:cycle} up to \Fref{phase:path}, therefore
\[ |Z'|\le |Z'_{j_4}|+\frac{|V(X_{j_4})|+w_{j_4}(X)}{3}\le |Z'_1|+\frac{|V(X_1)|+w_1(X)}{3}. \]

Lastly, using \fref{ineq:Z1}, we get
\begin{align*}
	  & |Z|=|\tau(Z')|=|\tau(Z'\setminus Z'_1)|+|\tau(Z'_1)|\le |Z'\setminus Z'_1|+|Z'_1|-|B'_H|= \\
	  & =|Z'|-|B'_H|\le c'+\frac{r'+2s'-2|B'_H|}{3}=c'+\frac{(c'-4)+2s'-2|B'_H|}{3}=              \\
	  & =\frac{4\left(c'+\tfrac12 s'\right)-4-2|B'_H|}{3}=\frac{4|V(M')|-4-2|B'_H|}{3}=                             \\
	  & =\frac{4|M'_H|+4|M'_V|-4-2|B'_H|}{3}= \frac{4|M_H|+4|B_H|+4|M_V|-4-2|B'_H|}{3}=           \\
	  & =\frac{4(|M_H|+|M_V|)-4}{3},\text{ as desired.}
\end{align*}

\subsection{\texorpdfstring{\boldmath $M$}{𝑀} is connected, but not 2-connected}\label{case:conn}

Let the 2-connected components (or blocks) of $M$ be $M_i$ for $i=1,\ldots,q$. Since induced graphs of $G$ inherit the chordal bipartite property, by~\Fref{case:2conn}, there exists a subset $Z_i\subseteq E(M_i)$, such that for any edge $e_0\in E(G[N_G(M_i)])$, there exists an edge $e_1\in Z_i$ which has $r$-vision of $e_0$ in $G[N(M_i)]$, and $|Z_i|\le \frac43 (|V(M_i)|-1)$. Let $Z=\cup_{i=1}^q Z_i$.

\medskip

Since the intersection graph of the vertex sets of the 2-connected components is a tree (and any two components intersect in zero or one elements), we have
\[ |Z|\le \frac43 \left(-q + \sum_{i=1}^q |V(M_i)| \right)= \frac{4\left(-q + |V(M)| + (q-1)\right)}{3}=\frac{4(|V(M)|-1)}{3}. \]

\medskip

Furthermore, given an arbitrary $e_0=\{v_0,h_0\}\in E(G)$, there exists a $v_1\in M_V$ and an $h_1\in M_H$ such that $\{v_1,h_0\},\{v_0,h_1\}\in E(G)$.
\begin{itemize}
	\item If $v_0\in M_V$ or $h_0\in M_H$, then $\{v_0,h_1\}$ or $\{v_1,h_0\}$ is in $E(M)$.

	\item Otherwise, there exists a path in $M$ whose endpoints are $v_1$ and $h_1$, and this path and the edges $\{v_1,h_0\}$,$\{h_0,v_0\}$,$\{v_0,h_1\}$ form a cycle in $G$.
	By the bichordality of $G$, there exists a $C_4$ in $G$ which contains an edge of $M$ and $e_0$.
\end{itemize}
In any case, $e_0$ is $r$-visible from some $e_1\in E(M)$. As $e_1$ is an edge of one of the 2-connected components $M_i$, we have $e_0\subset N_G(M_i)$, therefore $e_0\in E(G[N_G(M_i)])$. Thus, some $e_2\in Z_i$ has $r$-vision of $e_0$.

\subsection{\texorpdfstring{\boldmath $M$}{𝑀} has more than one connected component.}

Let us take a decomposition of $M$ into connected components $M_i$ for $i=1,\ldots,t$.

\medskip

Let $N_i=N(M_i)$, so we have $M_i\subseteq N_i$ and $\cup_{i=1}^t N_i=V(G)$.

\medskip

For all $i>1$ let $q_i$ be the number of components of $G[\cup_{k=1}^{i-1} N_k\setminus\cup_{k=i}^{t}N_k]$ to which $N_i\setminus\cup_{k=i+1}^{t}N_k$ is joined in $G[\cup_{k=1}^i N_k\setminus\cup_{k=i+1}^{t}N_k]$.
Let $F_{i,j}$ be the set of edges joining $N_i\setminus\cup_{k=i+1}^{t}N_k$ to the $j^{th}$ component of $G[\cup_{k=1}^{i-1} N_k\setminus\cup_{k=i}^{t}N_k]$.
Furthermore, let $F_{i,j}^V=\{f\in F_{i,j}\ |\ f\cap A_V\cap N_i\neq \emptyset\}$ and $F_{i,j}^H=\{f\in F_{i,j}\ |\ f\cap A_H\cap N_i\neq \emptyset\}$.

\begin{claim}\label{claim:f1}
	For any two edges $f_1,f_2\in F_{i,j}^V$ either $f_1\cap f_2\neq\emptyset$ or $\exists f_3\in F_{i,j}^V$ such that $f_3$ intersects both $f_1$ and $f_2$. The analogous statement holds for $F_{i,j}^H$.
\end{claim}
\begin{proof}
	Suppose $f_1$ and $f_2$ are disjoint. Since $M_i$ is connected, there is a path in $G$ whose endpoints are $f_1\cap N_i$ and $f_2\cap N_i$, while its internal points are in $V(M_i)$; let the shortest such path be $Q_1$. There is also a path in the $j^{th}$ component of $G[\cup_{k=1}^{i-1}N_k\setminus \cup_{k=i}^t N_k]$ whose endpoints are $f_1\setminus N_1$ and $f_2\setminus N_i$, let the shortest one be $Q_2$.

	\medskip

	Now $Q_1,f_1,Q_2,f_2$ form a cycle in $G[\cup_{k=1}^i N_k\setminus \cup_{k=i+1}^t N_k]$, which is bipartite chordal. Since $V(Q_2)\cap N_i=\emptyset$, there cannot be a chord between $V(M_i)\cap V(Q_1)$ and $V(Q_2)$. This implies that $|V(Q_1)|=3$ by its choice, and that either $(f_1\cap N_i)\cup(f_2\setminus N_i)$ or $(f_2\cap N_i)\cup (f_1\setminus N_i)$ is a chord.
\end{proof}

\begin{claim}\label{claim:f2}
	For any two edges $f^V\in F_{i,j}^V$ and $f^H\in F_{i,j}^H$, the two-element set
	\[ (f^V\cap N_i)\cup (f^H\cap N_i)\]
	is an edge of $G[N_i]$.
\end{claim}
\begin{proof}
	Similar to the proof of \Fref{claim:f1}.
\end{proof}

Let $f_{i,j}^V\in F_{i,j}^V$ be the element which intersects the maximum number of edges from $F_{i,j}$, and choose $f_{i,j}^H\in F_{i,j}^H$ in the same way. If only one of these exist, let $w_{i,j}$ be the existing one, otherwise let $w_{i,j}=(f_{i,j}^V\cap N_i)\cup (f_{i,j}^H\cap N_i)$ (as in \Fref{claim:f2}). Let us finally define
\[ W=\{w_{i,j}\ |\ i=2,\ldots,t\text{ and }j=1,\ldots,q_i\}.\]

\begin{claim}\label{claim:Wsize}
	$|W|=t-1$.
\end{claim}
\begin{proof}
	Observe that for every $i=1,\ldots,t$, the subgraph $G[N_i\setminus\cup_{k=i+1}^{t}N_k]$ is connected, since $M_i\subseteq N_i\setminus\cup_{k=i+1}^{t}N_k\subseteq N_i=N(M_i)$. Moreover, $G[\cup_{k=1}^{t}N_k]=G$ is connected, therefore $t-1=\sum_{i=2}^t q_i=|W|$.
\end{proof}

By~\Fref{case:conn}, there exists a subset $Z_i\subseteq E(M_i)$, such that for any edge $e_0\in E(G[N_i])$ there exists an edge $e_1\in Z_i$ which has $r$-vision of $e_0$ in $G[N_i]$, and $|Z_i|\le \frac43 (|V(M_i)|-1)$.

Let $Z=W\cup \left(\cup_{i=1}^t Z_i\right)$. An easy calculation gives that
\begin{align*}
	|Z| & \le (t-1)+\sum_{i=1}^t \frac{4|V(M_i)|-4}{3}\le \frac{4|V(M)|-4t+3(t-1)}{3}\le \\
	    & \le \frac{4(|M_H|+|M_V|-1)}{3}.
\end{align*}

Take an arbitrary edge $e_0=\{v_0,h_0\}\in E(G)$. We have three cases.

\begin{enumerate}
	\item
	      If $e_0\in F_{i,j}^V$ for some $i,j$, then we claim that $f_{i,j}^V\cap e_0\neq\emptyset$. Suppose not; by \Fref{claim:f1} there exists $f\in F_{i,j}^V$ which intersects both $e_0$ and $f_{i,j}^V$. For any edge $e\in F_{i,j}^V$ intersecting $f_{i,j}^V$ it either intersects $f$ too, or there is an edge intersecting both $e$ and $f$.
	      Thus, $f$ intersects at least as many edges as $f_{i,j}^V$, plus it intersects $e_0$ too, which contradicts the choice of $f_{i,j}^V$.

	      \medskip

	      If $w_i=f_{i,j}^V$, then $w_i$ trivially has $r$-vision of $e_0$. If both $f_{i,j}^V$ and $f_{i,j}^H$ exist, we have two cases.
	      \begin{itemize}
		      \item If $v_0\in f_{i,j}^V$, then $v_0\in w_i$ too, so $w_i$ has $r$-vision of $e_0$.

		      \item If $h_0\in f_{i,j}^V$, then \Fref{claim:f2} yields that $\{v_0\}\cup (f_{i,j}^H\cap N_i)\in E(G)$.
		            Thus, $\Big\{\{v_0,h_0\},f_{i,j}^V,w_i,\{v_0\}\cup (f_{i,j}^H\cap N_i)\Big\}$ is the edge set of a $C_4$ in $G$, so $w_i$ has $r$-vision of $e_0$.
	      \end{itemize}

	\item If $e_0\in F_{i,j}^H$ for some $i,j$, the same argument as above gives that $w_{i,j}$ has $r$-vision of $e_0$.

	\item If neither of the previous two cases holds, then $e_0\in E(G[N_i])$ for some $i$, so some element of $Z_i$ has $r$-vision of it.
\end{enumerate}

Thus, $Z$ satisfies \Fref{thm:mainprime}, and the proof is complete.

\chapter{Algorithms and complexity}\label{chap:artcomplexity}

In this chapter, we develop algorithms based on the main theorems of the previous chapters. Also, the computational complexity of art gallery problems is discussed. It turns out that our algorithms are efficient, although a considerable amount of preparation and review of literature is necessary. Given their efficiency, our algorithms have the potential to be applicable in practice, too.

\medskip

To achieve linear running times, we will rely on \citeauthor{MR1115104}'s triangulation algorithm. Although the algorithm of \citet{MR1148949} runs only in $O(n\log\log n)$, it is not nearly as complex, which may be preferable in real world applications. Nonetheless, our algorithms run in linear time, given a triangulated input.

\medskip

As data structures, we represent polygons by the (cyclically) ordered, doubly linked lists of their vertices.

\begin{theorem}[\citet{MR1115104}]\label{thm:chazelle}
  An $n$-vertex polygon can be triangulated in $O(n)$ time.
\end{theorem}

A common subroutine in our art gallery algorithms is the construction of $R$-trees (\Fref{sec:treestructure}).

\begin{proposition}[{\citet[Section 5]{GyH}}]\label{prop:Rtree}
	The horizontal (vertical) $R$-tree of an orthogonal polygon $P$ can be constructed in linear time.
\end{proposition}
\begin{proof}
  \Fref{algo:Rtree} produces the $R$-tree. The main part of this, \Fref{algo:horizontalcuts1}, is an adaptation of the (in my opinion incomplete) algorithm of \citet[Algorithm~4b]{fournier1984triangulating}.
  Suppose we are at the beginning of a loop (\Fref{step:beginloop} in \Fref{algo:horizontalcuts2}), and let $D$ be the set of sides of $P$ that are contained in triangles already deleted from $T$.
  For each diagonal $d$ drawn by the triangulation, in $S[d]$ we store a $y$-coordinate decreasing list of (the whole of or a segment of) the vertical sides of $P$ in $D$ seen (via horizontal vision) by $d$.

  \medskip

  Almost every step of the algorithm in the main loop (starting at \Fref{step:beginloop}) can be executed in $O(1)$. The following steps require further consideration.
  \begin{itemize}
    \item \Fref{step:process1} requires $O(|S[s_v]|+|S[s_w]|)$ time. After this step, their elements are discarded.
    \item Steps~\ref{step:project2}-\ref{step:process2} can be computed in $O(|R|+|S[s_w]|)$. The elements in $R$ and $S[s_w]$ are discarded.
    \item A similar result holds for Steps~\ref{step:project3}-\ref{step:process3}.
    \item Lastly, the for-loop starting following Step~\ref{step:sideofP} completes $|S[s_u]|$ cycles, after which the contents of $S[s_u]$ are discarded.
 \end{itemize}
  At any point in the algorithm, let $W$ be the set of sides of $\cup T$. Observe, that
  \[\bigcup\limits_{d\in W}S[d] \]
  contains every original side of the polygon at most twice. Every element is processed at most once by each of the previous highlighted steps, therefore \Fref{algo:horizontalcuts1} runs in $O(n)$.

  \medskip

  The rest of \Fref{algo:Rtree} is straightforward. \Fref{step:splitpolygon} runs in $O(1)$ as we store the orthogonal polygon pieces in doubly linked lists. Therefore \Fref{algo:Rtree} runs in linear time.
\end{proof}

\section{Partitioning orthogonal polygons}

\citet{Liou1989} proved that an $n$-vertex rectilinear domains can be partitioned into the minimum number of rectangles in $O(n)$ (assuming the preprocessing step uses linear time triangulation \Fref{thm:chazelle}). An $O(n\log n)$ algorithm by \citet{Lopez1996} produces an optimal partition of an $n$-vertex rectilinear domain into at most 6-vertex rectilinear domains. Both algorithms have important applications to computer graphics, image processing, and automated VLSI layout designs.
%For rectilinear domains with holes, there is an $O()$ %% in Keil 1995

\medskip

However, beyond these cases, not much is known about the complexity of determining a minimum size partition of a rectilinear domain into at most $2k$-vertex rectilinear domains, when $k\ge 4$. Instead of insisting on finding a minimum size partition, we present an efficient algorithm to find a partition whose cardinality is not greater than the extremal optimum.

\begin{theorem}
	An $n$-vertex orthogonal polygon can be partitioned into at most $\lfloor\frac{3 n +4}{16}\rfloor$ orthogonal polygons of at most 8 vertices in linear time.
\end{theorem}
\begin{proof}
The proof \Fref{thm:mobile} describes a recursive algorithm. Use \Fref{algo:Rtree} to construct the horizontal $R$-tree of $P$, such that the edge list of a vertex is ordered by the $x$ coordinates of the corresponding cuts. Furthermore, compute the list of pockets, corridors, and special rectangles of \Fref{case:toporbotcontained}. These structures can be maintained in $O(1)$ for the partitions after finding and performing a cut in $O(1)$.
\end{proof}

The art gallery problem corresponding to partitioning into at most 8-vertex pieces is the~\textsc{MSC} problem (see~\Fref{table:slidingcameras}). The \textsc{NP}-hardness of finding a minimum cardinality partition of an orthogonal polygon into at most 8-vertex pieces is an open problem, similarly to the \textsc{MSC} problem discussed in the next section.

\section{Finding guard systems}

\begin{table}[ht]
	\centering
	\bgroup%
	\def\arraystretch{1.5}
	\begin{tabular}{ r  l  }
		\toprule
		\multicolumn{2}{c}{\textsc{Minimum cardinality Sliding Cameras} problem (\textsc{MSC})} \\
		\midrule
		\textit{Input} & An orthogonal polygon $P$\\
    \cmidrule(lr){2-2}
		\multirow{2}{*}{\textit{Feasible solution}} & A set $M$ of vertical and horizontal mobile $r$-guards \\ & covering the domain enclosed by $P$\\
    \cmidrule(lr){2-2}
		\textit{Objective} & Minimize $|M|$ \\
    \cmidrule(lr){2-2}
    \textit{Decision version} & Given $(P,k)$, decide whether $\min |M|\le k$. \\
		\bottomrule
		& \\
		\toprule
		\multicolumn{2}{c}{\textsc{Minimum cardinality Horizontal Sliding Cameras} problem (\textsc{MHSC})} \\
		\midrule
		\textit{Input} & An orthogonal polygon $P$\\
    \cmidrule(lr){2-2}
		\multirow{2}{*}{\textit{Feasible solution}} & A set $M_H$ of \textbf{horizontal} mobile $r$-guards \\
		& covering the domain enclosed by $P$\\
    \cmidrule(lr){2-2}
		\textit{Objective} & Minimize $|M_H|$ \\
    \cmidrule(lr){2-2}
    \textit{Decision version} & Given $(P,k)$, decide whether $\min |M_H|\le k$. \\
		\bottomrule
	\end{tabular}

	\egroup%

	\bigskip

	\caption{Definitions of the sliding camera problems (\textsc{MHSC} and \textsc{MSC})}\label{table:slidingcameras}
\end{table}

\begin{table}[ht]
	\centering
	\bgroup%
	\def\arraystretch{1.5}
	\begin{tabular}{ r  l  }
		\toprule
		\multicolumn{2}{c}{\textsc{Dominating Set} problem} \\
		\midrule
		\textit{Input} & A simple graph $G$\\
    \cmidrule(lr){2-2}
		\multirow{2}{*}{\textit{Feasible solution}} & A set $S\subseteq V(G)$ which satisfies \\
		& $\forall v\in V(G)\setminus S\quad\exists u\in S\quad\text{such that}\quad\{u,v\}\in E(G)$ \\
    \cmidrule(lr){2-2}
		\textit{Objective} & Minimize $|S|$ \\
    \cmidrule(lr){2-2}
    \textit{Decision version} & Given $(G,k)$, decide whether $\min |S|\le k$. \\
		\bottomrule
		& \\
		\toprule
		\multicolumn{2}{c}{\textsc{Total Dominating Set} problem} \\
		\midrule
		\textit{Input} & A simple graph $G$\\
    \cmidrule(lr){2-2}
		\multirow{2}{*}{\textit{Feasible solution}} & A set $S\subseteq V(G)$ which satisfies \\
		& $\forall v\in V(G)\quad\exists u\in S\quad\text{such that}\quad\{u,v\}\in E(G)$ \\
    \cmidrule(lr){2-2}
		\textit{Objective} & Minimize $|S|$ \\
    \cmidrule(lr){2-2}
    \textit{Decision version} & Given $(G,k)$, decide whether $\min |S|\le k$. \\
		\bottomrule
	\end{tabular}

	\egroup%

	\bigskip

	\caption{Definitions of the \textsc{Total Dominating Set} and \textsc{Dominating Set} problem}\label{table:domination}
\end{table}

Finding a minimum cardinality horizontal mobile $r$-guard system, which is also known as the \textsc{Minimum cardinality Horizontal Sliding Cameras} or \textsc{MHSC} problem (\Fref{table:slidingcameras}),
is known to be polynomial~\cite{KM11} in orthogonal polygons without holes. In orthogonal polygons with holes, the problem is \textsc{NP}-hard as shown by \citet{BCLMMV16}. In their paper, a polynomial time constant factor approximation algorithm for the \textsc{MHSC} problem is described, too.
As explained in \Fref{sec:translating}, the \textsc{MHSC} problem translates to the \textsc{Total Dominating Set} problem (\Fref{table:domination}) in the pixelation graph (\Fref{sec:translating}), which can be solved in polynomial time for chordal bipartite graphs~\cite{DMK90}.

\medskip

Finding a minimum cardinality mixed vertical and horizontal
mobile $r$-guard system (also known as the \textsc{Minimum cardinality Sliding Cameras} or \textsc{MSC} problem) has been shown by \citet{DM13} to be \textsc{NP}-hard for orthogonal polygons with holes. For orthogonal polygons without holes, the problem translates to the \textsc{Dominating Set} problem in the pixelation graph. This reduction in itself has little use, as \citet{MB87} have shown that \textsc{Dominating Set} is \textsc{NP}-complete even in chordal bipartite graphs. To our knowledge, the complexity of \textsc{MSC} is still an open question. There is, however, a polynomial time 3-approximation algorithm by \citet{KM11} for the \textsc{MSC} problem for orthogonal polygons without holes. In case holes are allowed,~\cite{BCLMMV16} give a polynomial time constant factor approximation algorithm.

\medskip

The algorithm for the \textsc{MHSC} problem in~\cite{KM11} relies on a polynomial algorithm solving the \textsc{Clique Cover} problem in chordal graphs. Our analysis of the $R$-tree structures and the pixelation graph allows us to reduce the polynomial running time to linear.

\begin{theorem}[\citet{GyM2017}]\label{thm:mgalg}
 \Fref{algo:mhsc} finds a solution to the \textsc{MHSC} problem in linear time.
\end{theorem}
\begin{proof}
	By \Fref{prop:Rtree}, both the horizontal $R$-tree $T_H$ and the vertical $R$-tree $T_V$ of $D$ can be constructed in linear time.

	\medskip

	The main idea of the algorithm is to only sparsely construct the pixelation graph $G$ of $D$. Observe, that the neighborhood of a vertical slice in $G$ is a path in $T_H$, and vice versa. Label each horizontal edge of $D$ by the horizontal slice that contains it. Furthermore, label each vertical edge of each horizontal slice by the edge of $D$ containing it; do this for the horizontal edges of vertical slices as well. This step also takes linear time. The endpoints of a path induced by the neighborhood of any node in $G$ can be identified via these labels in $O(1)$ time.

	\medskip

	In \Fref{sec:translating}, we showed that a horizontal guard system is a subset of $V(T_H)$ which intersects (covers) each element of $\mathcal{F}_H=\{N_G(v)\ |\ v\in V(T_V)\}$. Dirac's theorem~\cite[p.~10]{frank_dopt} states that $\nu$, the maximum number of disjoint subtrees of the family, is equal to $\tau$, the minimum number of nodes covering each subtree of the family. Obviously, $\nu\le \tau$. The other direction is proved using a greedy algorithm:
	\begin{enumerate}
		\item Choose an arbitrary node $r$ of $T_H$ to serve as its root. The distance of a vertical slice $v\in V(T_V)$ from $r$ is $\mathrm{dist}_r(v)=\min_{h\in N_G(v)}\mathrm{dist}(h,r)$, and let $h_r(v)=\arg\min_{h\in N_G(v)}\mathrm{dist}(h,r)$.
		\item Enumerate the elements of $V(T_V)$ in decreasing order of their distance from $r$, let $v_1,v_2,\ldots,v_{|V(T_V)|}$ be such an indexing. Let $S_0=\emptyset$.
		\item If $N_G(v_i)$ is disjoint from the elements of $\{N_G(v)\ |\ v\in S_{i-1}\}$, let $S_i=S_{i-1} \cup \{ v_i\}$; otherwise let $S_i=S_{i-1}$.
	\end{enumerate}
	We claim that $\{h_r(v)\ |\ v\in S_{|V(T_V)|}\}$ is a cover of $\mathcal{F}_H$. Suppose there exists $v_j\in V(T_V)$ such that $N_G(v_j)$ is not covered. Let $i$ be the smallest index such that $v_i\in S_i$ and $N_G(v_j)\cap N_G(v_i)\neq\emptyset$. Clearly, $i<j$, therefore $\mathrm{dist}_r(v_i)\ge \mathrm{dist}_r(v_j)$. However, this means that $h_r(v_i)\in N_G(v_j)$.

	\medskip

	Now $\{h_r(v)\ |\ v\in S_{|V(T_V)|}\}$ is a cover of the same cardinality as the disjoint set system $\{N_G(v)\ |\ v\in S_{|V(T_V)|}\}$, proving that $\nu=\tau$.

	\medskip

	Each neighborhood $N_G(v)$ for $v\in V(T_V)$ induces a path in $T_H$. Therefore, the first part of the algorithm, including calculating $\mathrm{dist}_r(v)$ and $h_r(v)$ for each $v$, can be performed in $O(n)$ time, using the off-line lowest common ancestors algorithm of \citet{MR801823}.

	\medskip

	Calculating the distance decreasing order takes linear time via breadth-first search started from the root. In the $i^\mathrm{th}$ step of the third part of the algorithm, we maintain for each node in $V(T_H)$ whether it is under an element of $\{h_r(v)\ |\ v\in S_i\}$. Summed up for the $|V(T_H)|$ steps, this takes only linear time. $N_G(v_{i+1})$ is disjoint from the elements of $\{N_G(v)\ |\ v\in S_i\}$ if and only if one of the ends of the path induced by $N_G(v_{i+1})$ is under one of the elements of $\{h_r(v)\ |\ v\in S_i\}$, which now can be checked in constant time. Thus, the algorithm takes in total some constant factor times the size of the input time to run.
\end{proof}

\begin{table}[ht]
	\centering
	\bgroup%
	\def\arraystretch{1.5}
	\begin{tabular}{ r  l  }
		\toprule
		\multicolumn{2}{c}{\textsc{Point guard} problem} \\
		\midrule
		\textit{Input} & An orthogonal polygon $P$ (which bounds the domain $D$)\\
    \cmidrule(lr){2-2}
		\multirow{2}{*}{\textit{Feasible solution}} & A subset of points $X\subset D$ satisfying \\
    & $\forall y\in D$ there exists $x\in X$ such that $\overline{xy}\subset D$. \\
    \cmidrule(lr){2-2}
		\textit{Objective} & Minimize $|X|$ \\
    \cmidrule(lr){2-2}
    \textit{Decision version} & Given $(P,k)$, decide whether $\min |X|\le k$. \\
		\bottomrule
		& \\
    \toprule
		\multicolumn{2}{c}{\textsc{Point $r$-guard} problem} \\
		\midrule
		\textit{Input} & An orthogonal polygon $P$ (which bounds the domain $D$)\\
    \cmidrule(lr){2-2}
		\multirow{2}{*}{\textit{Feasible solution}} & A subset of points $X\subset D$ satisfying \\
    & $\forall y\in D$ there exists $x\in X$ such that $x$ has $r$-vision of $y$ in $D$. \\
    \cmidrule(lr){2-2}
		\textit{Objective} & Minimize $|X|$ \\
    \cmidrule(lr){2-2}
    \textit{Decision version} & Given $(P,k)$, decide whether $\min |X|\le k$. \\
		\bottomrule
	\end{tabular}

	\egroup%

	\bigskip

	\caption{Definitions of the \textsc{Point ($r$-)guard} problem}\label{table:pointguards}
\end{table}

The computational complexity of the \textsc{Point guard} problem (see \Fref{table:pointguards}) in orthogonal polygons with or without holes has attracted significant interest since the inception of the problem. \citet{MR1330860} showed that even for orthogonal polygons (without holes), \textsc{Point guard} is \textsc{NP}-hard. However, a minimum cardinality \textsc{Point $r$-guard} system  of an orthogonal polygon can be computed in $\tilde O(n^{17})$ time~\cite{WM07}. To our knowledge, the exponent of the running time is still in the double digits, which makes its use impractical. Therefore, approximate solutions to the problem are still relevant. A linear-time 3-approximation algorithm is described in~\cite{LWZ12}.

\begin{corollary}
	An $\frac83$-approximation of the minimum size of a point guard system of an orthogonal polygon can be computed in linear time.
\end{corollary}
\begin{proof}
	Compute $m_V$ and $m_H$ using the previous algorithm. By \Fref{thm:main} and the trivial statement that both $m_H\le p$ and $m_V\le p$, we get that $\frac43\cdot(m_H+m_V)$ is an $\frac83$-approximation for $p$.
\end{proof}

Unfortunately, we can only compute the corresponding solution (guard system) in $O(n^2)$, because the pixelation graph may have $\Omega(n^2)$ edges. I consider it an interesting open problem to reduce this running time to linear as well, so that it matches the algorithm of~\cite{LWZ12}.

\part{Terminal-pairability (edge-disjoint path problem)}\label{part:terminals}
\tikzset{every picture/.style={line cap=none, thick} }

\chapter{The terminal-pairability problem}

\section{Problem statement and origins}
We discuss the graph theoretic concept of \textbf{terminal-pairability} emerging from a practical networking problem introduced by
\citet{MR1189290}, further studied by
Faudree, Gyárfás, and Lehel~\cites{MR1208923,MR1167462,MR1677781} and \citet{MR1985088}.
Given a simple undirected graph $G=(V(G),E(G))$ and an undirected multigraph $D=(V(D), E(D))$  on the same vertex set ($V(D)=V(G))$, we say that $D$ can be \textbf{realized} in $G$ iff there exist edge-disjoint paths $P_1,\ldots,P_{|E(D)|}$ in $G$ such that $P_i$ joins that endpoints of $e_i\in E(D)$ for any $i=1,2,\ldots,|E(D)|$. We call $D$ and its edges the \textbf{demand graph} and the \textbf{demand edges} of $G$, respectively.
Given $G$ and a family $\mathcal{F}$ of (demand)graphs defined on $V(G)$ we call $G$ \textbf{terminal-pairable} with respect to $\mathcal{F}$ if every demand graph in $\mathcal{F}$ can be realized in $G$.

\medskip

In particular, let $|V(G)|$ be even and let $\mathcal{M}$ consist of all perfect matchings of the complete graph on $|V(G)|$ vertices; we call $G$ a \textbf{path-pairable} graph if it is terminal-pairable with respect to $\mathcal{M}$. See \Fref{fig:tpexamples} for an example of a terminal- and a path-pairability problem and their corresponding solutions.

\medskip

The way the terminal-pairability problem was originally formulated considered an instance of the problem to consist of $G$ (the graph of the internal vertices) and a prescribed number of pairwise distinct terminal vertices attached to each vertex of $G$. This graph extended with stars is called terminal-pairable if any matching of the terminal vertices has a realization. In our language, the family of demand graphs is of the form
\[ \left\{ D\ :\ V(D)=V(G),\  \forall v\in V(G)\ d_D(v)\le c(v)\right\},\]
where $c:V(G)\to \mathbb{N}$ describes the number of terminals incident to a vertex of $G$. We call the process of substituting the demand edges by disjoint paths in $G$ a \textbf{realization} of the demand graph.

\medskip

% A graph $G$ on $2k$ vertices is called {\itshape path-pairable\/} if,
% for any two disjoint subsets $X = \{x_1,\dots,x_k\}$ and $Y = \{y_1,\dots,y_k\}$ of the vertices of $G$, there exist $k$
% edge-disjoint $x_iy_i$-paths. The vertices of the set $X\cup Y$ are often called {\itshape terminals\/} while the pairs
% $(x_i,y_i)$ of terminals are simply called {\itshape pairs\/}.
%
% A more general pairability concept is {\itshape terminal-pairability\/}. Let $G$ be a graph with maximum degree $\Delta$ and with vertex set $V(G) = T(G)\cup I(G)$ where the set $T(G)$ consists of an even number of vertices of degree 1. We call $G$ a {\itshape terminal-pairable\/} network if for any pairing of the vertices of $T(G)$ there exist edge-disjoint paths in $G$ between the paired vertices. $T(G)$ is referred to as the set of {\itshape terminal nodes\/} or {\itshape terminals\/} and $I(G)$ is called the set of interior nodes of the network.  For an inner vertex $v$ we denote the number of terminal vertices incident to $v$ by $d_{T(G)}(v)$. Let $H$ be a graph and let $G$ be the graph obtained by adding a terminal at each vertex of $H$
% ($d_{T(G)}(v)$ = 1, for every $v\in V(H))$. Observe that $H$ is path-pairable if and only if $G$ is terminal-pairable.

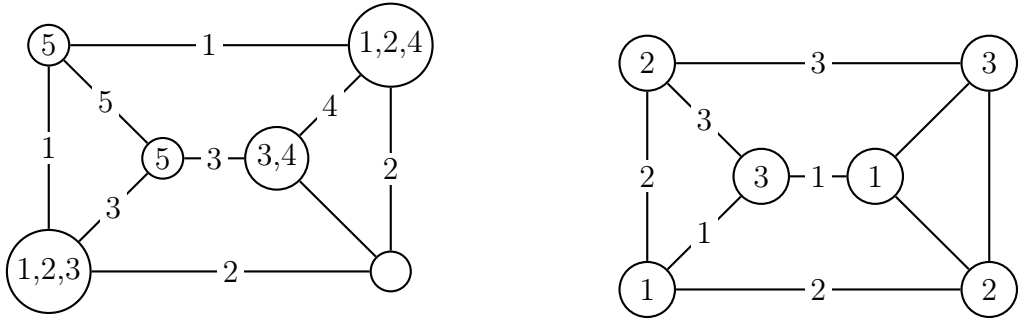
\begin{figure}
	\centering
	\begin{subfigure}{.5\textwidth}
		\centering

		\begin{tikzpicture}[inner sep=2pt]
			\def \l {3cm}

			\fill[white] (-0.25*\l,-0.25*\l) rectangle (1.75*\l,1.25*\l);

			\node[draw,circle] (0) at (0,0) {1,2,3};
			\node[draw,circle] (1) at (0,\l) {5};
			\node[draw,circle] (2) at (1.5*\l,\l) {1,2,4};
			\node[draw,circle,minimum size=15pt] (3) at (1.5*\l,0) {};

			\node[draw,circle] (4) at (0.5*\l,0.5*\l) {5};
			\node[draw,circle] (5) at (1.0*\l,0.5*\l) {3,4};

			\foreach \x/\y/\z in {0/1/1,1/2/1,0/3/2,3/2/2,0/4/3,4/5/3,5/2/4,4/1/5}
			\draw (\x) edge node[fill=white]{\z} (\y);

			\foreach \x/\y/\z in {5/3}
			\draw (\x) edge  (\y);

		\end{tikzpicture}

		\caption{A terminal-pairability problem and its solution}\label{fig1:tp}
	\end{subfigure}%
	\begin{subfigure}{.5\textwidth}
		\centering

		\begin{tikzpicture}
			\def \l {3cm}

			\fill[white] (-0.25*\l,-0.25*\l) rectangle (1.75*\l,1.25*\l);

			\node[draw,circle] (0) at (0,0) {1};
			\node[draw,circle] (1) at (0,\l) {2};
			\node[draw,circle] (2) at (1.5*\l,\l) {3};
			\node[draw,circle] (3) at (1.5*\l,0) {2};

			\node[draw,circle] (4) at (0.5*\l,0.5*\l) {3};
			\node[draw,circle] (5) at (1.0*\l,0.5*\l) {1};

			\foreach \x/\y/\z in {0/1/2,0/3/2,0/4/1,4/5/1,4/1/3,1/2/3}
			\draw (\x) edge node[fill=white, inner sep=2pt]{\z} (\y);

			\foreach \x/\y/\z in {2/3,2/5,5/3}
			\draw (\x) edge  (\y);

		\end{tikzpicture}

		\caption{A path-pairability problem and its solution}\label{fig1:pp}
	\end{subfigure}
	\caption{Examples for terminal-pairability and path-pairability problems. Each  vertex is labeled with the demand edges to which it is incident to.}\label{fig:tpexamples}
\end{figure}

\medskip

Given a simple graph $G$, one central question in the topic of terminal-pairability is the maximum value of $q$ for which any demand graph in the set
\[ \left\{ D\ :\ V(D)=V(G),\  \forall v\in V(G)\ d_D(v)\le q\right\}\]
is realizable in $G$. As at a given vertex $v\in V(G)$ at most $d_{G}(v)$ edge-disjoint paths can start, the minimum degree $\delta(G)$ of the base graph provides an obvious upper bound on $q$. Often, a better upper bound on $q$ is obtained by choosing a $q$-regular multigraph $D$, such that except $e_0$ elements, edges of $D$ can only be resolved into paths of length at least $\ell$ in $G$. By the pigeonhole principle, we must have
\[ \left(\frac12\cdot q \cdot |V(G)|-e_0\right)\cdot \ell\le e(G)-e_0,\]
or equivalently,
\begin{align}
q\le \frac{\overline{d}(G)}{\ell}+\frac{2e_0\left(\ell-1\right)}{|V(G)|}.\label{ineq:pigeonhole}
\end{align}
An improvement on this method is described by \citet{Girao2017} for $G=K_n$.

\section{Outline of Part~\ref{part:terminals}}

In \Fref{chap:complete}, the case $G=K_n$ is studied in detail, where the best known lower bound on $q$ is determined (\Fref{thm:delta_n/3}). Using the techniques of this proof, the exact extremal edge number of a demand graph realizable in $K_n$ is obtained (\Fref{thm:alpha_2n-5}).

\medskip

On the other end of the spectrum, a relatively sparse graph which is still path-pairable is sought after in \Fref{chap:grid}. At the end of the chapter, some open problems are discussed. The results of \Fref{chap:complete}~and~\ref{chap:grid} are a result of a fruitful collaboration with my supervisor, Ervin Győri, and Gábor Mészáros.

\medskip

Lastly, in \Fref{chap:termcomplexity}, subjects related to terminal-pairability and the problem's complexity are surveyed.

\tikzset{every picture/.style={line cap=none, thick} }

\chapter{Complete base graphs}\label{chap:complete}

\citet{MR1189290} studied the extremal value of the maximum degree of the demand graph for the complete graph $K_n$ as the base graph and investigated the following question:

\begin{problem}[\cite{MR1189290}]\label{problem:tpcomplete} What is the highest number $q$ for which any demand graph on $n$ vertices and maximum degree $q$ is realizable in $K_n$?
\end{problem}
One can easily verify that the parameter $q$ cannot exceed $\frac{n}{2}$. Indeed, take a demand graph $D$ obtained by replacing every edge in a one-factor on $n$ vertices by $q$ parallel edges. In order to create edge-disjoint paths, most paths need to use at least two edges in $K_n$, thus \fref{ineq:pigeonhole} implies the indicated upper bound.

\medskip

The authors of~\cite{MR1189290} conjectured that if ${n\equiv 2\pmod{4}}$, then $q=\frac{n}{2}$. However, \citeauthor{Girao2017} showed, that asymptotically, $q/n$ is less than $\frac12$.
\begin{proposition}[\citet{Girao2017}]\label{prop:upperbound}
	If every $n$-vertex demand graph $D$ with $\Delta(D)\le q$ is realizable in $K_n$, then $q\leq \frac{13}{27}n + O(1)$.
\end{proposition}

Csaba, Faudree, Gyárfás, Lehel, and Shelp showed the following lower bound.

\begin{theorem}[\citet{MR1189290}]
	Any demand graph $D$ on $n$-vertices with $\Delta(D)\leq \frac{n}{4+2\sqrt{3}}$ is realizable in $K_n$.
\end{theorem}

We improve their result by proving the following theorem:

\begin{theorem}[\citet{terminalcomplete}]\label{thm:delta_n/3}
	Any demand graph $D$ on $n$-vertices with $\Delta(D)\leq 2\lfloor\frac{n}{6}\rfloor-4$ is realizable in $K_n$.
\end{theorem}

\citet{MR1985088} investigated terminal-pairability properties of the Cartesian product of complete graphs. In their paper, the following ``Clique-Lemma'' was proved and frequently used:
\begin{lemma}[\citet{MR1985088}]\label{lemma:clique}
	Let $D$ be an $n$-vertex demand graph, where $n\geq 5$. If $\Delta(D)\le n-1$ and $e(D)=n$, then $D$ is realizable in $K_n$.
\end{lemma}

In the same paper, the following related problem was raised about the possible strengthening of \Fref{lemma:clique}:
\begin{problem}[\cite{MR1985088}]
Let $D$ be an $n$-vertex demand graph such that $\Delta(D)\le n-1$. What is the largest value of $\alpha$ such that $e(D)\le\alpha\cdot n$ implies that $D$ is realizable in $K_n$?
\end{problem}
Obviously, $1\leq\alpha$ due to \Fref{lemma:clique}. It is also easy to see that $\alpha < 2$. Let $D$ be a demand graph on $n\ge 4$ vertices, in which two pairs of vertices, $u,v$ and $x,y$ are both joined by $(n-2)$ parallel edges ($d_D(w)=0$ for $w\not\in\{x,y,u,v\}$). Observe that to realize the demand graph, any disjoint path system must contain a path from $x$ to $y$ passing through $u$ or $v$.
However, there are also $n-2$ disjoint paths connecting $u$ and $v$, meaning that $u$ or $v$ is incident to at least $2+(n-2)=n$ disjoint edges, which is clearly a contradiction. This implies that the number of edges in $D$ cannot exceed $2n-5$. We show that this bound is sharp by proving the following theorem:

\begin{theorem}[\citet{terminalcomplete}]\label{thm:alpha_2n-5}
	Let $D$ be a demand graph on $n$ vertices
	with at most $2n-5$ edges, such that no vertex is incident to more than $n-1$ edges. Then $D$ has a realization in $K_n$. % with $O(n)$ edges.
\end{theorem}

Before the proofs, we fix further notation and terminology.
For convenience, we call a pair of edges joining the same two vertices a $C_2$. For $k>2$, $C_k$ denotes the cycle on $k$ vertices.
For a subset $S\subset V(G)$ of vertices let $e(S,V(G)-S)$ denote the number of edges with exactly one endpoint in $S$. Let $G[S]$ denote the subgraph induced in $G$ by the subset of vertices $S$. We call a pair of vertices joined by $k$ parallel edges a \textbf{\boldmath $k$-bundle}.

\medskip

%; unless stated otherwise, vertices of a $k$-bundle have no additional neighbors. %% pont fordítva van...
For a vertex $v$ we denote the set of neighbors by $N(v)$ and use $\gamma(v)=|N(v)|$. We define  the \textbf{multiplicity} $m(v)$ of a vertex $v$ as follows: $m(v)=d(v)-\gamma(v)$. Observe that $m(v)$ is the minimal number of edges incident to $v$ that need to be replaced by longer paths in a realization to guarantee an edge-disjoint path-system for the terminals of $v$.
%We also use the term for a set of $k$ parallel edges, meaning that we replace at least $k-1$ of them by longer paths joining the same end-vertices. %% az egyetlen hely, ahol ez használva van, szintén kifejti

\medskip

\begin{figure}
	\centering
	\begin{subfigure}{.5\textwidth}
		\centering

		\begin{tikzpicture}

			\node[draw, circle] (u) at (-1,0) {$u$};
			\node[draw, circle] (v) at (1,0) {$v$};
			\node[draw, circle] (w) at (2,2) {$w$};
			\node[draw, circle] (z) at (-2,2) {$z$};

			\foreach \x/\y/\b/\c in {u/v/30/dotted,u/v/0/dashed,u/v/-30/thin,v/w/-30/thin}
				{
					\draw[\c] (\x) edge[bend left=\b] (\y);
				}
		\end{tikzpicture}

		\caption{Before}
	\end{subfigure}%
	\begin{subfigure}{.5\textwidth}
		\centering

		\begin{tikzpicture}

			\node[draw, circle] (u) at (-1,0) {$u$};
			\node[draw, circle] (v) at (1,0) {$v$};
			\node[draw, circle] (w) at (2,2) {$w$};
			\node[draw, circle] (z) at (-2,2) {$z$};

			\foreach \x/\y/\b/\c in {u/z/0/dotted,z/v/30/dotted,u/w/30/dashed,w/v/0/dashed,u/v/-30/thin,v/w/-30/thin}
				{
					\draw[\c] (\x) edge[bend left=\b] (\y);
				}
		\end{tikzpicture}

		\caption{After}
	\end{subfigure}
	\caption{Lifting 2 edges of $uv$ to $z$ and $w$}\label{fig:lifting}
\end{figure}
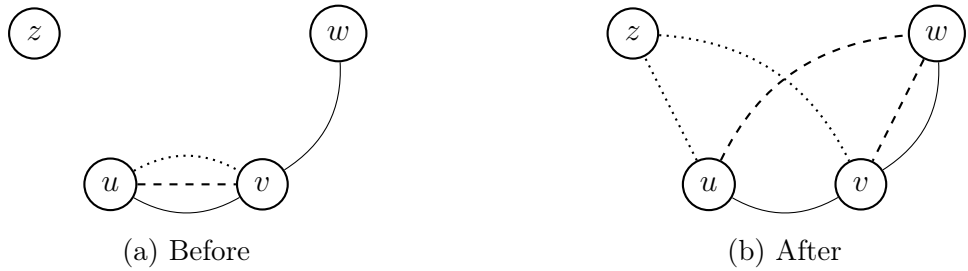

We define an operation that is repeatedly used in our proofs: given an edge ${uv}\in E$, we say that we \textbf{lift} ${uv}$ to a vertex $w$ when the edge ${uv}$ is substituted by ${uw}$ and ${wv}$ (forming a path of length 2). Note that this operation increases the degree of $w$ by 2, but does not affect the degree of any other vertex (including $u$ and $v$). Also, as a by-product of the operation, if $w$ is already joined by an edge to $u$ or $v$, the multiplicity of the appropriate pairs increase by one (see \Fref{fig:lifting}).

\medskip

Finally, note that if a graph $G$ has $n$ vertices and $d(v)\leq n-1$, all multiplicities of $v$ can be easily resolved by subsequent liftings.
Indeed, $v$ has $n-1-\gamma(v)$ non-neighbors and $m(v)=d(v)-\gamma(v)\leq n-1-\gamma(v)$ multiplicities, thus we can assign every edge of $v$ causing a multiplicity to a non-neighbor to which that particular edge can be lifted without creating new multiplicities at $v$. As a result, $d(v)$ does not change but $\gamma(v)$ becomes equal to $d(v)$. We call this the \textbf{resolution of the multiplicities} of $v$ (see \Fref{fig:resolve}).

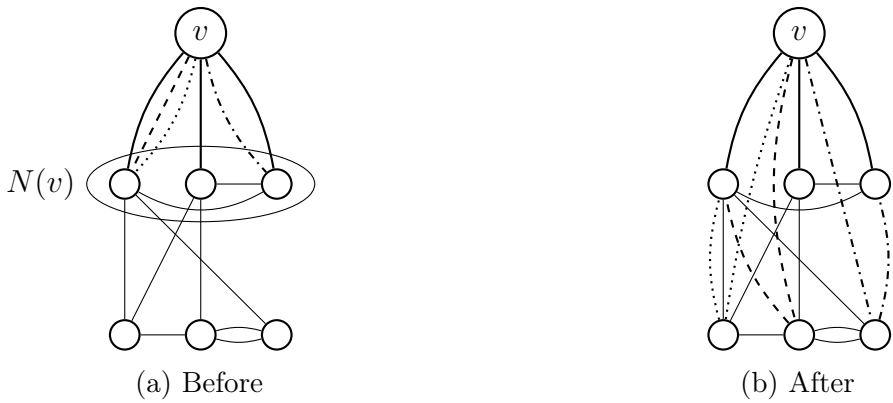
\begin{figure}
	\centering
	\begin{subfigure}{.5\textwidth}
		\centering

		\begin{tikzpicture}

			\node[draw, circle] (v) at (0,3) {$v$};

			\draw[thin] (0,1) ellipse (1.5 and 0.5);

			\foreach \x in {0,1,2}
			\node[draw, circle] (u\x) at (\x-1,1) {};

			\foreach \x in {0,1,2}
			\node[draw, circle] (w\x) at (\x-1,-1) {};

			\foreach \x/\b/\c in {0/15/dotted,0/0/dashed,0/-15/thick,1/0/thick,2/-15/dashdotted,2/15/thick}
				{
					\draw[\c] (v) edge[bend left=\b] (u\x);
				}

			\foreach \x/\y/\b in {u0/u2/-30,u1/u2/0,w1/w2/15,w1/w2/-15,u0/w0/0,w0/w1/0,u0/w2/0,w0/u1/0,u1/w1/0}
			\draw[thin] (\x) edge[bend left=\b] (\y);

			\node[anchor=east] (label_u) at (-1.5,1) {$N(v)$};
			\node[anchor=west,white] (label_u_invisible) at (1.5,1) {$N(v)$};

		\end{tikzpicture}

		\caption{Before}
	\end{subfigure}%
	\begin{subfigure}{.5\textwidth}
		\centering

		\begin{tikzpicture}

			\node[draw, circle] (v) at (0,3) {$v$};

			\foreach \x in {0,1,2}
			\node[draw, circle] (u\x) at (\x-1,1) {};

			\foreach \x in {0,1,2}
			\node[draw, circle] (w\x) at (\x-1,-1) {};

			\foreach \x/\b/\c in {0/-15/thick,1/0/thick,2/15/thick}
				{
					\draw[\c] (v) edge[bend left=\b] (u\x);
				}

			\foreach \x/\y/\b in {u0/u2/-30,u1/u2/0,w1/w2/15,w1/w2/-15,u0/w0/0,w0/w1/0,u0/w2/0,w0/u1/0,u1/w1/0}
			\draw[thin] (\x) edge[bend left=\b] (\y);

			\foreach \x/\y/\b/\c in {v/w0/-5/dotted,w0/u0/15/dotted,v/w1/-15/dashed,w1/u0/15/dashed,v/w2/0/dashdotted,w2/u2/-15/dashdotted}
				{
					\draw[\c] (\x) edge[bend left=\b] (\y);
				}

		\end{tikzpicture}
		\caption{After}
	\end{subfigure}
	\caption{Resolving the multiplicities at $v$}\label{fig:resolve}
\end{figure}

\section{Proof of \texorpdfstring{\Fref{thm:delta_n/3}}{Theorem~\ref{thm:delta_n/3}}}\label{sec:delta_n/3}

We show that if $D=(V,E)$ is a demand multigraph on $n$ vertices and $\Delta(G)\leq 2\lfloor\frac{n}{6}\rfloor-4$, then $D$ can be transformed into a simple graph by replacing parallel edges by paths of $D$. We prove the statement by induction on $n$. Observe first that the statement is obvious for $n < 18$. For $18\leq n < 24$, note that the demand graph $D$ is the disjoint union of 2-bundles, circles, paths, and isolated vertices (i.e., a 2-matching). It is easy to realize 2-matchings in $K_n$; the verification of the statement is left to the reader.

\medskip

From now on assume $n\geq 24$. We may assume without loss of generality that $D$ is an $\big(2\lfloor\frac{n}{6}\rfloor-4\big)$-regular multigraph; if necessary, additional parallel edges may be added to $D$. Should a single vertex $v$ fail to meet the degree requirement, we bump up its degree by further lifting operations as follows: as the deficit $\big(2\lfloor\frac{n}{6}\rfloor-4\big)-d(v)$ must be even, we can lift an arbitrary edge $e\in E([V(D)-v])$ to $v$. We remind the reader that lifting $e$ to $v$ increases $d(v)$ by two while it does not affect the degree of the rest of the vertices.

\medskip

We will use the well-known 2-Factor-Theorem of \citet{MR1554815}. Be aware that a 2-factor of a multigraph may contain several $C_2$'s (however, this is the only way parallel edges may appear in it).
\begin{theorem}[\citet{MR1554815}]\label{thm:petersen}
	Let $G$ be a $2k$-regular multigraph. Then $E(G)$ can be decomposed into the union of $k$ edge-disjoint $2$-factors.
\end{theorem}
Some operations, which are performed later in the proof, are featured in the following definition, claim, and lemma.
\begin{definition}[Lifting coloring]
	Let $F$ be a multigraph, and $c:E(F)\cup V(F)\to \{1,2,3\}$ be a coloring of the edges and vertices of $F$. We call $c$ a {\itshape lifting coloring\/} of $F$ if and only if
	\begin{enumerate}
		\item for any edge $e=uv\in E(F)$, $c(u)\neq c(e)$ and $c(v)\neq c(e)$, and
		\item for any two edges $e_1,e_2\in E(F)$ incident to a common vertex we have $c(e_1)\neq c(e_2)$.
	\end{enumerate}
	%(Lifting coloring is a total 3-coloring, but adjacent vertices can get the same color.)
	Moreover, if the number of vertices in different color classes differ by either 0, 1, or 2, then we call $c$ a {\itshape balanced lifting coloring\/} of $F$.
\end{definition}

\begin{claim}\label{claim:liftcolor}
	Let $F$ be a multigraph such that $\forall v\in V(F)$ we have $d_F(v)\le 2$. If $w_1,w_2,w_3\in V(F)$ are three pairwise non-adjacent different vertices, then $F$ has a balanced lifting coloring where $w_i$ gets color $i$.
\end{claim}
\begin{proof}
	The proof is easy but its complete presentation requires a rather lengthy (but straightforward) casework. We leave the verification of the statement to the reader. \Fref{fig:liftcoloring} shows an example output of this lemma.
	\begin{figure}
		\centering

		\begin{tikzpicture}
			\def \uno {thick}
			\def \due {dotted}
			\def \tre {dashed}
			\def \dotsize {5pt}

			\foreach \x/\c in {1/\uno,2/\due,3/\tre}
			\node[draw, circle, \c] (x\x) at (\x-2,3) {$x_\x$};

			%\draw[thick, dotted] (0,3) ellipse (2 and 0.75);

			\def \2fshift {1};

			\foreach \i/\c in {1/\uno,2/\due,3/\due}
			\node[shift={(-3,1)}, circle, draw, inner sep=\dotsize, \c] (b\i) at ({+90+90* \i}:1) {};

			\foreach \i/\c in {1/\due,2/\due,3/\uno,4/\tre}
			\node[shift={(0,0.5)}, circle, draw, inner sep=\dotsize, \c] (c\i) at ({+90+72* \i}:1) {};

			\foreach \i/\c in {1/\tre,2/\uno,3/\uno}
			\node[shift={(2,0.5)},draw, circle, inner sep=\dotsize, \c] (d\i) at (1,2.5-\i) {};

			\foreach \i/\j/\c in {b1/x1/\due,x1/b3/\tre,b3/b2/\uno,b2/b1/\tre}
			\draw (\i) edge[very thick,\c] (\j);

			\foreach \i/\j/\c in {c4/x2/\uno,x2/c1/\tre,c1/c2/\uno,c2/c3/\tre,c3/c4/\due}
			\draw (\i) edge[very thick,\c] (\j);

			\foreach \i/\j/\c in {d1/x3/\due,x3/d1/\uno,d2/d3/\due,d3/d2/\tre}
			\draw (\i) edge[very thick, bend left=15,\c] (\j);

		\end{tikzpicture}

		\caption{A balanced lifting coloring, where $x_1,x_2,x_3$ get pairwise different colors.}\label{fig:liftcoloring}
	\end{figure}
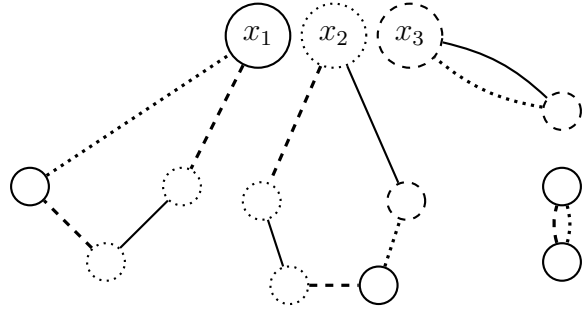
\end{proof}

\begin{lemma}\label{lemma:main}
	Let $D$ be a demand graph on $n$ vertices, such that $\Delta(D)\le \lfloor\frac{n}{3}\rfloor-4$.
	Furthermore, let $X=\{x_1,x_2,x_3\}$ be a subset of $V(D)$ of cardinality 3, such that $|E(D[X])|=0$.
	Let $B$ be an at most 3-element subset of $V(D)\setminus X$.
	Let $F$ be a $2$-matching of $D$, i.e., there are at most 2 edges of $F$ incident to any vertex of $D$.
	Moreover, either $d_F(x_i)=2$ or $d_D(x_i)\le \lfloor\frac{n}{3}\rfloor-5$ for $i=1,2,3$.
	Then there exists a demand graph $H$ which satisfies
	\begin{itemize}
		\item $V(H)=V(D)\setminus X$,
		\item $E(H)\supset E(D[V(H)])\setminus F$,
		\item $\{e\in E(H):\ e\text{ is incident to at least one of }B\}\subset E(D)$, and
		\item for any $v\in V(H)$ we have $d_H(v)\le d_D(v)-d_F(v)+\mathbb{1}(v\notin B)$.
	\end{itemize}
	Moreover, if $H$ has a realization in $K_{n-3}$, then $D$ is realizable in $K_n$.
\end{lemma}
\begin{proof}
	We will perform a series of liftings in $D$ in two phases, obtaining $D'$ and $D''$. At the end of the second phase, we will achieve that there are no incident parallel edges to $X$ in $D''$. Therefore, setting $H=D''-X$ will satisfy the second claim of the lemma.

	\medskip

	First, we determine the series of liftings to be executed in the first phase. Notice that \Fref{claim:liftcolor} implies the existence of a balanced lifting coloring $c$ of $F$ such that $c(x_i)\equiv i+1 \pmod 3$. Lift each edge $f\in F$ to $x_{c(f)}$, except if $f$ is incident to $x_{c(f)}$, then leave $f$ where it is. Let $F'$ be the set of lifted edges, that is
	\[ F'=\biguplus\limits_{\substack{f\in F,\\ x_{c(f)}\notin f}}\Big\{\text{the two edges joining $x_{c(f)}$ to the two vertices of $f$}\Big\}, \]
	where $\biguplus$ denotes the disjoint union. Let the multigraph $D'$ be defined on the same vertex set as $D$, and let its edge set be
	\[ E(D')=\{e\in E(D):\ e\notin F\text{ or }x_{c(e)}\in e\}\biguplus F'. \]
	In other words, $D'$ is the demand graph into which $D$ is transformed by lifting the elements of $F$. Let $Y=V(D)\setminus X$. Observe that $d_{D'}(y)=d_{D}(y)$ for $y\in Y$.	Let
	\[Y_i=\{y\in Y\setminus B\ |\ c(y)=i\}\]
	be the color $i$ vertices in $Y\setminus B$.  The balancedness of $c$ guarantees that
	\[ |Y_i|=|c^{-1}(i)\setminus X \setminus B|\ge |c^{-1}(i)|-1-|B|\ge \left\lfloor\frac{n}{3}\right\rfloor - 5. \]

	\medskip

	In the second phase, our task is to resolve all multiplicities of $x_i$ in $D'$. Observe that as edges of $F$ of the same color formed a matching, out of every two parallel edges that are incident to $x_i$ in $D'$, at least one is an initial edge in $E(D')\setminus F'$. The vertex $x_i$ is incident to $d_{D'}(x_i)-d_{F'}(x_i)$ edges of $E(D')\setminus F'$; we plan to lift these edges to elements of $Y_i$ by using every vertex in $Y_i$ for lifting at most once. If $d_F(x_i)=2$, then one of the two edges of $F$ incident to $x$ has color $i-1$, and this edge is lifted to $x_{i-1}$ in $D'$. Thus
	\[ d_{D'}(x_i)-d_{F'}(x_i)\le
		\left\{
		\begin{array}{lr}
			d_D(x_i)-1,                               & \text{ if }d_F(x_i)=2; \\
			\left\lfloor\frac{n}{3}\right\rfloor - 5, & \text{ if }d_F(x_i)<2; \\
		\end{array}
		\right\}
		\le |Y_i|.\]
	As elements of $c^{-1}(i)$ are not incident to edges of color $i$, the set $Y_i\setminus B$ offers enough space to carry out the liftings. That being said, note that neighbors of $x_i$ in $Y_i$ cannot be used for lifting as they would create additional multiplicities. On the other hand, if $v\in Y_i$ and  $e={vx_i}\in E(D)$, then $e$ is an initial edge of $x_i$ that either generates no multiplicity at all or is part of a bundle of parallel edges, one of which we do not lift. In other words, for every vertex of $Y_i$ that is excluded from the lifting we mark an initial edge of $x_i$ that we do not need to lift. Because of this, resolution of the remaining multiplicities at $x_i$ can be performed in $Y_i-N(x_i)$. Let $D''$ denote the demand graph obtained after resolving every multiplicity of $x_1$, $x_2$, and $x_3$.

	\medskip

	At most 1 element of $E(D')\setminus E(F')$ has been lifted to each $y\in Y$, therefore there are no multiple edges between the sets $X$ and $Y$ in the demand graph $D''$. Moreover, $D''[X]=D'[X]$ is a subgraph of a triangle, which emerges as we lift the at most one edge of color $i+2$ of $x_i$ to $x_{i+1}$ (take the indices cyclically), for $i=1,2,3$.

	\medskip

	Any vertex $y\in Y$ of color $i$ has at most two incident edges in $F'$, joining $y$ to a subset of $\{x_{i+1},x_{i+2}\}$.
	\begin{itemize}
		\item If an edge has been lifted to $y\in Y$ of color $i$, then $y$ is adjacent to $x_i$ and $d_{D''}(y)=d_{D'}(y)+2$. Thus, $y$ is joined to at least $d_{F'}(y)+1$ elements of $X$ in $D''$. As no edge of color $i$ can be incident to $y$, we have $d_{F'}(y)=d_F(y)$. Therefore
		\[ d_{D''[Y]}(y)\le d_{D''}(y)-d_{F'}(y)-1=d_{D'}(y)-d_{F}(y)+1=d_{D}(y)-d_{F}(y)+1.\]

		\item If no edges have been lifted to $y\in Y$, then $d_{D''}(y)=d_{D'}(y)$ and $y$ is adjacent to at least $d_{F}(y)$ elements of $X$ in $D''$. Therefore
		\[ d_{D''[Y]}(y)=d_{D''}(y)-d_F(y)\le d_{D'}(y)-d_F(y)=d_{D}(y)-d_F(y). \]
		%\item Lastly, if $y\in B$, then we made sure that no edges are lifted to it, but $y$ is only joined to one element of $X$. Therefore $$d_{D''[Y]}(y)=d_{D''}(y)-1=d_{D'}(y)-1=d_{D}(y)-1.$$
	\end{itemize}
	As elements of $B$ are excluded from $Y_i$, 0 edges are lifted to them, and so we proved the statement of the lemma.
\end{proof}

Let $X_1=\{x_1,x_2,x_3\}$ be a subset of 3 elements of $V(D)$, such that $D[X_1]$ has 0 edges. Such a set trivially exists, as any two non-adjacent vertices have $(n-2)-2\Delta(D)\ge \frac{n}{3}+2$ common non-neighbors. Since the degree in $D$ is at least $2\cdot(24/6)-4=4$, \Fref{thm:petersen} implies the existence of two disjoint 2-factors, $A_1$ and $A_2$ of $D$.
Notice that $A_2-X_1$ has 3 path components (as a special case, an isolated vertex is a path on one vertex).
Extend $A_2-X_1$ to a maximal $2$-matching $F_2$ of $D-X_1-A_1$. It is easy to see that there exists a 3-element subset $B_1$ of $V(D)\setminus X_1$ such that \begin{itemize}
	\item $B_1$ induces 0 edges in $D-A_1$,
	\item $\{v\in V(D)\setminus X_1:\ d_{F_2}(v)= 0\}\subset B_1$, and
	\item $B_2=\{v\in V(D)\setminus X_1:\ d_{F_2}(v)=1\}\setminus B_1$ has cardinality at most 3.
\end{itemize}

\medskip

We are ready to use \Fref{lemma:main}. First, apply it to $D$, where we lift $F=A_1$ to elements of $X=X_1$, while not creating new edges incident to $B=B_1$. Let the obtained graph be $H_1$. We have $\Delta(H_1)\le \Delta(D)-\delta(A_1)+1=\Delta(D)-1$. Furthermore, $E(H_1[B_1])\subseteq E(D[B_1])=\emptyset$, and for all $v\in B_1$ we have $d_{H_1}(v)\le \Delta(D)-\delta(A_1)\le \Delta(D)-2$.

\medskip

We apply \Fref{lemma:main} once again. Now $H_1$ is our base demand graph, $F_2$ is the $2$-matching to be lifted to elements of $B_1$, and we avoid lifting to elements of $B_2$. Let the resulting demand graph be $H_2$, whose vertex set is $V(D)\setminus X_1\setminus B_1$ of cardinality $n-6$. We have
\begin{align*}
	d_{H_2}(v) &\le\left\{
	\begin{array}{ll}
		d_{H_1}(v)-d_{F_2}(v)+1 & \text{ if }v\notin B_2, \\
		d_{H_1}(v)-d_{F_2}(v)   & \text{ if }v\in B_2.
	\end{array}
	\right\}\le \\
	&\le\left\{\begin{array}{ll}
		(\Delta(D)-1)-2+1 & \text{ if }v\notin B_2, \\
		(\Delta(D)-1)-1   & \text{ if }v\in B_2.
	\end{array}
	\right\}\le\Delta(D)-2=2\left\lfloor\frac{n-6}{6}\right\rfloor-4.
\end{align*}

\medskip

By induction on $n$, we know that $H_2$ is realizable in $K_{n-6}$, implying that $H_1$ is realizable in $K_{n-3}$, which in turn implies that $D$ has a realization in $K_n$.

%%%%%%%%%%%%%%%%%%%%%%%%%%%%%%%%%%%%%%%%%%%%%%%%%%%%%%%%%%%%%%%
%					Proof of Theorem alpha < 4 max.
%%%%%%%%%%%%%%%%%%%%%%%%%%%%%%%%%%%%%%%%%%%%%%%%%%%%%%%%%%%%%%%
\section{Proof of \texorpdfstring{\Fref{thm:alpha_2n-5}}{Theorem~\ref{thm:alpha_2n-5}}}\label{sec:alpha_2n-5}
We prove our statement by induction on $n$. For $n\leq 4$ the statement is straightforward, the cases $n=5,6$ require a somewhat cumbersome casework. Note that if $n\geq 4$, we may assume that $D$ has exactly $2n-5$ edges, otherwise we join two non-neighbors whose degree is smaller than $n-1$.

\medskip

For the inductive step, we choose a vertex $x$, resolve each of its multiplicities, and delete it from the demand graph. There are two additional conditions to assert as the number of vertices decreases from $n$ to $n-1$:
\begin{itemize}
	\item[i)] The number of edges of $D$ must decrease by at least two.
	\item[ii)] Every vertex of degree $n-1$ must lose at least one edge. Decreasing the degree $d(v)$ of a vertex $v$ can be achieved by lifting an edge incident to $v$ to $x$. Note that this operation might create additional multiplicities that need to be resolved before the deletion of $x$.
\end{itemize}
In addition, observe that we can lift at least one edge to a vertex $v$ without its degree exceeding the degree bound for $n'=n-1$ if and only if $d(v)< n-2$. Let
\[B=\{z_1,\dots,z_{|B|}\}=\{v\in V(D): d(v)\geq n-2\}.\]
As $\sum\limits_{v\in V(D)}d(v) = 4n-10$, it follows that $|B|\leq 3$. We perform a casework on $|B|$.

\medskip

\begin{description}
	\item[$|B|=0:$] If $B$ is empty, then the only condition we need to guarantee is the deletion of at least two edges in $D$.  We have two cases.
	      \begin{itemize}
					\item \textit{If $\forall x\in V(D)$ we have $\gamma(x)\le 1$,} then $D$ is the disjoint union of bundles and isolated vertices. Let $\{x,y\}$ be the edge with the highest multiplicity. If $d(x)=d(y)\le n-2$, and every other degree is at most $n-4$, then lift copies of $\{x,y\}$ to the other $n-2$ vertices and delete both $x$ and $y$; we can use induction on the remaining graph. If $D$ is composed of an $n-2$ and an $n-3$ bundle, it is easy to find a realization of $D$ in $K_n$ directly.
		      \item \textit{If there is an $x\in V(D)$ with $\gamma(x)\ge 2$:} we have $n-1-\gamma(x)$ vertices as a lifting target to resolve the $d(x)-\gamma(x)$ multiplicities of $x$. Obviously, $d(x)-\gamma(x)\leq n-3 -\gamma(x)$, thus we have enough space to resolve all multiplicities of $x$. After the deletion of $x$, the graph has $\gamma(x)\ge 2$ fewer edges, and the maximum degree is still two less than the number of vertices.
	      \end{itemize}

	\item[$|B|=1:$] We perform the same operation as in the previous case with the choice $x=z_1$. Observe that  our inequality becomes $d(z_1)-\gamma(z_1)\leq n - 1 -\gamma(z_1)$, thus we have enough vertices in the multigraph to perform all the necessary liftings.

	\item[$|B|=2:$] Observe first that $z_1$ and $z_2$ are joined by an edge $e$ or else \[2n-5\geq e(B,V(D)-B)= d(z_1)+d(z_2)\geq 2n-4,\] a contradiction. Let us first assume that $z_1$ or $z_2$ (say, $z_1$) has more than one neighbor (i.e., $e(B,V(D)-B)>0$). Observe that in this case
	\[ m(z_1)=d(z_1)-\gamma(z_1)\leq (n-1)-\gamma(z_1),\]
	thus, each multiplicity of $z_1$ can be resolved by lifting the appropriate edges to $V(D)-\{z_1\}-N(z_1)$.

	\medskip

  In the remaining case $z_1$ and $z_2$ form a bundle of at most $n-1$ edges. We can lift $n-2$ of these edges to $V(D)-B$ without difficulties, delete one of the vertices in $B$, and proceed by induction.

	\item[$|B|=3$:] Observe that any two vertices of $\{z_1,z_2,z_3\}$ must be joined by an edge, else the same reasoning as above leads to a contradiction. Note also, that a simple average degree calculation guarantees the existence of an isolated vertex $x$. We distinguish two cases:
	      \begin{itemize}
		      \item[i)] If $e(B,V(D)-B)=0$, we may assume that $V(D)-B$ induces an edge, otherwise $3(n-3)\geq 4n-10\Rightarrow n\leq 7$ and all edges are contained in $B$. For $n=5,6,7$, that leads to 4 possible demand graphs whose realization can easily be completed; a case for $n=6$ is shown in \Fref{fig:realization}.

		            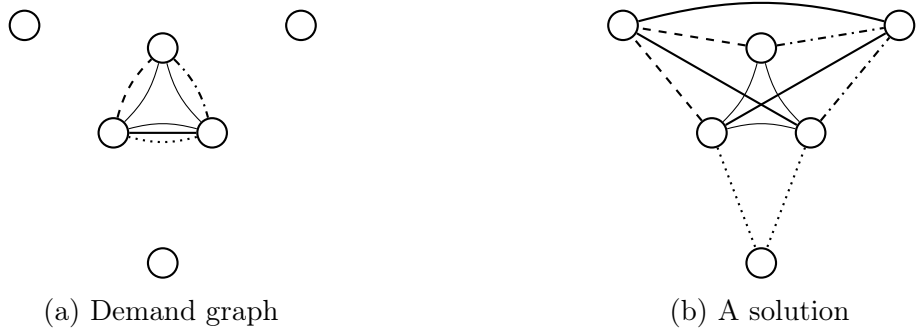
\begin{figure}
			            \centering
			            \begin{subfigure}{.5\textwidth}
				            \centering

				            \begin{tikzpicture}
					            \def \radius {1.5cm}

					            \foreach \s in {1,2,3}
						            {
							            \node[draw, circle] (\s) at ({90+120 * \s}:\radius/2) {};
							            \node[draw, circle] (O\s) at ({30+360/3 * \s}:1.4*\radius) {};
						            }
					            \foreach \x/\y/\c in {1/2/thin,2/1/dotted,1/3/dashed,3/1/thin,2/3/thin,3/2/dashdotted}
					            \draw[\c] (\x) edge[bend left=15] (\y);

					            \draw[thick] (1) edge (2);
					            \draw[color=white] (O1) edge[bend left=15] (O3);
				            \end{tikzpicture}

				            \caption{Demand graph}\label{fig1:before}
			            \end{subfigure}%
			            \begin{subfigure}{.5\textwidth}
				            \centering

				            \begin{tikzpicture}

					            \def \radius {1.5cm}

					            \foreach \s in {1,2,3}
						            {
							            \node[draw, circle] (\s) at ({90+120 * \s}:\radius/2) {};
							            \node[draw, circle] (O\s) at ({30+360/3 * \s}:1.4*\radius) {};
						            }

					            \foreach \x/\y in {1/2,2/3,3/1}
					            \draw[thin] (\x) edge[bend left=15] (\y);

					            \foreach \x/\y/\z/\c in {1/2/O2/dotted,2/3/O3/dashdotted,3/1/O1/dashed}
						            {
							            \draw[\c] (\x) edge (\z);
							            \draw[\c] (\y) edge (\z);
						            }

					            \draw[thick] (1) edge (O3);
					            \draw[thick] (2) edge (O1);
					            \draw[thick] (O1) edge[bend left=15] (O3);

				            \end{tikzpicture}

				            \caption{A solution}\label{fig1:after}
			            \end{subfigure}
			            \caption{A demand graph and a possible realization in $K_6$}\label{fig:realization}
		            \end{figure}

		            Let $f$ denote an arbitrary edge in $V(D)-B$. We lift two non-parallel edges of $B$ as well as $f$ to $x$; observe that the degrees of all vertices in $B$ dropped by at least 1. As $n\geq 7$, the multiple edge created at vertex $x$ can be lifted to a vertex of $V(D)-B$ that was not incident to $f$.

		      \item[ii)] If $e(B,V(D)-B) > 0$, let $f$ be an edge between $B$ and $V(D)-B$. Without loss of generality, we may assume $f$ is incident to $z_3$. We lift $f$ as well as an edge $e$ between $z_1$ and $z_2$ to $x$; as $e$ and $f$ are disjoint, no new multiplicity is created, thus we can delete $x$ and proceed by induction.
	      \end{itemize}
\end{description}

%\subsection{The number of edges in the realization is $O(n)$}
%
%Whenever we have the freedom to choose a vertex whose multiplicities are to be resolved, choose one with the maximum degree. Also, if possible, lift edges first to vertices with exactly 1 neighbor. Let $D_n=D$, and let $D_{n-i}$ be the demand graph produced by the $i^{th}$ recursion.
%
%\medskip
%
%Observe, that for any vertex $v\in V(D_i)$ we have $|d_{D_i}(v)-d_{D_{i+1}}(v)|\le 1$. Thus, there is a constant $k_0$ such that $\Delta(D_k)<\frac{k}{3}$.
%
%\medskip
%
%However, if $\Delta(D)<\frac{n}{3}$, let $v$ be a vertex of maximum degree.
%Take a set $W$ of $2\Delta(D)$ vertices that are non-adjacent to $v$. We claim that each edge $e=\{u,v\}$ of $v$ to be lifted can be matched to an element $w_e\in W$ such that $w_e\notin N(u)$.

\chapter{Complete grid base graphs}\label{chap:grid}

\section{Introduction}
A long-standing open question concerning path-pairability of graphs is the
minimal possible value of the maximum degree $\Delta(G)$ of a path-pairable
graph $G$. \citet{MR1677781} proved that the
maximum degree must grow together with the number of vertices in path-pairable
graphs.
They in fact showed that a path-pairable simple graph with maximum degree $\Delta$ has at most $2\Delta^\Delta$ vertices. The result yields that for a path-pairable simple graph $G$ on $n$ vertices we have
\[ c\cdot\frac{\log n}{\log \log n}\le \Delta(G).\]
This bound is conjectured to be asymptotically sharp, although to date only constructions of much higher order of magnitude have been found. Until recently, the best-known construction was due to \citet{MR1985088} who showed that two dimensional complete grids on an even number of vertices (of at least 6) are path-pairable.

\medskip

A two-dimensional complete grid is the \textbf{Cartesian product} $K_s\square K_t$ of two complete graphs $K_s$ and $K_t$ and it can be constructed by taking the Cartesian product of the sets $\{1,2,\dots\,s\}$ and $\{1,2,\dots\,t\}$ and joining two distinct vertices if they share a coordinate. Higher dimensional complete grids can be defined similarly: let $d,t_1,\ldots,t_d$ be positive integers and let $V$ denote the set of $d$-dimensional vectors of positive integer coordinates not exceeding $t_i$ in the $i$\textsuperscript{th} coordinate, that is,

\[ V_{(t_1,\ldots,t_d)}=\{(a_1,\dots,a_d): 1\leq a_i\leq t_i, i=1,2,\ldots,d\}. \]

The $d$-dimensional grid graph $K_{(t_1,\ldots,t_d)}$ is constructed by taking $V_{(t_1,\ldots,t_d)}$ as its vertex set, and two vertices are joined by an edge if the corresponding vectors differ at exactly one coordinate. Note that this graph is isomorphic to the Cartesian product $K_{t_1}\square K_{t_2}\square\ldots\square K_{t_d}$. For $t_1=t_2=\cdots=t_d=t$ we use the notation
\[ K_t^d=\overbrace{K_t\square\ldots\square K_t}^d.\]
For a more detailed introduction to the Cartesian product of graphs the reader is referred to~\cite{product}.

\medskip

With $s=t$ the construction of \citeauthor{MR1985088} gives examples of path-pairable graphs on $n=s\cdot t$ vertices with maximum degree $2\cdot\sqrt{n}$. This bound was recently improved to $\sqrt{n}$ by \citet{MR3518137}. It was also conjectured in~\cite{MR1985088} that $K_t\square K_t\square K_t$ is path-pairable for sufficiently large even values of $t$.

\medskip

In this chapter, we significantly improve the upper bound on the minimal value of the maximum degree of path-pairable graphs: we prove that high dimensional complete grids are path-pairable. We eventually study the more general terminal-pairability variant of the above path-pairability problem and prove the following theorem:

\begin{theorem}[\citet{MR3612429}]\label{thm:gridmain}
	Let $G=K_t^d $ and let $D=(V(D),E(D))$ be a demand graph with $V(D)=V(K_t^d)$ and $\Delta(D)\leq 2\lfloor\frac{t}{12}\rfloor- 2$. Then $D$ can be realized in $G$.
\end{theorem}

\Fref{thm:gridmain} immediately implies the following corollary:

\begin{corollary}\label{cor}
	If $t\geq 24$, $K_t^d$ is path-pairable.
\end{corollary}

The above construction provides examples of path-pairable graphs on $n=t^d$ vertices with maximum degree
\[ \Delta(K_t^d)=(t-1)\cdot d= \log n\cdot \frac{t-1}{\log t}.\]
Observe that $t$ can be chosen to be a constant ($t=24$) thus we have obtained path-pairable graphs on $n$ vertices with $\Delta\approx 7.24\log n$.

%%%%%%%%%%%%%%%%%%%%%%%%%%%%%%%%%%%%%%%%%%%%%%%%%%%%%%%%%%%%%%%
%							Proof
%%%%%%%%%%%%%%%%%%%%%%%%%%%%%%%%%%%%%%%%%%%%%%%%%%%%%%%%%%%%%%%
\section{Proof of \texorpdfstring{\Fref{thm:gridmain}}{Theorem~\ref{thm:gridmain}}}
For $i=1,\ldots,t$, let $L_i$ be the subgraph of $K_t^d$ induced by \[\left\{(a_1,\dots,a_{d-1},i): 1\leq a_j\leq t_j, j=1,2,\ldots,d-1\right\}.\]
We call $L_1,\dots,L_{t}$ the \textbf{layers} of $K_t^d$. Similarly, by fixing the first $d-1$ coordinates we get $t^{d-1}$ copies of $K_t$;
we denote these complete subgraphs by $l_1,\dots, l_{t^{d-1}}$ and refer to them as \textbf{columns}.

\medskip

Given an edge $uv$ of the demand graph with $u,v\in K_t^d$ we replace $uv$ by a path of three edges $uu'$, $u'v'$, and $v'v$, where $u',v'\in L_i$ for some $i\in\{1,2,\dots,t\}$ and $u,u'$ and $v,v'$ lie in the same columns.
Having done that, we consider the new demand edges defined within the $t$ layers and $t^{d-1}$ columns and break the initial problem into $t^{d-1}+t$ subproblems that we solve inductively. The upcoming paragraphs discuss the details of the drafted solution.

\medskip

For the discussion of the base case $d=1$ as well as for the inductive step we use \Fref{thm:delta_n/3}. We mention that instead of using this theorem we could use a weaker version of it with $\Delta(D)\leq \frac{t}{4+2\sqrt{3}}$ proved by \citet{MR1189290}. With every further step of our proof unchanged, a result corresponding to the weaker theorem can be proved with a smaller bound on $\Delta(D)$.

\medskip

Let $q$ be an even number with $2\leq q \leq \lfloor\frac{t}{6}\rfloor-1$ and let $D=(V(D),E(D))$ be a demand multigraph with $V(D)=K_t^d$ and $\Delta(D)\leq q$. Let $E'(D)$ denote the set of demand edges whose vertices lie in different $l_i, l_j$ columns. We construct an auxiliary graph $H$ with $V(H)=V(K_t^{d-1})$ and project every edge of $E'(D)$ into $H$ by deleting the last coordinates of the vertices. It is easy to see that $\Delta(H)\leq t\cdot q$.

\medskip

We may assume without loss of generality that $D$ is $(t\cdot q)$-regular by joining additional pairs of vertices or replacing edges by paths of length two if necessary. Again, we use the 2-factor decomposition theorem of Petersen (\Fref{thm:petersen}) to distribute the original demand edges among the layers $L_1,\dots, L_t$ and define new subproblems on them.

\medskip

Obviously, the graph $H$ satisfies the conditions of \Fref{thm:petersen}, thus $E(H)$ can be partitioned into $\frac{q}{2}\cdot t$ edge-disjoint 2-factors. By arbitrarily grouping the above 2-factors into $\frac{q}{2}$-tuples, we can partition $E(H)$ into $t$ edge-disjoint subgraphs $H_1,\dots,H_t$ with $\Delta(H_i)\leq q$.

\medskip

Assume now that the vertices $u=(\underline{a},i)$ and $v=(\underline{b},j)$  ($a,b\in {[t]}^{d-1}$) are joined by a demand edge belonging to $E'(D)$ (thus $\underline{a}\neq \underline{b}$) and the corresponding edge in $H$ is contained by $H_k$. We then replace the demand edge $uv$ by the following triple of newly established demand edges:
\[
\left\{(\underline{a},i),(\underline{a},k)\right\}, \left\{(\underline{a},k),(\underline{b},k)\right\},
\left\{(\underline{b},k),(\underline{b},j)\right\}.\]
We claim the following:
\begin{itemize}
	\item[\textit{(i)}] For  every layer $L_j$ the condition $\Delta(L_j)\leq q$ holds.
	\item[\textit{(ii)}] For every column $l_j$ the condition $\Delta(l_j)\leq 2q$ holds.
\end{itemize}
The first statement obviously follows from the partition of $E'(D)$. For the second one, observe that a vertex $v$ in $l_j$ is initially incident to $q$ demand edges and at most $q$ additional demand edges are joined to it (otherwise \textit{(i)} is violated). Notice now that every layer $L_j$ contains a $(d-1)$-dimensional subproblem that can be solved (within the layer) by the inductive hypothesis. Also, every column $l_j$ contains a subproblem (note that the original demand edges in $E(D)\backslash E'(D)$ are incorporated into these subproblems) that can be solved by \Fref{thm:delta_n/3}. This completes our proof.

\section{Remarks}

By using \Fref{thm:delta_n/3} and the described inductive approach we proved that $K_t^d$ is path-pairable for $t\geq 24$, $d\in\mathbb{Z}^+$.
Even if the bound in \Fref{thm:delta_n/3} is improved to the point of being sharp, it only improves the constant $5.2$ in \Fref{thm:gridmain} and decreases the lower bound on $t$ in Corollary~\ref{cor}, however, it does not affect the $\Omega(\log n)$ order of magnitude of the maximum degree.

\medskip

Let us measure the sharpness of our theorem. Assume $d\ge 2$ and that $t$ is even. Take the following matching $M$ of the vertices of $V(K_t^d)$: pair vertex $(x_1,\ldots,x_d)$ with $(t+1-x_1,t+1-x_2,\ldots,t+1-x_d)$. Let $D=(V(K_t^d),q\cdot M)$, i.e., every edge in $M$ is taken with multiplicity $q$. Since the vertices of each edge are different in each coordinate, \fref{ineq:pigeonhole} becomes
\[\frac12\cdot t^d\cdot q\cdot d \le \frac12\cdot d\cdot (t-1)\cdot t^d, \]
implying that $q \le (t-1)$. This means that there is at most a factor of $(6+\varepsilon_t)$ between the extremal bound and the result of \Fref{thm:gridmain}.

\medskip

We mention that one particularly interesting and promising path-pairable candidate (with the same order of magnitude of vertices but with a better constant for $\Delta$) is the $d$-dimensional hypercube $Q_d$ on $2^d$ vertices ($\Delta(Q_d) = d$). Observe that hypercubes are special members of the above studied complete grid family as $Q_d =K_2^d$.
Although it is known that $Q_d$ is not path-pairable
for even values of $d$ (see~\cite{MR1167462}), the question is open for odd dimensional
hypercubes for $d\geq 5$ ($Q_1$ and $Q_3$ are both path-pairable).

\begin{conjecture}[\cite{MR1189290}]
	The $(2k+1)$-dimensional hypercube $Q_{2k+1}$ is path-pairable for all
	$k\in\mathbb{N}$.
\end{conjecture}

\chapter{Related subjects, algorithms and complexity}\label{chap:termcomplexity}

\section{Immersions}

Recently, the study of graph immersions has become increasingly popular. As the following definition shows, it is very closely related to the concepts of terminal-pairability.
\begin{definition}
	Let $H$ and $G$ be {(multi)}graphs. There is an immersion of $H$ in $G$ (or $H$ is immersed in $G$, or $G$ contains an immersion of $H$) iff there is map $\phi:V(H)\cup E(H)\to V(G)\cup E(G)$ such that
	\vspace{-6pt}
	\begin{itemize}
		\item $\phi$ maps vertices of $H$ into distinct vertices of $G$,
		\item a loop on a vertex $u$ is mapped to a cycle of $G$ which traverses $\phi(u)$,
		\item an edge $uv$ of $H$ is mapped to a path connecting $\phi(u)$ to $\phi(v)$ in $G$, and
		\item for two distinct edges $e_1$ and $e_2$ in $H$, their images $\phi(e_1)$ and $\phi(e_2)$ are edge-disjoint.
	\end{itemize}
	This relation is denoted as $H\le_{\mathrm{i}} G$. If in addition, for any vertex $v\in V(H)$ and edge $e\in E(H)$ we have $v\notin e\implies \phi(v)\notin V(\phi(e))$, then $H$ has a strong immersion in $G$, denoted $H\le_{\mathrm{si}} G$.
\end{definition}

\begin{remark}
	In the study of graph immersions, the lifting operation is defined as the inverse of our lifting operation.
\end{remark}

In the terminal-pairability problem, the function mapping vertices of $H$ to vertices of $G$ is fixed. However, if $G=K_n$ and $H$ is loopless, the notion of realizable and immersible coincide.

\medskip

The transitive property of the immersion relations means that both $\le_{\mathrm{i}}$ and $\le_{\mathrm{si}}$ define partial orders on the set of finite graphs. The following fundamental result about graph immersions was conjectured by Nash-Williams.
\begin{theorem}[\citet{MR2595703}]\label{thm:RobertsonSeymour}
	In any infinite sequence of graphs ${(G_i)}_{i=1}^\infty$, there exists a pair $i<j$ such that $G_i\le_{\mathrm{i}} G_j$.
\end{theorem}
This property of the finite graphs with the order $\le_{\mathrm{i}}$ is called \textbf{well-partial-ordered}. An immediate consequence of the theorem is that any subset of graphs that is upward-closed w.r.t.~$\le_{\mathrm{i}}$ has finitely many minimal elements. Similarly, any property which is downward-closed on immersion order can be described by finitely many graphs. This is a useful property, since for every fixed graph $H$ there is a polynomial time algorithm to check whether there exists an immersion of $H$ in the input graph $G$, see~\cite{MR1146902}.

\medskip

\citeauthor{MR2063516} have explored connections between the immersion order and graph colorings, and made the following conjecture.

\begin{conjecture}[\citet{MR2063516}]\label{conj:immersion}
	If $\chi(G)\ge n$, then $K_n\le_{\mathrm{i}} G$.
\end{conjecture}

This is the immersion analogue of Hajós's refuted conjecture (where immersion order is replaced with topological minor order), and Hadwiger's unsolved conjecture (where immersion order is replaced with minor order). Observe, that
\[ \mathcal{F}_n=\left\{ G\text{ finite graph}\ :\ \chi(G)<n\text{ and }\forall H\le_{\mathrm{i}}G\text{ satisfies }\chi(H)<n \right\}\]
is a downward-closed set w.r.t.~$\le_{\mathrm{i}}$. Thus, its complement is upward-closed, and therefore has finitely many minimal elements; these graphs are called \textbf{\boldmath $n$-immersion-critical}. It is easy to see that $K_n$ is $n$-immersion-critical (any graph properly immersed in $K_n$ has two vertices that are not joined by an edge). If true, \Fref{conj:immersion} would imply that $K_n$ is immersed in any graph that is not in $\mathcal{F}_n$, i.e., $K_n$ is the only $n$-immersion-critical graph. They also proved that

\begin{theorem}[\citet{MR2063516}]
	If $G$ is $n$-immersion-critical and $G\not\cong K_n$, then $G$ is $n$-edge-connected.
\end{theorem}

A trivial consequence of $n$-edge-connectivity is that $\delta(G)\ge n$.
The above ideas motivate the following problem, which is the dual of the terminal-pairability problem in complete graphs (Problem~\ref{problem:tpcomplete}).
\begin{problem}[\citet{MR2729363}]
Determine the minimum value of $f(n)$ such that any simple graph with minimum degree $f(n)$ contains an immersion of $K_n$.
\end{problem}

Clearly, $f(n)\ge n-1$. For small values of $n\le 7$, it has been verified in~\cite{MR2729363} that ${f(n)=n-1}$. However, a class of counterexamples to this equality have been constructed for $n\ge 8$ in~\cite{MR3231086}. The first bound proved on $f(n)$ is the following theorem.

\begin{theorem}[\citet{MR3223965}]
	If $H$ is a simple graph with $\delta(H)\ge 200n$, then $K_n\le_{\mathrm{si}}{H}$.
\end{theorem}
\citet{dvorak2015complete} claim to have improved the lower bound on the minimum degree to $11n+7$.

\medskip

\Fref{thm:delta_n/3} has the following alternative statement in the language of immersions.
\begin{corollary}
	If $H$ is a loopless multigraph on at most $n$ vertices with $\Delta(H)\le 2\lfloor\frac{n}{6}\rfloor-4$, then $H\le_{\mathrm{i}} K_n$.
\end{corollary}

\medskip

Combining our result with that of \citet{dvorak2015complete}, we get a sufficient condition (with fairly strong asymptotic consequences) on when a loopless multigraph immerses in a simple graph.

\begin{theorem}
	Let $H$ be a loopless multigraph, and let $G$ be a simple graph. If \[\max\left\{33\Delta(H)+172,11|V(H)|+7\right\}\le \delta(G),\]
	then $H\le_{\mathrm{i}} G$.
\end{theorem}

\section{Algorithms and complexity}
\begin{table}
	\centering
	\bgroup%
	\def\arraystretch{1.5}
	\begin{tabular}{ r  l  }
		\toprule
		\multicolumn{2}{c}{Maximum Edge-Disjoint Paths problem (\textsc{MaxEDP})} \\
		\midrule
		\textit{Input} & Two loopless multigraphs, $D$ and $G$, on the same vertex set \\
		\cmidrule(lr){2-2}
		\textit{Feasible solution} & A subgraph $D^* \subseteq D$ and its realization in $G$\\
		\cmidrule(lr){2-2}
		\textit{Objective} & Maximize $e(D^*)$ \\
		\cmidrule(lr){2-2}
		\textit{Decision version} & Given $(D,G,k)$, decide whether $\max e(D^*)\ge k$. \\
		\bottomrule
	\end{tabular}
	\egroup%

	\bigskip

	\caption{Definition of the Maximum Edge-Disjoint Paths problem (\textsc{MaxEDP})}\label{table:maxedpdef}
\end{table}

The maximum edge-disjoint paths problem (see~\Fref{table:maxedpdef}) is among the early problems shown to be \textsc{NP}-complete by Richard Karp~\cite{Karp}, although he referred to it as the ``disjoint paths problem''.
For a fixed number of paths the problem is solvable in polynomial time (see~\cite{MR2595703}). However, if the number of required paths is part of the input, then the (decision version of the) problem is \textsc{NP}-complete even for series-parallel~\cite{MR1869356} and complete graphs~\cite{MR2202473}. This has been one of the reasons that forced us to consider an extremal approach to the terminal-pairability problem.

\medskip

Surprisingly, and inadvertently, \Fref{thm:delta_n/3} gives the to date tightest approximation to the maximum edge-disjoint paths problem in complete graphs (see~\Fref{thm:3approxEDP}). Although these results are by-products of our study, we believe that this efficiency is not a coincidence, even though our approach has been an extremal one from the beginning.

\medskip

We store graphs concurrently as edge lists and adjacency lists; an edge contains pointers to its copies in both lists.
For a note on models of computation and graph representations, the reader is advised to consult~\Fref{appendix:modelsofcomputation}.

\subsection{Algorithmic versions of Theorem~\ref{thm:delta_n/3}~and~\ref{thm:gridmain}}

We will use the following results to find 2-factors.
\begin{theorem}[{\citet[Thm.~2]{MR1805711}}]\label{thm:bipmatching}
	Given a regular bipartite multigraph on $m$ edges, there is a deterministic algorithm that finds a complete matching in $O(m)$.
\end{theorem}

\begin{theorem}[{\citet[Thm.~1]{MR1805711}}]\label{thm:bipedgecolor}
	Given a regular bipartite multigraph on $m$ edges with maximum degree $\Delta$, there is a deterministic $O(m\log \Delta)$ time algorithm that finds a proper edge coloring using $\Delta$ colors.
\end{theorem}

Using randomization, there are even more efficient algorithms to find perfect matchings (see \citet{MR3504633}), however, using them would not improve the order of magnitude of the running time of the following theorem.

\begin{theorem}\label{thm:K_nalgo}
	Given a loopless multigraph $D$ on $n$ vertices with $\Delta(D)\le 2\lfloor\frac{n}{6}\rfloor-4$, there is
	%a randomized $O(\Delta(D)n\log(n))$ time (expected and with high probability), and
	a deterministic $O({\Delta(D)}^2 n)$ time algorithm which finds a realization of $D$ in $K_n$.
\end{theorem}
\mynote{can we prove a $O(m\Delta(D))$ running time?}
\begin{proof}
	\textbf{Preprocessing.} We label each edge of $D$ with a unique label. Whenever we lift a labeled edge, the two new edges inherit their ancestors label. Via these labels, we can recover the edge-disjoint paths in the solution.

	\medskip

	Let the vertex set of the demand graph be $V(D)=\{v_1,\ldots,v_n\}$. By lifting existing edges or joining non-maximal degree vertices (via an edge with a yet unused label), we can make the input graph $2\lceil\frac{\Delta(D)}{2}\rceil$-regular in $O(\Delta(D)n)$ time.

	\medskip

	\textbf{Iterative step (see \Fref{thm:delta_n/3}).}
	If $D$ is 2-regular after preprocessing, a realization can be found in $O(n)$ time.

	\medskip

	We can find an Eulerian orientation $\overrightarrow{D}$ in $O(\Delta(D)n)$ time. Construct a bipartite graph $G$ by taking two copies, $\{v'_1,\ldots,v'_n\}$ and $\{v''_1,\ldots,v''_n\}$, of the vertex set of $D$, and join $v'_i$ to $v''_j$ iff there is an edge from $v_i$ to $v_j$ in $\overrightarrow{D}$. Observe, that perfect matchings of $G$ correspond to 2-factors of $D$.

	\medskip

	Via \Fref{thm:bipmatching},
	two perfect matchings of $G$ can be found in $O(\Delta(D)n)$ time, which correspond to two edge-disjoint 2-factors, $A_1$ and $A_2$, of $D$.

	\medskip

	A lifting coloring of a 2-matching can be constructed in $O(n)$ time. A run of \Fref{lemma:main} can be computed in $O(n)$ time, as a lifting operation can be performed in $O(1)$. A set $X_1$ can be chosen in $O(n)$, and $A_2-X_1$ can be extended to a maximal (not maximum!) 2-matching in $O(n)$, as well. Similarly, $B_1$ and $B_2$ can be determined in $O(n)$ time. The edges incident to the 6 vertices removed by two iterations of \Fref{lemma:main} are saved to a separate solution graph.

	\medskip

	By lifting or adding a constant number of edges, we can make the remaining graph $H_2$ regular. Recurse on $H_2$.
	%Furthemore, an Eulerian orientation of this graph can be computed in $O(n)$ time: by properly orienting lifted edges, the lifting operation preserves Eulerian property. as two 2-factors were removed from $D$ (these steps preserve the Eulerian orientation) and there are 0 or 2 new incident edges at every remaining vertex.

	\medskip

	\textbf{Running time.} The degree of the demand graph decreases by two in every iteration, whose running time is dominated by finding 2-factors. Theoretically, it could be profitable to compute a 2-factor decomposition of $D$ during preprocessing (for example, via \Fref{thm:bipedgecolor}) and maintain this structure for subsequent iterations. However, it is unclear how an Eulerian orientation could be maintained in $o(\Delta(D)n)$ time, let alone a 2-factor decomposition.
\end{proof}

In the following we show using an argument which is similar to \citeauthor{MR2416955}'s~\cite[see][Thm.~6]{MR2416955}, that there is a polynomial time $(3+\varepsilon_n)$-approximation scheme for the \textsc{MaxEDP} problem in $K_n$. We will need the following theorem.

\begin{theorem}[{\citet[Thm.~4.1]{gabow1983efficient}}]\label{thm:gabow}
	Let $H$ be a multigraph on the vertex set $\{v_1,v_2,\ldots,v_n\}$ with $m$ edges. Let $u_i\in \mathbb{N}$. The spanning subgraph of $H$ with the maximum number of edges in which $d_H(v_i)\le u_i$ (for all $i=1,\ldots,n$) can be found in $O(mn\log n)$ time and $O(m)$ space.
\end{theorem}

% \begin{theorem}[\citet{MR0030203}]\label{thm:shannon}
% 	For any multigraph $G$, its edge coloring number is at most $\frac32\Delta(G)$. Furthermore, such a coloring can be found in $O(m(\Delta(G)+n))$ time.
% \end{theorem}

We are ready to prove our approximation result of \textsc{MaxEDP} in complete graphs.

\begin{theorem}\label{thm:3approxEDP}
	Let $D$ be a demand graph on the vertex set of $K_n$. There is an $O(mn\log n + n^3)$ time algorithm which gives a $(3+O(1/n))$-approximation solution to the \textsc{MaxEDP} problem in $K_n$.
\end{theorem}
\begin{proof}
	Let $D_\text{opt}$ be a subgraph of $D$ which is realizable in $K_n$, such that it has the maximum possible number of edges. Obviously, $\Delta(D_\text{opt})\le n-1$. Run the algorithm of \Fref{thm:gabow} on $D$ with $u_i=2\lfloor\frac{n}{6}\rfloor-4$ (for $i=1,\ldots,n$) to obtain $D^*$.

	\medskip

	According to \Fref{thm:petersen}, we can partition $E(D_\text{opt})$ into $\lceil\frac{n-1}{2}\rceil$ edge-disjoint 2-matchings. 	Order the 2-matchings in decreasing order of their cardinality, and choose the first $\lfloor\frac{n}{6}\rfloor-2$. Let the spanning subgraph of $D_\text{opt}$ formed by the union of the chosen 2-matchings be $D'$.
	Since $\Delta(D')\le 2\lfloor\frac{n}{6}\rfloor-4$, we have $e(D')\le e(D^*)$. Furthermore, because we take the largest 2-matchings,
	\[ e(D^*)\ge e(D')\ge \frac{\lfloor\frac{n}{6}\rfloor-2}{\lfloor\frac{n}{2}\rfloor}\cdot e(D^\text{opt})\ge  \left(\frac13-\frac{5+\frac13}{n-1}\right)\cdot e(D^\text{opt}). \]

	\medskip

	By \Fref{thm:K_nalgo}, we can compute a realization of $D^*$ in $K_n$ in $O(n^3)$ time.
\end{proof}

\Fref{thm:K_nalgo} and \Fref{thm:gridmain} have the following consequence.

\begin{corollary}
	Let $D$ be a loopless multigraph as a demand graph in $K_t^d$ (which has  $n=t^d$ vertices). If $\Delta(D)\le 2\lfloor\frac{t}{12}\rfloor-2$, then there is
	a deterministic $O(d\cdot n\cdot {\Delta(D)}^2)$ time algorithm which finds a realization of $D$ in $K_t^d$.
\end{corollary}
\begin{proof}
	Without loss generality, we may make $D$ regular with an even degree. We can use \Fref{thm:bipedgecolor} to find a 2-factor decomposition of $D$ in $O(t^d\cdot\Delta(D)\cdot \log \Delta(D))$.
	Furthermore, this decomposition can be inherited by the layers $L_i$, so it does not have to be recomputed when the algorithms invokes recursion on the layers.
	Let $T_t^{\Delta(D)}(d)$ be a bound on the running time of the rest of the algorithm. We have
	\begin{align*}
		T_t^{\Delta(D)}(1)&=O({\Delta(D)}^2 t) \\
		T_t^{\Delta(D)}(d)&=O(\Delta(D)\cdot t^d)+t^{d-1}\cdot T_t^{2\Delta(D)}(1)+t\cdot T_t^{\Delta(D)}(d-1)
	\end{align*}
	Solving the recursion, we get
	\[ T_t^{\Delta(D)}(d)= O\left( d\cdot t^d \cdot {\Delta(D)}^2 \right)=O(d\cdot n\cdot {\Delta(D)}^2), \]
	which clearly dominates the time spent preprocessing the graph.
\end{proof}

\subsection{Comparing our results to the state of the art}
According to a result of \citet{MR2080076}, the solution produced by a shortest-path-first or a bounded-length greedy algorithm is not better than a 3-approximation result for every input graph. \Fref{thm:3approxEDP} almost achieves this bound by producing a $(3+\varepsilon_n)$-approximation for any instance.

\medskip

The champion before \Fref{thm:3approxEDP} was the 3.75-approximation algorithm of \citet{MR2416955}. On demand graphs where $\overline{d}(D)=o(\Delta(D))$, our algorithm is up to a factor of $n$ slower than that of \citeauthor{MR2416955}.
However, if $\Delta(D)\le (1+o(1))\cdot n$, we may replace \Fref{thm:gabow} with a 2-matching decomposition to gain a $\log n$ on the running time. The solution produced by this modified algorithm is a $\left(\frac{3\Delta(D)}{n}+o(1)\right)$-approximation of the optimum.

%\medskip
%
%The algorithm of \Fref{prop:alpha_algo} improves on the previous $m<n$ bound by a factor of two (\cite[Prosition~3]{MR2416955}), though one has to pay the (low) price of a superlinear running time.

\section{Further base graphs}

I am hopeful that our new results demonstrated in this part will increase interest in the terminal-pairability problem. Continuing this line of research, complete bipartite base graphs have been studied by \citeauthor{CEGyM1} in two settings: the case when the demand graph is bipartite with respect to the classes of the base graph has been studied in~\cite{CEGyM1}, and when no structural restrictions (other than maximum degree) are made on the demand graph is explored in~\cite{CEGyM2}.

\medskip

A risky and undertaking research direction is to study the terminal-pairability problem very generally, and to try to discover sufficient conditions for a demand graph to be realizable in a base graph, without a priori specifying too much information about any of them.

\medskip

The degree conditions in our theorems are special cut conditions. Can we prove stronger theorems if we consider more than only single vertex cuts? In other words, does interpreting the problem as an integer multi-commodity flow task help?
\begin{problem}
	Suppose $D$ and $G$ are loopless multigraphs on the vertex set $V=\{1,\ldots, n\}$. What is the minimum value of $f(n)$ so that
	\[ f(n)\le \min_{\emptyset\neq A\subset V} \frac{e_G(A,V-X)}{e_D(A,V-X)}\implies\text{$D$ is realizable in $G$?} \]
\end{problem}

\medskip

By choosing $G$ as a regular expander graph, one can prove that $f(n)\ge \Omega(\log n)$. If $G$ is required to be simple, a similar construction implies $f(n)\ge \Omega(\log n/\log\log n)$.

% ********************************** Back Matter *******************************
% Backmatter should be commented out, if you are using appendices after References
%\backmatter

% ********************************** Bibliography ******************************
\begin{spacing}{0.9}

	% To use the conventional natbib style referencing
	% Bibliography style previews: http://nodonn.tipido.net/bibstyle.php
	% Reference styles: http://sites.stat.psu.edu/~surajit/present/bib.htm

	%\bibliographystyle{apalike}
	%\bibliographystyle{unsrt} % Use for unsorted references
	%\bibliographystyle{plainnat} % use this to have URLs listed in References

  %\cleardoublepage
	%\bibliography{references} % Path to your References.bib file

	% If you would like to use BibLaTeX for your references, pass `custombib' as
	% an option in the document class. The location of 'reference.bib' should be
	% specified in the preamble.tex file in the custombib section.
	% Comment out the lines related to natbib above and uncomment the following line.

	\printbibliography[heading=bibintoc, title={References}]

\end{spacing}

% ********************************** Appendices ********************************

\begin{appendices} % Using appendices environment for more functionality
\fancyhead[LO]{\nouppercase Appendix \thechapter}

\chapter{A note on models of computation and representations of graphs}\label{appendix:modelsofcomputation}

When multigraphs are part of the input, one should exercise great care when determining running times. The problem has its roots in the details of the graph representation used when describing the input and output graphs.

\medskip

Let us choose the very common edge list representation. Consider the following problem: given an input multigraph with $m\le \binom{n}{2}$ edges on the vertex set $\{1,\ldots, n\}$, output any simple graph with $m$ edges on the vertex set of the input graph. Clearly, the output needs $\Omega(m\log n)$ space. However, if the edges of the multigraph are incident to only a subset of the vertex set, say, $\{1,2,3,4\}$, then the input may need only $O(m+\log n)$ bits of space. Thus, an algorithm solving this problem cannot have a running time which is at most a linear function of the size of the input.

\medskip

The running time of breadth first and depth first search algorithms is usually regarded as $O(m+n)$, but a factor of $\log n$ is clearly missing. However, these algorithms are correctly regarded as having a linear running time as a function of the size of the input. It is easy to see that describing a simple graph with $\Omega(n)$ edges requires $\Omega(n\log n)$ space, and thus the extra $\log n$ factor is usually not a problem.

\medskip

One could argue that the previous problem is only a question of whether one chooses the unit cost or the logarithmic cost RAM machine model. However, we may exacerbate the problem (of describing the running time of an algorithm in terms of the size of its input) further by describing a multiedge by the vertices it joins and its multiplicity. Then the size of the input may be as low as $O(\log m+\log n)$.

\medskip

However, the logarithmic cost RAM machine has surprising limitations. \citet{MR963171} showed, that storing $n$ arbitrary bits takes $\Omega(n\log^* n)$ time in this model.

\medskip

For these reasons,
%we will assume that an $m$ edge multigraph is described on $\Omega(m)$ space. Furthermore,
our choice for the model of computation is the unit cost RAM machine for both \Fref{chap:artcomplexity} and \Fref{chap:termcomplexity}. Alternatively, one may multiply the running time of our algorithms in \Fref{chap:termcomplexity} by a factor of $O(\log n)$ (where $n$ is the number of vertices of the output graph) to get the logarithmic cost running times of our algorithms.
However, the $O(n)$ algorithms outlined in \Fref{chap:artcomplexity} remain linear even in the logarithmic cost model.

\chapter{Algorithms on orthogonal polygons}\label{appendix:algoPartI}

\makeatletter
\providecommand*{\toclevel@algorithm}{1}%
\makeatother

\begin{algorithm}
  \caption{Finding the horizontal cuts of an orthogonal polygon, part I}\label{algo:horizontalcuts1}
  \begin{algorithmic}[1] % line numbering start
      \Require $P$ orthogonal polygon
      \Ensure $\textit{pair}[v]$ will contain the vertical side of $P$ which the other end of the horizontal cut starting at the reflex vertex $v$ intersects
      \Statex
      \Function{Find horizontal cuts}{$P$}
        \State $n\gets$ number of vertices of $P$
        \State $T\gets \Call{Triangulate}{P}$\Comment{List of triangles}
        \State Initialize $L[v]=\emptyset$, doubly linked lists for each vertex $v$ of $P$
        \ForAll{$t\in T$}
          \For{$i=1,2,3$}
            \State $L[t.v_i]\gets$ append a link to $t$
            \State $t.l_i\gets$ a link to $t$'s location in $L[t.v_i]$
          \EndFor
        \EndFor
        \Statex
        \ForAll{side or diagonal $s$ of $P$}
          \State Initialize $S[s]$ to empty doubly linked list
          \If{$s$ is a vertical side of $P$}
            \State $S[s]\gets$ insert $s$
          \EndIf
        \EndFor
        \algstore{Rcutbreak}
    \end{algorithmic}
\end{algorithm}

\begin{algorithm}
  \caption{Finding the horizontal cuts of an orthogonal polygon, part II}\label{algo:horizontalcuts2}
  \begin{algorithmic}
    \algrestore{Rcutbreak}
    \State $Q\gets $ a queue of vertices with $|L[v]|=1$
    \While{$Q\neq\emptyset$}\label{step:beginloop}
      \State $u\gets $ pop the first element of $Q$
      \State $t\gets L[u]$
      \State $v,w\gets$ other two vertices of $t$ so that $w.y\le v.y$
      \State $s_u,s_v,s_w\gets $ sides of $t$ opposite the vertex in the index
      \State Delete $t$ from $T$ and the 3 links to it in $L$, update $Q$
    \If{$w.y\le u.y\le v.y$}
      \State $S[s_u]\gets S[s_w]$ append $S[s_v]$\label{step:project1}
    \ElsIf{$s_u$ is horizontal}
      \State Process $S[s_v]$ and $S[s_w]$ in $y$-order and update $\textit{pair}[]$ for the vertices in them\label{step:process1}
    \ElsIf{$v.y<u.y$}
      \State $S[s_u]\gets$ segment of $S[s_v]$ visible along the $x$-axis from $s_u$\label{step:project2}
      \State $R\gets$ segment of $S[s_v]$ \textbf{not} visible along the $x$-axis from $s_u$
      \State Process $R$ and $S[s_w]$ in $y$-order and update $\textit{pair}[]$ for the vertices in them\label{step:process2}
    \ElsIf{$w.y>u.y$}
      \State $S[s_u]\gets$ segment of $S[s_w]$ visible along the $x$-axis from $s_u$\label{step:project3}
      \State $R\gets$ segment of $S[s_w]$ \textbf{not} visible along the $x$-axis from $s_u$
      \State Process $R$ and $S[s_v]$ in $y$-order and update $\textit{pair}[]$ for the vertices in them\label{step:process3}
    \EndIf
    \Statex
    \If{$s_u$ is a vertical side of $P$}\label{step:sideofP}
      \ForAll{reflex vertex $z$ in $S[s_u]$}
        \If{the horizontal cut of $z$ starts towards $s_u$ }
          \State $\textit{pair}[z]\gets s_u$
        \EndIf
      \EndFor
      \If{$v$ is a reflex vertex}
          \State $\textit{pair}[v]\gets $ the first or last element of $S[s_u]$ that contains a point with the same $y$-coordinate as $v$
      \EndIf
      \If{$w$ is a reflex vertex}
          \State $\textit{pair}[w]\gets $ the first or last element of $S[s_u]$ that contains a point with the same $y$-coordinate as $w$
      \EndIf
    \EndIf
  \EndWhile
  \State\Return the list $\textit{pair}[]$
\EndFunction
  \end{algorithmic}
\end{algorithm}

\begin{algorithm}
  \caption{Constructing the horizontal $R$-tree of an orthogonal polygon}\label{algo:Rtree}
  \begin{algorithmic}[1]
    \Function{horizontal $R$-tree}{$P$}
      \State $\textit{opposite\_side}{[]}\gets\Call{Find horizontal cuts}{P}$\Comment{see~\Fref{algo:horizontalcuts1}}
      \ForAll{reflex vertex $v$ of $P$}
      	\State $w\gets$ new vertex at the height of $v.y$ in $\textit{opposite\_side}[v]$
      	\State if not present, insert $w$ into $P$
      	\State $\textit{cut\_pair}[v]\gets w$
      	\State $\textit{cut\_pair}[w]\gets v$
      \EndFor
      \Statex
      \State $G\gets (\{P\},\emptyset)$\Comment{$G$ is a graph with a single node $P$}
      \State Initialize $\textit{cut\_to\_edge}[]$ empty
      \State $\textit{CurrentNode}\gets P$
      \ForAll{vertex $v$ of $P$ in clockwise order}
		 \If{$v$ is a reflex or a new vertex}
		 	\If{$\textit{cut\_to\_edge}[v]=\emptyset$}
		 		\State Split $\textit{CurrentNode}$ along $\{v,\textit{cut\_pair}[v]\}$ in $G$\label{step:splitpolygon}
		 		\State $e\gets \{N_1,N_2\}$ the two new pieces whose union is $\textit{CurrentNode}$
		 		\State $G\gets G+e$
		 		\State $\textit{cut\_to\_edge}[v]\gets e$
        \State $\textit{cut\_to\_edge}[\textit{cut\_pair}[v]]\gets e$
		 		\State $\textit{CurrentNode}\gets $ the piece containing $\textit{clockwise\_next}_P(v)$
		 	\Else
				\State $\textit{CurrentNode}\gets $ pair of $\textit{CurrentNode}$ in $\textit{cut\_to\_edge}[v]$
		 	\EndIf
		 \EndIf
      \EndFor
      \State\Return $G$
    \EndFunction
  \end{algorithmic}
\end{algorithm}

\begin{algorithm}
  \caption{Finding a minimum cardinality horizontal $r$-guard system}\label{algo:mhsc}
  \begin{algorithmic}[1]
    \Function{Solve \textsc{MHSC}}{P}
      \State $T_H\gets \Call{horizontal $R$-tree}{P}$\Comment{\Fref{algo:Rtree}}
      \State $T_V\gets \Call{vertical $R$-tree}{P}$
      \ForAll{vertical slice $t\in T_V$}
        \State $a,b\gets$ vertical sides of $P$ bounding $t$
        \State $h_a\gets$ horizontal slice in $V(T_H)$ containing $a$
        \State $h_b\gets$ horizontal slice in $V(T_H)$ containing $b$
        \State $\textit{N}[t]\gets \{h_a, h_b\}$
      \EndFor
      \Statex
      \State $r\gets $ arbitrary node of $T_H$ to serve as root
      \State $\textit{dist}[]\gets \Call{Breadth First Search}{T_H, r}$\Comment{distance from $r$}
      \Statex
      \State $\textit{LCA}[]\gets \Call{Lowest Common Ancestors}{T_H, r, \textit{N}[]}$\Comment{Algorithm of~\cite{MR801823}}
      \State \Comment{$\textit{LCA}[t]$ contains the lowest common ancestors of the elements of $\textit{N}[t]$}
      \Statex
      \State $S\gets\emptyset$
      \State Set every node of $T_H$ unmarked
      \ForAll{$t\in V(T_V)$ so that $\textit{dist}[\textit{LCA}[t]]$ is not increasing}\Comment{reverse BFS-order}
        \If{both elements of $\textit{N}[t]$ are unmarked}
          \State $S\gets S\cup \{\textit{LCA}[t]\}$
          \State \Call{Set Mark}{\textit{LCA}[t]}
        \EndIf
      \EndFor
      \State\Return $S$
    \EndFunction

    \Statex

    \Function{Set Mark}{$u$}
      \ForAll{neighbor $w$ of $u$ in $T_H$}
        \If{$\textit{dist}[w]>\textit{dist}[u]$ and $w$ is unmarked}
          \State \Call{Set Mark}{w}
        \EndIf
      \EndFor
    \EndFunction
  \end{algorithmic}
\end{algorithm}

\end{appendices}

% *************************************** Index ********************************
%\printthesisindex % If index is present

\end{document}